\documentclass[11pt]{article}
\usepackage{amsmath,amsfonts,amsthm}
\usepackage{color}
\usepackage{hyperref}
\hypersetup{colorlinks=true,
            linkcolor={blue},
            citecolor={falured}}
\theoremstyle{definition}
\newtheorem{assumption}{Assumption}
\usepackage{amsmath,amssymb,fullpage}
\usepackage{enumerate}
\usepackage{graphicx}
\usepackage{makeidx}
\usepackage[utf8]{inputenc}
\usepackage{wrapfig}
\usepackage{mathrsfs}
\usepackage{amssymb}
\usepackage{latexsym}
\usepackage{amsbsy}
\usepackage{color}
\usepackage{bbm}
\usepackage{mathrsfs}
\usepackage{todonotes}
\usepackage{url}

\usepackage{wasysym}
\def\email#1{\it #1\par}

\providecommand{\otherindexspace}[1]{}

\usepackage{amsthm}
\newtheorem{theorem}{Theorem}[section]
\newtheorem{lemma}[theorem]{Lemma}
\newtheorem{proposition}[theorem]{Proposition}
\newtheorem{remark}[theorem]{Remark}
\newtheorem{definition}[theorem]{Definition}

\newtheorem{corollary}[theorem]{Corollary}

\definecolor{falured}{rgb}{0.5, 0.09, 0.09}

\usepackage{comment}

\usepackage{dsfont}

\numberwithin{equation}{section}

\title{Control and optimal stopping Mean Field Games: a linear programming approach}
\author{ Roxana Dumitrescu \thanks{Department of Mathematics, King's College London, Strand, London, WC2R 2LS, United Kingdom, Email: \email roxana.dumitrescu@kcl.ac.uk} \and Marcos Leutscher  \thanks{CREST, ENSAE, Institut Polytechnique de Paris, 5 avenue Henry Le Chatelier, 91120 Palaiseau, France, Email: \email marcos.leutscherdelasnieves@ensae.fr} \and Peter Tankov \thanks{CREST, ENSAE, Institut Polytechnique de Paris, 5 avenue Henry Le Chatelier, 91120 Palaiseau, France, Email: \email peter.tankov@ensae.fr}}

\date{}

\begin{document}
\maketitle
\begin{abstract}
{We develop the \textit{linear programming approach} to mean-field games in a general setting. This relaxed control approach allows to prove existence results under weak assumptions, and lends itself well to numerical implementation. We consider mean-field game problems where the \textit{representative agent} chooses both the optimal control and the optimal time to exit the game, where the instantaneous reward function and the coefficients of the state process may depend on the distribution of the other agents. Furthermore, we establish the equivalence between mean-field games equilibria obtained by the linear programming approach and the ones obtained via the controlled/stopped martingale approach, another relaxation method used in a few previous papers in the case when there is only control.}

\end{abstract}


\textbf{AMS: 91A55, 91A13, 60G40}

\section{Introduction}

Mean-Field Games (MFGs) represent the limit version of stochastic differential games with a large number of agents, symmetric interactions and negligible individual influence of each player on the others. This theory has been introduced independently by Lasry and Lions \cite{lasry2006a,lasry2006b,lasry2007} and Huang, Malhamé and Caines \cite{huang2006}. Since the $N$-player game is rarely tractable, MFGs provide a useful tool for approximating the $N$-player Nash equilibria.


{\color{black}In this paper, we develop the \textit{linear programming approach} to mean-field games in a general setting. The linear programming approach is a control relaxation technique, which allows to prove existence results under weak assumptions and lends itself well to numerical implementation. It is well known in the field of stochastic control, but has been introduced to MFGs only recently in \cite{bdt2020}. That paper considers MFGs of optimal stopping (where each agent only decides when to stop) and under restrictive assumptions, in particular the coefficients of the state process of the representative agent do not depend on the distribution of the other agents. The goal of this paper is to present the linear programming approach in a much wider context of control and stopping MFG, with coefficients depending on the measure and under weaker assumptions than in \cite{bdt2020}, as well as to establish the equivalence of the linear programming approach with the other control relaxation approaches. }

Our aim is to study MFGs in a general setting, including optimal stopping, continuous control and absorption. To explain the concept, assume that we have a `large' number $N\in \mathbb N^*$ of players. Each agent $k\in\{1,\dots, N\}$ has a private state process $X^{k, N}=(X_t^{k, N})_{t\in[0, T]}$, whose dynamics are given by the stochastic differential equation (SDE)
\begin{equation*}
dX_{t}^{k, N} = b\left(t, X_{t}^{k, N}, m_t^N, \alpha^k_t\right) dt + \sigma\left(t, X_{t}^{k, N}, m_t^N, \alpha^k_t\right) dW_{t}^{k},
\end{equation*}
where $W^1, \dots, W^N$ are independent Brownian motions, $\alpha^k=(\alpha^k_t)_{t\in[0, T]}$ is the control process taking values in a closed subset $A\subset \mathbb R$, chosen by agent $k$ and $m_t^N$ is the empirical occupation measure of the players still present in the game and their controls: 
\begin{equation*}
m_t^N(dx, da)=\frac{1}{N}\sum_{k=1}^N\delta_{\left(X_{t}^{k, N}, \alpha^k_t\right)}(dx, da)\mathds 1_{t\leq \tau^k \wedge \tau_{\mathcal{O}}^{k}}.
\end{equation*}
Here $\tau^k$ is the stopping time, valued in $[0, T]$, chosen by player $k$, and 
$$\tau_{\mathcal{O}}^{k}=\inf\{t\geq 0: X_t^{k, N}\notin \mathcal{O}\},$$
denotes the first exit time of agent $k$ from an open  subset $\mathcal{O} \subset \mathbb{R}$, with the convention $\inf \emptyset=\infty$.  

Each agent $k$ seeks to choose an optimal stopping time $\tau^k$ and an optimal control $\alpha^k$ to maximize the reward functional defined as follows:
\begin{equation*}
\mathbb E \left[ \int_0^{\tau^k\wedge \tau_{\mathcal{O}}^{k}} f\left(t, X_{t}^{k, N}, m_t^N, \alpha^k_t\right) dt + g\left(\tau^k\wedge \tau_{\mathcal{O}}^{k}, X_{\tau^k\wedge \tau_{\mathcal{O}}^{k}}^{k, N}, \mu^N\right) \right],
\end{equation*}
where $\mu^N$ is the empirical joint distribution of the stopping time and the state process at the time of stopping:
$$\mu^N(dt, dx)=\frac{1}{N}\sum_{k=1}^N\delta_{\left(\tau^k\wedge \tau_{\mathcal{O}}^{k}, X_{\tau^k\wedge \tau_{\mathcal{O}}^{k}}^{k, N}\right)}(dt, dx).$$

The objective functionals and the dynamics of the agents are coupled through the empirical measures $(m_t^N)_{t\in[0, T]}$ and $\mu^N$, so that it is natural to look for  a Nash equilibrium.
When the number of players $N$ goes to infinity, we expect by a ``propagation of chaos'' type result that the empirical occupation measures converge to a \textit{deterministic flow of subprobability measures} $(m_t)$, while the empirical joint distributions of the stopping time/state process when each player exits the game (via discretionary stopping or absorption), converges to a \textit{deterministic limiting probability measure} $\mu$. In our setting the MFG problem therefore reads as follows: 
\begin{enumerate}[(i)]
\item Fix $(\mu, (m_t)_{t\in[0,T]})$ and find the solution to the mixed control / optimal stopping problem
\begin{equation}\label{opti_intro}
\begin{aligned}
\sup_{\tau,\alpha } & \quad \mathbb{E}\left[\int_{0}^{\tau\wedge \tau_\mathcal{O}^{\alpha, m}} f\left(t, X^{\alpha, m}_t, m_t, \alpha_t\right) dt + g\left(\tau\wedge \tau_{\mathcal{O}}^{\alpha, m}, X^{\alpha,m}_{\tau\wedge \tau_{\mathcal{O}}^{\alpha, m}}, \mu\right)\right],\\
\text{s.t. } & \quad dX_{t}^{\alpha, m}=b\left(t, X_{t}^{\alpha, m}, m_t, \alpha_t \right) dt+\sigma\left(t, X_{t}^{\alpha, m}, m_t, \alpha_t \right) dW_{t},
\end{aligned}
\end{equation}
where $\tau_\mathcal{O}^{\alpha, m}=\inf\{t\geq 0:X_t^{\alpha, m}\notin \mathcal{O}\}$.
\item Given the solution $(\tau^{\mu, m}, \alpha^{\mu, m})$ of the problem \eqref{opti_intro} for the agent facing a mean-field $(\mu, (m_t)_{t\in[0,T]})$, find $(\mu, (m_t)_{t\in[0,T]})$  such that
\begin{equation*}
m_{t}(B)=\mathbb{P}\left[(X_{t}^{\alpha^{\mu,m}, m}, \alpha^{\mu,m}_t) \in B, t\leq \tau^{\mu, m}\wedge \tau_{\mathcal{O}}^{\alpha^{\mu, m}, m}\right], \quad B \in \mathcal{B}(\bar{\mathcal{O}} \times A),\quad t \in[0, T],
\end{equation*}
and 
$$\mu=\mathcal{L}\left(\tau^{\mu, m}\wedge \tau_{\mathcal{O}}^{\alpha^{\mu, m}, m}, X^{\alpha^{\mu,m}, m}_{\tau^{\mu, m}\wedge \tau_{\mathcal{O}}^{\alpha^{\mu, m}, m}}\right).$$
\end{enumerate}

In the literature on MFGs, there are two main approaches to prove existence of an MFG Nash equilibrium. The first approach, developed by Lasry and Lions \cite{lasry2007}, is an analytic one and consists in finding a Nash equilibrium by solving a coupled system of nonlinear partial differential equations: a Hamilton-Jacobi-Bellman equation (backward in time) satisfied by the value function of the representative agent for a given distribution  and a Fokker-Planck-Kolmogorov equation (forward in time) describing the evolution of the density when the optimal control is used. The second approach, introduced by Carmona and Delarue \cite{carmona2012,carmona2018}, is based on the stochastic maximum principle which reduces the problem to a system of coupled forward-backward stochastic differential equations of McKean-Vlasov type.

In the standard stochastic control theory, the \textit{controlled martingale problem approach} (see e.g. \cite{elkaroui1987}, \cite{haussmann1990} and \cite{elkaroui2015}\footnote{We thank Xiaolu Tan for fruitful discussions on the paper \cite{elkaroui2015}.}) is a powerful tool allowing to simplify the existence proofs, by compactification of the stochastic control problem. In the original MFG framework (regular control, without optimal stopping), the controlled martingale problem approach was first used in \cite{lacker2015} to show the existence of a mean field game equilibrium under general assumptions. Further developments have been made in the case of mean field games with branching (\cite{claisse2019}) or mean field games with singular controls (\cite{fu2017}). Another relaxation technique used in the classical stochastic control theory is based on the \textit{linear programming formulation} (see e.g. \cite{stockbridge1998,stockbridge2002,stockbridge2017}). In the context of mean-field games, this method has only been used in the case of optimal stopping in \cite{bdt2020}. 

Mean field games of optimal stopping are a very recent trend in the MFG literature. More generally, only a few papers study mean field games with possible exit of the players leading to a decrease of the total mass of the players still in the game. We refer here to the MFGs with absorption (see  e.g. \cite{campi2018}) and the MFGs of optimal stopping, introduced in the case of bank run models in \cite{nutz2018,carmona2017}, studied using an analytic approach in \cite{bertucci2017}, and in a more general framework in \cite{bdt2020}.


In this paper, we extend the linear programming approach initiated in \cite{bdt2020} to a more general setting including mixed optimal stopping/control, allowing for measure depending coefficients, {\color{black}and involving weaker assumptions.}
Furthermore, we clarify the relationship  between  \textit{linear programming MFG equilibria} and \textit{\color{black}MFG equilibria in the controlled/stopped martingale problem approach} (also known as the \textit{weak formulation}), and {\color{black}state precise conditions of equivalence of the two approaches.} In the pure control case, {\color{black}this equivalence enables us to recover directly the result of existence of a Markovian equilibrium shown in \cite{lacker2015} by using the Markovian projection technique.}  In addition, our method allows us to establish the existence of mixed solutions in the sense of \cite{bertucci2017}, under a more general framework (in particular, with coefficients depending on both the control and the measure which was not the case in \cite{bertucci2017}).

The existence theorems of MFG equilibria obtained through {\color{black}the controlled martingale problem approach} are in general rather abstract and provide little insight into the computation of MFG solutions. However, the linear programming approach we develop leads to a tractable method of computing the MFG equilibria, which has been instrumental in several concrete applications (see e.g. \cite{adt2020,bdt2020b}).

The paper is organized as follows. In Section 2,  we first study the single-agent problem under the linear programming formulation: we show the existence of a solution and  {\color{black}prove its equivalence with the various weak formulations, as well as with PDEs.} In Section 3, we solve the MFG problem and relate the notions of linear programming equilibria, weak equilibria and mixed solutions. In the Appendix we give some technical results and in particular we make the connection between the linear programming and the weak formulations, extending some of the existing results in the literature to less regular coefficients (see e.g. \cite{stockbridge1998,stockbridge2002,stockbridge2012,stockbridge2017}).

\paragraph{Notation.}
For any topological space $E$ we denote by $\mathcal{B}(E)$ the Borel $\sigma$-algebra, by $\mathcal{P}(E)$ the set of probability measures on $(E, \mathcal{B}(E))$ and by $\mathcal{M}(E)$ the set of positive finite measures on $(E, \mathcal{B}(E))$. We endow $\mathcal{P}(E)$ and $\mathcal{M}(E)$ with the topology of weak convergence and the associated Borel $\sigma$-algebra. We denote by $C(E)$ the set of continuous functions from $E$ to $\mathbb R$ and by $C_b(E)$ the space of continuous and bounded functions from $E$ to $\mathbb R$ which is endowed with the supremum norm
$$\|\varphi\|_\infty=\sup_{x\in E}|\varphi(x)|.$$
Let $T>0$ be a terminal time horizon, $\mathcal{O}$ an open subset of $\mathbb R$ with closure $\bar{\mathcal{O}}$ and $A$ be a compact subset of $\mathbb{R}$. We denote by $C_b^{1, 2}([0, T]\times \bar{\mathcal{O}})$ the set of functions $u\in C_b([0, T]\times \bar{\mathcal{O}})$ such that $\partial_t u, \partial_x u, \partial_{xx}u \in C_b([0, T]\times \bar{\mathcal{O}})$. We denote by $\mathbb R_+$ the set $[0, +\infty[$. For a given process $(Y_t)$ and a Borel subset $B$ of $\mathbb R$, we define the random time
$$\tau_{B}^Y(\omega):=\inf\{t\geq 0: Y_t(\omega)\notin B\},$$
with the convention $\inf\emptyset =\infty$.

Let $V_0$ be the space of flows of measures on $\bar{\mathcal{O}} \times A$, $\left(m_{t}\right)_{t\in [0, T]}$, such that: for every $t \in[0, T]$, $m_{t}$ is a Borel finite signed measure on $\bar{\mathcal{O}} \times A$, for every $B \in \mathcal{B}(\bar{\mathcal{O}} \times A)$, the mapping $t \mapsto m_{t}(B)$ is measurable, and $\int_{0}^{T} |m_{t}|(\bar{\mathcal{O}} \times A) dt < \infty$, where $|m_{t}|$ is the variation of $m_{t}$.

We define $V_1$ as the quotient space given by $V_0$ and the almost everywhere equivalence relation on $[0, T]$, that is, if, $dt$-almost everywhere on $[0,T]$, the measures $m_{t}^1$ and $m_{t}^2$ coincide, the measure flows $(m^1_t)_{t\in [0, T]}$ and $(m^2_t)_{t\in [0, T]}$ are considered equivalent. $V_1$ endowed with the usual sum and scalar multiplication is a vector space, where the zero vector is given by the family of null measures $(\mathbf 0)_{t\in [0, T]}$. To each $(m_t)_{t\in [0, T]}\in V_1$ we associate a Borel finite signed measure on $[0, T]\times\bar{\mathcal{O}}\times A$ defined by $m_t(dx, da)dt$ and we endow $V_1$ with the topology of weak convergence of the associated measures. We denote by $V$ the set of measure flows $(m_t)_{t\in [0, T]}\in V_1$ such that $dt$-a.e. $m_t$ is a positive measure. We note that $V_1$ is a Hausdorff locally convex topological vector space and $V$ is metrizable (we refer to Appendix \ref{sec A} for more details).

Let $W=(W_t)_{t\in [0, T]}$ be a standard Brownian motion on a complete probability space $(\Omega,\mathcal{F}, \mathbb P)$. We denote by $\mathbb{F}^W$ the filtration given by $\mathcal{F}_{t}=\sigma\left(W_{s}, 0 \leq s \leq t\right) \vee \mathcal{N}$, where $\mathcal{N}$ denotes the $\mathbb P$-null sets of $\mathcal F$. Denote by $\mathcal{T}$ the set of stopping times with respect to this filtration with values in $[0, T]$. Let $\mathbb{A}$ be the set of $\mathbb{F}^W$-progressively measurable control processes taking values in $A$.

In the paper we adopt the following terminology: controls of the type $\alpha_t$ with values in $A$ are called \textit{strict controls}; controls of the form  $\alpha_t=\alpha(t, X_t)$, with $\alpha$ a given  measurable function  are called \textit{Markovian strict controls}; controls of the form $\nu_t$ (respectively $\nu_{t, X_t}$ for some kernel $(\nu_{t,x})$) with values in $\mathcal{P}(A)$ are called \textit{relaxed controls} (respectively \textit{Markovian relaxed controls}). Relaxed controls are related to mixed strategies in game theory and consist in randomizing the action, which allows to embed the controls in a well behaved space. More precisely, instead of choosing an action valued in $A$, the agent chooses an action in $\mathcal{P}(A)$.

\section{Single agent problem}

In this section, we study the linear programming formulation of the mixed optimal stopping/stochastic control problem in the case when there is no interaction. In the following section, these results will be used in the MFG setting. We adopt here the following definitions and assumptions.

\begin{definition}
We denote by $\mathcal{S}$ the set of bounded measurable functions $h:[0, T]\times \bar{\mathcal{O}}\times A\rightarrow \mathbb R$ such that $h(t, \cdot)\in C(\bar{\mathcal{O}}\times A)$ for each $t\in [0, T]$.
\end{definition}

Throughout this section, unless specified otherwise, we will impose the following assumption.

\begin{assumption}\label{as_single_agent}\leavevmode
\begin{enumerate}[(1)]
\item The functions $b(t,x,a):[0, T]\times \mathbb R\times A\rightarrow \mathbb R$ and $\sigma(t,x,a):[0, T]\times \mathbb R\times A\rightarrow \mathbb R_+$ are measurable, bounded and Lipschitz in $x$ uniformly on $(t, a)$.

\item The functions $b|_{[0, T]\times \bar{\mathcal{O}}\times A}$ and $\sigma|_{[0, T]\times \bar{\mathcal{O}}\times A}$ are in $\mathcal{S}$. 

\item The function $f:[0, T]\times \bar{\mathcal{O}}\times A\rightarrow \mathbb R$ is measurable, bounded and for each $t\in [0, T]$, $f(t, \cdot)$ is upper semicontinuous and the function  $g:[0, T]\times \bar{\mathcal{O}}\rightarrow \mathbb R$ is upper semicontinuous and bounded from above.
    
\item {\color{black}$m_0^*\in \mathcal P(\mathcal O)$ satisfies} $\int_\mathcal{O} \ln(1+|x|)m_0^*(dx)<\infty$.
\end{enumerate}
\end{assumption}

Consider the classical mixed stochastic control/optimal stopping problem
\begin{equation}\label{optistop}
\begin{aligned}
\max_{\tau \in \mathcal{T},  \alpha \in \mathbb{A}} & \mathbb{E}\left[\int_{0}^{\tau \wedge \tau_{\mathcal{O}}^{X^\alpha}} f\left(t, X^\alpha_{t}, \alpha_t \right) dt  + g\left(\tau \wedge \tau_{\mathcal{O}}^{X^\alpha},  X^\alpha_{\tau \wedge \tau_{\mathcal{O}}^{X^\alpha}}\right)\right],\\
\text{s.t. } \quad & dX^\alpha_{t}=b\left(t, X^\alpha_{t},\alpha_t \right) dt+\sigma\left(t, X^\alpha_{t}, \alpha_t\right) dW_{t},\\
& X_0^\alpha \sim m_0^*,
\end{aligned}
\end{equation}
which will be called the \textit{strong problem} for the single agent.\\

We shall now provide the \textit{linear programming} formulation which consists in introducing the occupation measures and the forward equation satisfied by them.

\begin{definition}[\textit{Linear Programming (LP) formulation}]
Let $\mathcal{R}$ be the set of pairs $(\mu, m)\in \mathcal{P}([0, T]\times \bar{\mathcal{O}})\times V$, such that for all $u\in C_b^{1, 2}([0, T]\times \bar{\mathcal{O}})$,
\begin{equation}\label{constraint_fp}
\int_{[0, T]\times \bar{\mathcal{O}}} u(t, x)\mu(dt, dx)= \int_\mathcal{O} u(0, x)m_0^*(dx) + \int_0^T \int_{\bar{\mathcal{O}} \times A} \left(\frac{\partial u}{\partial t} +\mathcal L u\right) (t, x, a)m_t(dx,da)dt,
\end{equation}
where
$$\mathcal L u(t, x, a):= b(t, x, a)\frac{\partial u}{\partial x}(t, x) + \frac{\sigma^2}{2}(t, x, a)\frac{\partial^2 u}{\partial x^2}(t, x).$$
Define now the map $\Gamma: \mathcal{R}\rightarrow \mathbb R \cup \{-\infty\}$ as follows:
$$\Gamma (\mu, m)= \int_0^T \int_{\bar{\mathcal{O}}\times A} f(t, x,a) m_t(dx,da)dt + \int_{[0, T]\times \bar{\mathcal{O}}} g(t, x) \mu(dt, dx).$$
The \textit{linear programming optimization problem} takes the form
\begin{equation}\label{optistoprelax}
\max_{(\mu, m)\in\mathcal{R}} \; \Gamma (\mu, m).
\end{equation}
The value for the LP formulation is defined by 
\begin{equation}\label{value lp}
V^{LP} := \sup_{(\mu, m)\in \mathcal{R}} \Gamma (\mu, m).
\end{equation}
\end{definition}

\begin{remark}
The set $\mathcal{R}$ is nonempty. In fact, if we define
$$\mu(B\times C):=\delta_0(B)m_0^*(C\cap \mathcal{O}), \quad B\in \mathcal{B}([0, T]), \quad C\in \mathcal{B}(\bar{\mathcal{O}}),$$
then $(\mu, (\mathbf 0)_t)\in \mathcal{R}$, where $\mathbf 0$ denotes the null measure on $\bar{\mathcal{O}}\times A$.
\end{remark}

\begin{remark}
By the disintegration theorem, {\color{black}for each $(m_t)_{t\in[0,T]} \in V$,} there exists a mapping $\nu_{t,x}:[0,T]\times \bar{\mathcal{O}} \to \mathcal P(A)$ such that for each $B\in \mathcal B(A)$, the function $(t, x)\mapsto \nu_{t, x}(B)$ is $\mathcal{B}([0, T]\times \bar{\mathcal{O}})$-measurable, and 
$$m_t(dx, da)dt=\nu_{t, x}(da)m_t(dx, A)dt,$$
where $m_t(dx, A):=\int_A m_t(dx, da)$. 
\end{remark}

\begin{remark} Throughout the paper, solutions of the LP problem taking the form $m_t(dx, da)=\delta_{\alpha(t, x)}(da)m_t(dx, A)$ for some measurable function $\alpha:[0, T]\times \bar{\mathcal{O}}\rightarrow A$ are called \textit{strict control LP solutions}.
\end{remark}

\subsection{Existence of a solution for the \textit{linear programming} problem}\label{sec 2.1}

Let us first study some preliminary properties of the set $\mathcal{R}$. \paragraph{Preliminary properties of the set of constraints $\mathcal{R}$.} 
We start by showing the following admissibility result.
\begin{proposition}[\textit{Admissibility of the occupation measures}]\label{SDE to measures}
Let $(\Omega, \mathcal{F}, \mathbb F, \mathbb P)$ be a filtered probability space, $\tau$ an $\mathbb F$-stopping time such that $\tau\leq T$ $\mathbb P$-a.s., $\nu$ an $\mathbb{F}$-progressively measurable process with values in $\mathcal{P}(A)$, $M$ a continuous $\mathbb F$-martingale measure such that $M^\tau$ has intensity $\nu_t(da)\mathds{1}_{t\leq \tau}dt$, $X$ an $\mathbb F$-adapted process such that
$$dX_t= \int_A b(t, X_t, a)\nu_t(da)dt + \int_A \sigma(t, X_t, a)M(dt, da),\quad t\leq \tau, \quad \mathbb P \circ X_0^{-1}= m_0^*.$$
Define now the measures
$$\mu =\mathbb P \circ \left(\tau\wedge \tau_{\mathcal{O}}^X, X_{\tau\wedge \tau_{\mathcal{O}}^X}\right)^{-1},$$
$$m_t(B\times C)= \mathbb E^{\mathbb P}\left[ \mathds{1}_B(X_t)\nu_{t}(C) \mathds{1}_{t\leq \tau\wedge \tau_{\mathcal{O}}^X}\right],  \quad B\in \mathcal{B}(\bar{\mathcal{O}}), \quad C\in \mathcal{B}(A), \quad t\in [0, T].$$
Then $(\mu, m)\in \mathcal{R}$.
\end{proposition}

We refer to Appendix \ref{sec B.1} for the definition of martingale measures and their properties.

\begin{proof}
Let $u\in C_b^{1, 2}([0, T]\times \bar{\mathcal{O}})$. Applying Itô's formula,
\begin{align*}
u\left(\tau \wedge \tau_{\mathcal{O}}^X, X_{\tau \wedge \tau_{\mathcal{O}}^X}\right) &= u(0, X_0) + \int_0^{\tau \wedge \tau_{\mathcal{O}}^X} \int_A \left(\frac{\partial u}{\partial t} +\mathcal L u\right) (t, X_t, a)\nu_{t}(da)dt \\
& \quad + \int_0^{\tau \wedge \tau_{\mathcal{O}}^X} \int_A \left(\sigma \frac{\partial u}{\partial x}\right)(t, X_t, a) M(dt, da).
\end{align*}
Now taking the expectation and using the fact that $\sigma\partial_x u $ is bounded, we get $(\mu, m)\in \mathcal{R}$.
\end{proof}

We now show that from the forward equation $\eqref{constraint_fp}$, we can deduce that for almost every $t \in [0,T]$, $m_t$ is a subprobability measure.

\begin{lemma}[\textit{Subprobability property of the flow of measures}]\label{lemma1}
Let $(\mu, m)\in \mathcal{R}$, then $m_t(\bar{\mathcal{O}}\times A)\leq 1$ $t$-a.e. on $[0, T]$.
\end{lemma}

\begin{proof}
For every test function $u(t, x)=\int_{t}^{T} f(s) ds$ with $f$ a non-negative bounded continuous function, we have
\begin{equation*}
\int_{0}^{T} f(t) m_t(\bar{\mathcal{O}}\times A)dt \leq \int_{0}^{T} f(t) dt,
\end{equation*}
since $\mathcal{L}u=0$. Let $B=\{ t\in[0, T]: m_t(\bar{\mathcal{O}}\times A)>1 \}\in \mathcal B ([0, T])$ (because $t\mapsto m_t(\bar{\mathcal{O}}\times A)$ is measurable). We define $f:[0, T]\rightarrow \mathbb R$ such that $f(t)=\mathds{1}_B(t)$. Since $C([0,T])$ is dense in $L^1([0, T])$, there exists a sequence $(f_n)_{n\geq 1}\subset C([0, T])$ converging to $f$ in $L^1([0, T])$. We define a new sequence $(\overline{f}_n)_{n\geq 1}\subset C([0, T])$ as
\begin{equation*}
\overline{f}_{n}(t) :=\min \left\{\max \left\{f_{n}(t), 0\right\}, 1\right\}, \quad t \in[0, T].
\end{equation*}
Since for all $n\geq 1$,
\begin{equation*}
\int_0^T|\overline{f}_n(t)-f(t)|dt\leq \int_0^T|f_n(t)-f(t)|dt,
\end{equation*}
we conclude that $(\overline{f}_n)_{n\geq 1}$ converges to $f$ in $L^1([0, T])$. Up to taking a subsequence, we suppose without loss of generality that $(\overline{f}_n)_{n\geq 1}$ converges to $f$ $t$-almost everywhere on $[0, T]$. On the other hand, for all $t\in [0, T]$ and all $n\geq 1$,
\begin{equation*}
|\overline{f}_n(t)(1-m_t(\bar{\mathcal{O}}\times A))|\leq 1 + m_t(\bar{\mathcal{O}}\times A).
\end{equation*}
By dominated convergence,
\begin{align*}
0 &\geq \int_{0}^{T} \mathds{1}_B(t)(1- m_t(\bar{\mathcal{O}}\times A)) dt = \int_{0}^{T} f(t)(1- m_t(\bar{\mathcal{O}}\times A)) dt\\
&=\lim_{n\rightarrow\infty} \int_{0}^{T} \overline{f}_{n}(t)(1- m_t(\bar{\mathcal{O}}\times A)) dt \geq 0.
\end{align*}
Since $t\mapsto \mathds{1}_B(t)(1- m_t(\bar{\mathcal{O}}\times A))$ is non-positive, we conclude that $m_t(\bar{\mathcal{O}}\times A)\leq 1$ $t$-almost everywhere on $[0, T]$.
\end{proof}

The next Lemma extends Lemma 3.3.ii. in \cite{bdt2020} to our general framework and since the proof is different, we give it in detail. Before presenting this result, we first recall the definition of the space of functions of bounded variation.
\begin{definition}[\textit{The space of bounded variation functions $\operatorname{BV}(]0, T[)$}]
A function $\varphi$ is said to be of bounded variation on the open interval $]0, T[$, denoted $\varphi\in \operatorname{BV}(]0, T[)$, if $\varphi\in L^1(]0, T[)$ and
$$V(\varphi, ]0, T[):=\sup\left\{ \int_0^T\psi'(t)\varphi(t)dt: \psi\in C_c^1(]0, T[), \; \|\psi\|_\infty\leq 1 \right\}<\infty,$$
where $C_c^1(]0, T[)$ denotes the set of $C^1$ functions on $]0, T[$ with compact support. The space $\operatorname{BV}(]0, T[)$ is endowed with the norm
$$\|\varphi\|_{\operatorname{BV}}:=\|\varphi\|_1 + V(\varphi, ]0, T[),$$
where $\|\cdot\|_1$ denotes the usual $L^1$-norm.
\end{definition}

\begin{lemma}[\textit{A bounded variation property}]\label{lemma2}
Let $h \in C^{1,2}_b([0, T]\times \bar{\mathcal{O}})$ and $(\mu, m)\in \mathcal{R}$. Then for every $\psi \in C^{1}([0, T])$,
\begin{equation*}
\int_{0}^{T} \psi^{\prime}(t)\left(\int_{\bar{\mathcal{O}}\times A} h(t, x) m_{t}(dx, da)\right) dt \leq C\|\psi\|_{\infty},
\end{equation*}
for some $C=C(b, \sigma, h) > 0$. In particular,
$$t\mapsto \int_{\bar{\mathcal{O}}\times A} h(t, x) m_{t}(dx, da)\in \operatorname{BV}(]0, T[),$$
and
$$\left\|\int_{\bar{\mathcal{O}}\times A} h(\cdot, x) m_{\cdot}(dx, da)\right\|_{\operatorname{BV}}\leq T\|h\|_\infty + C.$$
\end{lemma}

\begin{proof}
We consider the test function
\begin{equation*}
u(t, x) = -\psi(t) h(t,x) .
\end{equation*}
We have $u \in C^{1,2}_b([0, T] \times \bar{\mathcal{O}})$. Now, using the constraint \eqref{constraint_fp}, the fact that $\mu$ belongs to  $\mathcal{P}([0,T] \times \bar{\mathcal{O}})$ and bounding $h$, its derivatives, the diffusion coefficients and the measures $(m_t)$ (Lemma \ref{lemma1}) by constants, we get
\begin{equation}\label{bv norm}
\int_{0}^{T} \psi^{\prime}(t)\left(\int_{\bar{\mathcal{O}}\times A} h(t,x) m_{t}(dx, da)\right) dt \leq C\|\psi\|_{\infty},
\end{equation}
for some $C=C(b, \sigma, h) > 0$. We conclude that 
$$t\mapsto \int_{\bar{\mathcal{O}}\times A} h(t, x) m_{t}(dx, da)\in \operatorname{BV}(]0, T[).$$
The estimate on the $\operatorname{BV}$-norm comes from Lemma \ref{lemma1} and taking the supremum in \eqref{bv norm} over the set of $\psi\in C^1_c(]0, T[)$ such that $\|\psi\|_\infty\leq 1$.
\end{proof}
We now provide the following convergence result. {\color{black}We recall that $m^n\rightarrow m$ in $V$ if $m^n_t(dx, da)dt$ converges weakly to $m_t(dx, da)dt$.}

\begin{lemma}[\textit{An $L^1$ convergence result}]\label{lemma2bis}
Let $h \in C_b([0, T]\times \bar{\mathcal{O}})$ and $(\mu^n, m^n)_{n\geq 1}\subset \mathcal{R}$ such that $m^n\rightarrow m$ in $V$. Then,
\begin{equation*}
\int_{\bar{\mathcal{O}}\times A} h(\cdot, x) m_{\cdot}^n(dx, da)\underset{n\rightarrow\infty}{\longrightarrow} \int_{\bar{\mathcal{O}}\times A} h(\cdot, x) m_{\cdot}(dx, da)
\end{equation*}
in $L^1([0, T])$.
\end{lemma}

\begin{proof}
It is sufficient to show that given an arbitrary subsequence we can extract a subsubsequence converging to the above limit in $L^1([0, T])$. Consider a subsequence $(\mu^{n_k}, m^{n_k})_{k\geq 1}$. For all $k\geq 1$, by Lemma \ref{lemma2}, $t\mapsto m_t^{n_k}(\bar{\mathcal{O}}\times A)\in \operatorname{BV}(]0, T[)$ and
$$\sup_{k\geq 1}\|m^{n_k}_\cdot(\bar{\mathcal{O}}\times A)\|_{BV}<\infty.$$
By Theorem 3.23 in \cite{ambrosio2000}, up to a subsequence still denoted with $n_k$, the sequence of mappings $\left(m_\cdot^{n_k}(\bar{\mathcal{O}}\times A)\right)_{k \geq 1}$ converges in $L^{1}([0, T])$ to some mapping $z$. By weak convergence of measures and density of $C([0,T])$ on $L^1([0, T])$, we conclude that $m_t(\bar{\mathcal{O}}\times A)=z(t)$ $t$-a.e. on $[0, T]$. Since by Lemma \ref{lemma1}, $m^{n_k}_t(\bar{\mathcal{O}}\times A)\leq 1$ $t$-a.e. on $[0, T]$, then $m_t(\bar{\mathcal{O}}\times A)\leq 1$ $t$-a.e. on $[0, T]$. We fix some arbitrary $\varepsilon>0$. By Proposition 26.2 in \cite{jameson1974}, there exists $h^*\in  C_b^{1, 2}([0, T]\times \bar{\mathcal{O}})$  such that 
$$\|h^*-h\|_\infty < \frac{\varepsilon}{3 T}$$
Since $h^*\in C^{1, 2}_b([0, T]\times\bar{\mathcal{O}})$, we can use the same argument as before and conclude that up to another subsequence still denoted with $n_k$, there exists $k_0\geq 1$ such that for all $k\geq k_0$,
$$\int_0^T \left|\int_{\bar{\mathcal{O}}\times A} h^*(t, x) m^{n_k}_t(dx, da) -  \int_{\bar{\mathcal{O}}\times A} h^*(t, x) m_t(dx, da)\right|dt < \frac{\varepsilon}{3}.$$
From the above estimates, we obtain for all $k\geq k_0$
$$\int_0^T \left|\int_{\bar{\mathcal{O}}\times A} h(t, x) m^{n_k}_t(dx, da) -  \int_{\bar{\mathcal{O}}\times A} h(t, x) m_t(dx, da)\right|dt < \varepsilon.$$
\end{proof}

Since the elements of $V$ are identified with measures whose marginals with respect to the time variable are absolutely continuous with respect to the Lebesgue measure, we can expect less regularity on the time component of the test functions, as it can be seen in the following lemma.

\begin{lemma}[\textit{Stable convergence}]\label{lemma3}
Let $(\mu^n, m^n)_{n\geq 1} \subset \mathcal{R}$ such that $m^n\rightarrow m$ in $V$. Then $m^n_t(dx, da)dt\rightarrow m_t(dx, da)dt$ in the stable topology, that is, for any $h\in\mathcal{S}$, 
\begin{equation*}
\lim_{n\rightarrow\infty}\int_0^T\int_{\bar{\mathcal{O}}\times A} h(t, x, a)m_t^n(dx, da)dt =  \int_0^T \int_{\bar{\mathcal{O}}\times A} h(t, x, a)m_t(dx, da)dt.
\end{equation*}
\end{lemma}
\begin{proof}
We are going to use Corollary 2.9 of \cite{jacod1981}. We already know by definition of the convergence in $V$ that $m^n_t(dx, da)dt\rightharpoonup m_t(dx, da)dt$, where we use the standard notation $\rightharpoonup$ for the weak convergence. We need to prove that $(m^n_t(\bar{\mathcal{O}}\times A)dt)_{n\geq 1}$ is relatively compact in $\mathcal{M}([0, T])$ endowed with the weak topology generated by the bounded and measurable functions from $[0, T]$ to $\mathbb R$. Since $\mathcal{B}([0, T])$ is countably generated, by Proposition 2.10 in \cite{jacod1981}, this topology is metrizable, hence it is sufficient to show that for every subsequence of $(m^n(\bar{\mathcal{O}}\times A)dt)_{n\geq 1}$, there exists a subsubsequence converging for the weak topology generated by the bounded and measurable functions from $[0, T]$ to $\mathbb R$. 
Let $(m^{n_k}(\bar{\mathcal{O}}\times A)dt)_{k\geq 1}$ be a subsequence of $(m^n(\bar{\mathcal{O}}\times A)dt)_{n\geq 1}$. Then $(m^{n_k})_{k\geq 1}$ converges also to $m$ in V. By Lemma \ref{lemma2bis}, $\left(m^{n_{k}}_\cdot(\bar{\mathcal{O}}\times A)\right)_{k \geq 1}$ converges in $L^{1}([0, T])$ to $m_\cdot(\bar{\mathcal{O}}\times A)$. Finally, for any function $\phi:[0, T]\rightarrow \mathbb R$ bounded and measurable,
$$\left|\int_0^T \phi(t)m^{n_{k}}_t(\bar{\mathcal{O}}\times A)dt -  \int_0^T \phi(t)m_t(\bar{\mathcal{O}}\times A)dt\right|\leq \|\phi\|_\infty \|m^{n_{k}}_{\cdot}(\bar{\mathcal{O}}\times A) - m_{\cdot}(\bar{\mathcal{O}}\times A)\|_1\underset{k\rightarrow\infty}{\longrightarrow} 0.$$
\end{proof}

We now prove the compactness of the set of constraints $\mathcal{R}$, which extends Lemma 3.5. in \cite{bdt2020} to our setting. The proof is more involved and we present it here for sake of clarity.

\begin{theorem}\label{R compact}
The set $\mathcal{R}$ is compact.
\end{theorem}

\begin{proof}
Since $\mathcal{R}\subset\mathcal{P}([0, T]\times\bar{\mathcal{O}})\times V$, which is metrizable, it suffices to show that $\mathcal{R}$ is sequentially compact. Consider a sequence $(\mu^n, m^n)_{n\geq 1}\subset \mathcal{R}$. For $k\geq 1$, define the test function $u_k(t, x)=(T+1-t)\phi_k(x)$, where
$$
\phi_k(x)=\ln \left\{1+|x|^{3}\left(\frac{3 x^{2}}{5 k^{2}}-\frac{3|x|}{2 k}+1\right)\right\} \mathds{1}_{|x| \leq k}+\ln \left\{1+\frac{k^{3}}{10}\right\} \mathds{1}_{|x|>k}.
$$
For each $k\geq 1$, $\phi_k\in C^{1, 2}_b([0, T]\times \bar{\mathcal{O}})$ and $\phi_k$ is non-negative.
We have 
$$\int_\mathcal{O} u_k(0, x)m_0^*(dx) + \int_0^T \int_{\bar{\mathcal{O}}\times A} \left(\frac{\partial u_k}{\partial t} +\mathcal L u_k\right) (t, x, a)m_t^n(dx, da)dt\geq 0,$$
which implies,
$$\int_{\bar{\mathcal{O}}\times A} \phi_k(x)m^{n}_t(dx, da)dt\leq  (T+1)\int_\mathcal{O}\phi_k(x)m_0^*(dx) + \int_0^T \int_{\bar{\mathcal{O}}\times A} \mathcal L u_k(t, x, a)m_t^n(dx, da)dt.$$
One can show that there exists a constant $C\geq 0$ independent from $k$ such that $\phi_k'$ and $\phi''_k$ are bounded by $C$. By Lemma \ref{lemma1}, for all $n\geq 1$, $m^n_t(\bar{\mathcal{O}}\times A)\leq 1$ $t$-a.e. on $[0, T]$, which implies that there exists a constant $C'\geq 0$ independent from $n$ and $k$ such that for all $n\geq 1$ and $k\geq 1$,
$$\int_{\bar{\mathcal{O}}\times A} \phi_k(x)m^{n}_t(dx, da)dt\leq  (T+1)\int_\mathcal{O}\phi_k(x)m_0^*(dx) + C'.$$
Now, since $(\phi_k)_{k\geq 1}$ is a non-decreasing sequence converging to $\phi(x)=\ln(1+|x|^3)$, by monotone convergence theorem, we get for all $n\geq 1$
$$\int_{\bar{\mathcal{O}}\times A} \phi(x)m^{n}_t(dx, da)dt\leq  (T+1)\int_\mathcal{O}\phi(x)m_0^*(dx) + C'.$$
Letting $\nu^n(dt, dx, da)=m_t^n(dx, da)dt\in \mathcal{M}([0, T]\times\bar{\mathcal{O}}\times A)$, we conclude
$$\sup_{n\geq 1}\int_{[0, T]\times\bar{\mathcal{O}}\times A} \phi(x)\nu^n(dt, dx, da)<\infty $$
Since $\phi$ is non-negative and for all $r\geq 0$, the set
$$\{(t, x, a)\in [0, T]\times\bar{\mathcal{O}}\times A: \phi(x)\leq r \}=[0, T]\times (\bar{\mathcal{O}}\cap [-(e^r-1)^{1/3}, (e^r-1)^{1/3}]) \times A$$
is compact, we conclude that $(\nu^n)_{n\geq 1}$ is tight. Since by Lemma \ref{lemma1}, $\nu^n([0, T]\times\bar{\mathcal{O}}\times A)\leq T$, by Prokhorov's Theorem (Theorem 8.6.2 in \cite{bogachev2007} (Volume 2)), there exists $\nu\in \mathcal{M}([0, T]\times\bar{\mathcal{O}}\times A)$ such that, up to a subsequence, $\nu^n\rightharpoonup \nu$. Using the test function $u(t, x)=\int_t^T \varphi(t)dt$ with $\varphi$ a non-negative continuous function, for all $n\geq 1$
$$\int_{[0, T]\times\bar{\mathcal{O}}\times A} \varphi(t)\nu^n(dt, dx, da)\leq \int_0^T \varphi(t)dt.$$
Taking $n\rightarrow\infty$, we conclude that $\int_{\bar{\mathcal{O}}\times A} \nu(dt, dx, da)$ is absolutely continuous with respect to the Lebesgue measure on $[0, T]$, which allows the disintegration $\nu(dt, dx, da)=m_t(dx, da)dt$ for some $m\in V$. We conclude that $m^n\rightarrow m$ in $V$. Now, using the same test function $u_k$,
\begin{align*}
\int_{[0, T]\times \bar{\mathcal{O}}} u_k(t, x)\mu^n(dt, dx) &\leq C' + (T+1)\int_\mathcal{O} u_k(0, x)m_0^*(dx) - \int_0^T \int_{\bar{\mathcal{O}}\times A} \phi_k (x)m_t^n(dx, da)dt\\
&\leq C' + (T+1)\int_\mathcal{O} \phi_k(x)m_0^*(dx).
\end{align*}
By the monotone convergence theorem and using that $u_k(t, x)\geq \phi_k(x)$,
$$\sup_{n\geq 1} \int_{[0, T]\times \bar{\mathcal{O}}} \phi(x)\mu^n(dt, dx)<\infty,$$
which proves that $(\mu^n)_{n\geq 1}$ is tight. By Prokhorov's theorem there exists $\mu\in \mathcal{P}([0, T]\times \bar{\mathcal{O}})$ such that, up to another subsequence, $\mu^n\rightharpoonup \mu$. Let $u\in C_b^{1, 2}([0, T]\times \bar{\mathcal{O}})$. Taking limits in 
$$\int_{[0, T]\times \bar{\mathcal{O}}} u(t, x)\mu^n(dt, dx)= \int_\mathcal{O} u(0, x)m_0^*(dx) + \int_0^T \int_{\bar{\mathcal{O}}\times A} \left(\frac{\partial u}{\partial t} +\mathcal L u\right) (t, x, a)m_t^n(dx, da)dt,$$
and using that,
$$\frac{\partial u}{\partial t} +\mathcal L u\in \mathcal{S},$$
we get by Lemma \ref{lemma3}
$$\int_{[0, T]\times \bar{\mathcal{O}}} u(t, x)\mu(dt, dx)= \int_\mathcal{O} u(0, x)m_0^*(dx) + \int_0^T \int_{\bar{\mathcal{O}}\times A} \left(\frac{\partial u}{\partial t} +\mathcal L u\right) (t, x, a)m_t(dx, da)dt,$$
which shows that $(\mu, m)\in \mathcal{R}$ and hence $\mathcal{R}$ is compact.
\end{proof}

\paragraph{The existence result.}

We now give the main result of this subsection, which consists in showing that there exists an admissible maximizer $(\mu^\star, m^\star) \in \mathcal{R}$ for $\Gamma$.
\begin{theorem}[\textit{Existence of a solution for the LP problem}]\label{single agent linear}
There exists a solution to the linear programming problem for the single agent.
\end{theorem}

\begin{proof}
Let $(\mu^n, m^n)_{n\geq 1}\subset \mathcal{R}$ be a maximizing sequence, that is
$$\lim_{n\rightarrow\infty}\Gamma(\mu^n, m^n)=\sup_{(\mu, m)\in \mathcal{R}}\Gamma(\mu, m).$$
By Theorem \ref{R compact}, we get that up to a subsequence, $(\mu^n, m^n)_{n\geq 1}$ converges to some $(\mu^\star, m^\star)\in \mathcal{R}$. By Lemma \ref{lemma3}, $m_t^n(dx, da)dt\rightarrow m_t^\star(dx, da)dt$ in the stable topology. By Lemma \ref{lemma2bis} we have that $m_\cdot^n(\bar{\mathcal{O}}\times A)\rightarrow m_\cdot^\star(\bar{\mathcal{O}}\times A)$ in $L^1([0, T])$. By Proposition 2.11 in \cite{jacod1981}, 
$$\limsup_{n\rightarrow\infty} \int_0^T\int_{\bar{\mathcal{O}}\times A}(f(t, x, a)+\|f\|_\infty)m_t^n(dx, da)dt\leq \int_0^T\int_{\bar{\mathcal{O}}\times A}(f(t, x, a)+\|f\|_\infty)m_t^\star(dx, da)dt.$$
Now since $\|f\|_\infty\int_0^T m_t^n(\bar{\mathcal{O}}\times A)dt\rightarrow \|f\|_\infty\int_0^T m_t^\star(\bar{\mathcal{O}}\times A)dt$, we get
$$\limsup_{n\rightarrow\infty} \int_0^T\int_{\bar{\mathcal{O}}\times A}f(t, x, a)m_t^n(dx, da)dt\leq \int_0^T\int_{\bar{\mathcal{O}}\times A}f(t, x, a)m_t^\star(dx, da)dt.$$
On the other hand, since $\mu^n\rightharpoonup \mu^\star$ and $g$ is upper semicontinuous and bounded above, then Portmanteau theorem implies
$$\limsup_{n\rightarrow\infty}\int_{[0, T]\times \bar{\mathcal{O}}}g(t, x)\mu^n(dt, dx)\leq \int_{[0, T]\times \bar{\mathcal{O}}}g(t, x)\mu^\star(dt, dx).$$
We conclude that
$$\Gamma(\mu^\star, m^\star)= \sup_{(\mu, m)\in \mathcal{R}}\Gamma(\mu, m).$$
\end{proof}
\begin{remark}
In the case when there is no control and only optimal stopping, the above existence result holds under weaker assumptions on the coefficients and reward functions compared to \cite{bdt2020}.
\end{remark}

The following result is well known in the literature (see \cite{elkaroui1987,haussmann1990,stockbridge2012,lacker2015}) but we give a proof for sake of completeness. 

\begin{proposition}[\textit{Existence of a strict control LP solution}]\label{strict cont} 
Suppose that for all $(t, x)\in [0, T]\times \bar{\mathcal{O}}$, the subset
$$K(t, x):=\{(b(t, x, a), \sigma^2(t, x, a), z):a\in A, z\leq f(t, x, a)\}$$
of $\mathbb R \times\mathbb R_+ \times \mathbb R$ is convex, then there exists a strict control LP solution. 
\end{proposition}
\begin{proof}
Let $(\mu^\star, m^\star)$ be a maximizer of the LP problem which exists by Theorem \ref{single agent linear}. Let $\nu_{t, x}^\star$ such that 
$$m_t^\star(dx, da)dt=\nu_{t, x}^\star(da)m_t^\star(dx, A)dt.$$ 
Let $(t, x, a)\in [0, T]\times \bar{\mathcal{O}}\times A$ be arbitrary. We have that
$$\left(b(t, x, a), \sigma^2(t, x, a), f(t, x, a) \right)\in K(t, x).$$
As in Proposition 3.5 of \cite{haussmann1990} one can prove that $K(t, x)$ is closed. Now, by Theorem I.6.13 (p. 145) in \cite{warga1972},
$$\left(\int_A b(t, x, a)\nu_{t, x}^\star(da), \int_A\sigma^2(t, x, a)\nu_{t, x}^\star(da), \int_A f(t, x, a)\nu_{t, x}^\star(da) \right)\in K(t, x).$$
By definition of $K(t, x)$ and Theorem A.9 in \cite{haussmann1990} there exists a measurable function $(t, x)\mapsto \alpha(t, x)$ such that
$$\int_A b(t, x, a)\nu_{t, x}^\star(da)=b(t, x, \alpha(t, x)), \quad \int_A\sigma^2(t, x, a)\nu_{t, x}^\star(da)=\sigma^2(t, x, \alpha(t, x)),$$
$$\int_A f(t, x, a)\nu_{t, x}^\star(da)\leq f(t, x, \alpha(t, x)).$$
Define $\bar m_t(dx, da)=\delta_{\alpha(t, x)}(da)m_t^\star(dx, A)$ for each $t\in [0, T]$. We conclude that $(\mu^\star, \bar m)\in \mathcal{R}$ and $\Gamma(\mu^\star, m^\star)\leq \Gamma(\mu^\star, \bar m)$.
\end{proof}

\subsection{Relation with the \textit{weak formulation}}

Following the literature on the linear programming formulation of stochastic control problems for Markov processes, we now prove prove that solving the linear program allows to construct a solution to the weak problem. The terminology \textit{weak} is introduced in analogy to the notion of \textit{weak solution} of an SDE, the idea being to consider the probabilistic set-up as part of the solution. The weak formulation is of two types, depending on the type of control, either strict control (valued in $A$) or relaxed control (valued in $\mathcal{P}(A)$).

\begin{assumption}\label{as exit time}
\item We assume here that one of the following statements holds:
\begin{enumerate}[(1)]
\item \textit{Unattainable boundary}: $b$, $\sigma$ and $\mathcal{O}$ are such that, for every filtered probability space $(\Omega, \mathcal{F}, \mathbb F, \mathbb P)$, $\mathbb F$-stopping time $\tau$ such that $\tau\leq T$ $\mathbb P$-a.s., $\mathbb{F}$-progressively measurable process $\nu$ with values in $\mathcal{P}(A)$, continuous $\mathbb F$-martingale measure $M$ such that $M^\tau$ has intensity $\nu_t(da)\mathds{1}_{t\leq\tau}dt$, and $\mathbb F$-adapted process $X$ such that
$$dX_t= \int_A b(t, X_t, a)\nu_t(da)dt + \int_A \sigma(t, X_t, a)M(dt, da),\quad t\leq \tau, \quad \mathbb P \circ X_0^{-1}= m_0^*,$$
we have
$$\mathbb P \left(\tau_\mathcal{O}^{\tilde X}\geq T \right)=1,$$
where $\tilde X=X_{\cdot\wedge \tau}$.
\item \textit{Attainable boundary}: $\sigma$ does not depend on the control $a$ and there exists $c_\sigma>0$ such that for all $(t,x)\in [0, T]\times \mathbb R$, $\sigma(t,x)\geq c_\sigma$.
\end{enumerate}
\end{assumption}

We now give the weak formulations (with strict optimal stopping/control, resp.~with strict optimal stopping and relaxed control) of the single agent problem.

\begin{definition}[\textit{Weak formulation with strict optimal stopping/control}]
Define $\mathcal{U}^W$ as the set of tuples $U=(\Omega, \mathcal F, \mathbb F, \mathbb P, W, \alpha, \tau, X)$ such that $(\Omega, \mathcal{F}, \mathbb F, \mathbb P)$ is a filtered probability space, $W$ is an $\mathbb F$-Brownian motion, $\alpha$ is an $\mathbb{F}$-progressively measurable process with values in $A$, $\tau$ is an $\mathbb F$-stopping time such that $\tau\leq T$ $\mathbb P$-a.s., $X$ is an $\mathbb F$-adapted process such that
$$dX_t= b(t, X_t, \alpha_t)dt + \sigma(t, X_t, \alpha_t)dW_t,\quad t\leq \tau, \quad \mathbb P \circ X_0^{-1}= m_0^*.$$
Let $\mathcal{H}^W:\mathcal{U}^W\rightarrow \mathbb R$ be defined by
$$\mathcal{H}^W(U)=\mathbb{E}^{\mathbb P}\left[\int_{0}^{\tau \wedge \tau_{\mathcal{O}}^{X}} f\left(t, X_{t}, \alpha_t \right) dt  + g\left(\tau\wedge \tau_{\mathcal{O}}^{X}, X_{\tau\wedge \tau_{\mathcal{O}}^{X}}\right)\right]$$
for all $U=(\Omega, \mathcal F, \mathbb F, \mathbb P, W, \alpha, \tau, X)\in\mathcal{U}^W$. The value for the weak formulation with strict control/optimal stopping is defined by 
\begin{equation}\label{value weak strict}
V^W := \sup_{U\in \mathcal{U}^W} \mathcal{H}^W(U).    
\end{equation}
Moreover, $U^\star \in \mathcal{U}^{W}$ is a solution of the \textit{weak problem with strict optimal stopping/control} if 
$$\mathcal{H}^{W}(U^\star)=V^{W}.$$
\end{definition}

\begin{definition}[\textit{Weak formulation with strict optimal stopping and relaxed control}]
Define $\mathcal{U}^R$ as the set of tuples $U=(\Omega, \mathcal F, \mathbb F, \mathbb P, M, \nu, \tau, X)$ such that $(\Omega, \mathcal{F}, \mathbb F, \mathbb P)$ is a filtered probability space, $\tau$ is an $\mathbb F$-stopping time such that $\tau\leq T$ $\mathbb P$-a.s., $\nu$ is an $\mathbb{F}$-progressively measurable process with values in $\mathcal{P}(A)$, $M$ is a continuous $\mathbb F$-martingale measure such that $M^\tau$ has intensity $\nu_t(da)\mathds{1}_{t\leq\tau}dt$, $X$ is an $\mathbb F$-adapted process such that
$$dX_t= \int_A b(t, X_t, a)\nu_t(da)dt + \int_A \sigma(t, X_t, a)M(dt, da),\quad t\leq \tau, \quad \mathbb P \circ X_0^{-1}= m_0^*.$$
Let $\mathcal{H}^R:\mathcal{U}^R\rightarrow \mathbb R$ defined by
$$\mathcal{H}^R(U)=\mathbb{E}^{\mathbb P}\left[\int_{0}^{\tau \wedge \tau_{\mathcal{O}}^{X}} \int_A f\left(t, X_{t}, a \right) \nu_t(da)dt  + g\left(\tau\wedge \tau_{\mathcal{O}}^{X}, X_{\tau\wedge \tau_{\mathcal{O}}^{X}}\right)\right]$$
for all $U=(\Omega, \mathcal F, \mathbb F, \mathbb P, M, \nu, \tau, X)\in\mathcal{U}^R$. The value for the weak formulation with strict optimal stopping and relaxed control is defined by 
\begin{equation}\label{value weak relaxed}
V^R := \sup_{U\in \mathcal{U}^R} \mathcal{H}^R(U).    
\end{equation}
Moreover, $U^\star \in \mathcal{U}^{R}$ is a solution of the \textit{weak problem with strict optimal stopping and relaxed control} if 
$$\mathcal{H}^{R}(U^\star)=V^{R}.$$
\end{definition}

\begin{theorem}[\textit{Existence of a weak solution with Markovian relaxed control}]\label{single agent weak}
Suppose that Assumption \ref{as exit time} is also in force. Then there exists a solution to the weak problem with Markovian relaxed control.
\end{theorem}

\begin{proof}
Let $(\mu^\star, m^\star)$ be a maximizer of the LP problem which exists by Theorem \ref{single agent linear}. Let $\nu_{t, x}^\star$ such that 
$$m_t^\star(dx, da)dt=\nu_{t, x}^\star(da)m_t^\star(dx, A)dt.$$
By Theorem \ref{SDE rep}, there exist a filtered probability space $(\Omega, \mathcal{F}, \mathbb F, \mathbb P)$, an $\mathbb F$-adapted process $X$, an $\mathbb F$-stopping time $\tau$ such that $\tau\leq \tau_{\mathcal{O}}^X\wedge T$ $\mathbb P$-a.s., a continuous $\mathbb F$-martingale measure $M$ with intensity $\nu_{t, X_t}^\star(da)\mathds{1}_{t\leq \tau}dt$, such that
$$X_{t\wedge \tau}= \int_0^{t\wedge \tau}\int_A b(t, X_t, a)\nu_{t, X_t}(da)dt + \int_0^{t\wedge \tau}\int_A \sigma(t, X_t, a)M(dt, da), \quad \mathbb P \circ X_0^{-1}= m_0^*,$$
$$\mu^\star =\mathbb P \circ (\tau, X_\tau)^{-1},$$
$$m_t^\star(B\times C)= \mathbb E^{\mathbb P}\left[ \mathds{1}_B(X_t)\nu_{t, X_t}^\star(C) \mathds{1}_{t\leq \tau}\right], \quad B\in \mathcal{B}(\bar{\mathcal{O}}), \quad C\in \mathcal{B}(A), \quad t-a.e.$$

Let $(\Omega', \mathcal{F}', \mathbb F', \mathbb P')$ be another filtered probability space, $\tau'$ an $\mathbb F'$-stopping time such that $\tau'\leq T$ $\mathbb P'$-a.s., $\nu'$ an $\mathbb F'$-progressively measurable process with values in $\mathcal{P}(A)$, $M'$ a continuous $\mathbb F'$-martingale measure such that $(M')^\tau$ has intensity $\nu_t'(da)\mathds{1}_{t\leq\tau'}dt$, $X'$ an $\mathbb F'$-adapted process such that
$$dX_t'= \int_A b(t, X_t', a)\nu_t'(da)dt + \int_A\sigma(t, X_t', a)M'(dt, da),\quad t\leq \tau', \quad \mathbb P'\circ (X_0')^{-1}= m_0^*.$$
Define for $t\in [0, T]$
$$m_t'(B \times C)=\mathbb E^{\mathbb P'}\left[\mathds 1_B(X'_t) \nu_t'(C) \mathds 1_{t\leq \tau'\wedge \tau_{\mathcal{O}}^{X'}}\right], \quad B\in \mathcal{B}(\bar{\mathcal{O}}), \quad C\in \mathcal{B}(A),$$
$$\mu'=\mathbb P' \circ \left(\tau'\wedge \tau_{\mathcal{O}}^{X'}, X'_{\tau'\wedge \tau_{\mathcal{O}}^{X'}}\right)^{-1}.$$
By Proposition \ref{SDE to measures}, $(\mu', m')\in \mathcal{R}$. Since $(\mu^\star, m^\star)$ is a maximizer of the LP problem, $\Gamma(\mu', m')\leq \Gamma(\mu^\star, m^\star)$, which means
\begin{align*}
&\mathbb{E}^{\mathbb P'}\left[\int_{0}^{\tau'\wedge \tau_{\mathcal{O}}^{X'} }\int_A f\left(t, X'_{t},a\right)\nu_t'(da)dt  + g\left(\tau'\wedge \tau_{\mathcal{O}}^{X'}, X'_{\tau'\wedge \tau_{\mathcal{O}}^{X'}}\right)\right] \\
&\leq \mathbb{E}^{\mathbb P}\left[\int_{0}^{\tau\wedge \tau_{\mathcal{O}}^{X}} \int_A f\left(t, X_{t},a\right)\nu_{t,X_t}(da)dt  + g\left(\tau\wedge \tau_{\mathcal{O}}^{X}, X_{\tau\wedge \tau_{\mathcal{O}}^{X}}\right)\right].
\end{align*}
\end{proof}

\begin{corollary}[\textit{Existence of a weak solution with markovian strict control}]\label{weak sol strict cont} 
Suppose that for all $(t, x)\in [0, T]\times \bar{\mathcal{O}}$, the subset
$$K(t, x):=\{(b(t, x, a), \sigma^2(t, x, a), z):a\in A, z\leq f(t, x, a)\}$$
of $\mathbb R \times\mathbb R_+ \times \mathbb R$ is convex and Assumption \ref{as exit time} is in force. Then there exists a weak solution with markovian strict control.
\end{corollary}
\begin{proof}
The proof follows by Proposition \ref{strict cont}, Theorem \ref{SDE rep} and the same argument as in Theorem \ref{single agent weak}.
\end{proof}

\subsection{Equivalence of different formulations of the controlled/stopped diffusion processes problem and relation with PDEs}

{\color{black}
In this part, we aim to show the equivalence between the different formulations. The values for the linear programming and weak formulations are already defined, so we define now the value for the strong formulation.
\begin{definition}[\textit{Strong formulation}]
Let $t\in [0, T]$, we denote by $\mathbb{F}^t$ the filtration given by $\mathcal{F}_{s}^t=\sigma\left(W_{r}-W_t, t \leq r \leq s\right) \vee \mathcal{N}$, $s\geq t$. Denote by $\mathcal{T}_t$ the set of stopping times with respect to this filtration with values in $[t, T]$. Let $\mathbb{A}_t$ be the set of $\mathbb{F}^t$-progressively measurable control processes taking values in $A$. The value function for the strong formulation is given by 
\begin{equation}\label{value_fct}
v(t, x)=\sup _{\tau \in \mathcal{T}_t, \alpha \in \mathbb{A}_t} \mathbb{E}\left[\int_{t}^{\tau \wedge \tau_{\mathcal{O}}^{t, x, \alpha}} f\left(s, X_{s}^{t, x, \alpha}, \alpha_s\right) ds + g\left(\tau \wedge \tau_{\mathcal{O}}^{t, x, \alpha}, X_{\tau \wedge \tau_{\mathcal{O}}^{t, x, \alpha}}^{t, x, \alpha}\right)\right],
\end{equation}
with $(t, x) \in[0, T] \times \mathbb{R}$, $\tau_{\mathcal{O}}^{t, x, \alpha} :=\inf \left\{s \geq t : \, X_{s}^{t, x, \alpha} \notin \mathcal{O}\right\}$ and $(X^{t, x, \alpha}_s)_{s\in [t, T]}$ is the unique  strong solution of the following stochastic differential equation:
$$X^{t, x, \alpha}_{s}=x + \int_t^s b\left(r, X^{t, x, \alpha}_{r},\alpha_r \right) dr+\int_t^s\sigma\left(r, X^{t, x, \alpha}_{r}, \alpha_r\right) dW_{r}, \quad s\in [t, T].$$
We also define 
\begin{equation}\label{value strong}
V^S:=\int_{\mathcal{O}}v(0, x)m_0^*(dx),    
\end{equation}
which represents the value for the \textit{strong formulation}.
\end{definition}}

\paragraph{The case $\mathcal{O}=\mathbb R$.}

We show that the values at time zero  associated to the different formulations (LP, weak and strong) are equal. {\color{black}In this paragraph, instead of Assumption \ref{as_single_agent}, we impose the following assumption:}
{\color{black}
\begin{assumption}\label{as values}
Suppose $\mathcal{O}=\mathbb R$, and let the following conditions hold true:
\begin{enumerate}[(1)]
\item The coefficients $b:[0, T]\times \mathbb R\times A\rightarrow \mathbb R$ and $\sigma:[0, T]\times \mathbb R\times A\rightarrow \mathbb R_+$ are measurable and Lipschitz in $x$ uniformly on $(t, a)$.
\item The functions $b$, $\sigma$ and $f$ are in $\mathcal{S}$. 
\item The final payoff function $g$ is bounded, measurable and continuous in $x$ for each $t$. 
\end{enumerate}
\end{assumption}}
We give now the definition of the \textit{strong formulation} of the mixed stochastic control/optimal stopping problem.

\begin{theorem}[\textit{Equality of the values of the different formulations}]\label{eq values}
Let Assumption \ref{as values} hold true. Then the values associated to the formulations \eqref{value lp}, \eqref{value weak strict}, \eqref{value weak relaxed} and \eqref{value strong} are equal:
$$V^{S}=V^{W}=V^{R}=V^{LP}.$$
\end{theorem}

\begin{proof}
The proof is organized in two steps.\\

\textit{Step 1.} We first show that $V^{R}=V^{LP}$.\\
Note that since $\mathcal{O}=\mathbb R$, Assumption \ref{as exit time} is satisfied. By Proposition \ref{SDE to measures}, for each $U\in \mathcal{U}^R$, there exists $(\mu, m)\in \mathcal{R}$ such that $\mathcal{H}^R(U)=\Gamma (\mu, m)$. Therefore, we get
$$V^R\leq V^{LP}.$$
Moreover, by Theorem \ref{SDE rep}, for each $(\mu, m)\in \mathcal{R}$ there exists  $U\in \mathcal{U}^R$ satisfying $\Gamma (\mu, m)=\mathcal{H}^R(U)$, leading to
$$V^{LP}\leq V^R.$$

\textit{Step 2.} We prove that $V^{S}=V^{W}=V^{R}$.\\ This result follows by Theorem 4.5. in \cite{elkaroui2015}, which uses an equivalent formulation (see p. 18 in \cite{elkaroui2015}), consisting in fixing a canonical space\footnote{The canonical space used in \cite{elkaroui2015} is given by $[0, \infty]\times \mathbb D(\mathbb R_+, \mathbb R^d)\times \mathbb M$, where $\mathbb D(\mathbb R_+, \mathbb R^d)$ is the set of càdlàg paths from $\mathbb R_+$ to $\mathbb R^d$ and $\mathbb M$ is the set of all $\sigma$-finite (Borel) measures on $\mathbb R_+\times A$ whose marginal distribution on $\mathbb R_+$ is the Lebesgue measure. The first space in the product is for the stopping time, the second for the state process, and the third for the relaxed control. Since our state process is continuous, we can replace $\mathbb D(\mathbb R_+, \mathbb R^d)$ by $\mathbb C(\mathbb R_+, \mathbb R^d)$, which is the space of continuous paths.} and optimizing on a set of probability measures. To apply Theorem 4.5. in \cite{elkaroui2015}, we check that the assumptions are satisfied.
Define for $(t, \mathbf y, a)\in \mathbb R_+ \times \mathbb C(\mathbb R_+, \mathbb R^2)\times A$,
$$
\mu(t, \mathbf{y}, a):=
\left(
\begin{matrix}
b(t, \mathbf{y}_t^1, a)\\
f(t, \mathbf{y}_t^1, a)
\end{matrix}
\right)
\mathds{1}_{[0, T]}(t),
$$
$$
\tilde \sigma(t, \mathbf{y}, a):=
\left(
\begin{matrix}
\sigma(t, \mathbf{y}_t^1, a)\\
0
\end{matrix}
\right)
\mathds{1}_{[0, T]}(t).
$$
For each $(t, \mathbf{y}, \alpha)\in \mathbb R_+ \times \mathbb C(\mathbb R_+, \mathbb R^2) \times \mathbb A $, there exists a unique strong solution of the SDE 
$$Y_s^{t, \mathbf{y}, \alpha}=\mathbf{y}_t + \int_t^s \mu(r, Y_{r\wedge \cdot}^{t, \mathbf{y}, \alpha}, \alpha_r)dr + \int_t^s \tilde \sigma(r, Y_{r\wedge \cdot}^{t, \mathbf{y}, \alpha}, \alpha_r)dW_r,$$
with initial condition $Y_s^{t, \mathbf{y}, \alpha}:=\mathbf{y}_s$ for all $s\in [0, t]$. In fact, one can find a strong solution for the first component using the assumptions on $b$ and $\sigma$, and since the second component is fully determined by the first one, we get the existence. We denote by $X$ the first component and by $Z$ the second component. Therefore, the associated controlled/stopped martingale problem has a solution. Note that the coefficients are continuous in the control variable for any $(t, \mathbf{y})\in \mathbb R_+ \times \mathbb C(\mathbb R_+, \mathbb R^2)$. For $(t, \mathbf{y})\in [0,\infty] \times \mathbb C(\mathbb R_+, \mathbb R^2)$ let
$$\Phi(t, \mathbf{y})=\left(\mathbf{y}^2_{t\wedge T} + g\left(t\wedge T, \mathbf{y}_{t\wedge T}^1\right)\right).$$
Fix $(t, \mathbf{x})\in \mathbb R_+ \times \mathbb C(\mathbb R_+, \mathbb R)$ and $\mathbf{y}=(\mathbf{x}, \mathbf{0})\in \mathbb C(\mathbb R_+, \mathbb R^2)$, then we have
\begin{align*}
\sup _{\tau \in \mathcal{T}_t, \alpha \in \mathbb{A}_t} \mathbb E \left[ \Phi(\tau, Y^{t, \mathbf{y}, \alpha}_\cdot) \right] 
& = \sup _{\tau \in \mathcal{T}_t, \alpha \in \mathbb{A}_t} \mathbb E \left[ \int_t^{\tau \wedge T} f(s, X^{t, \mathbf{x}, \alpha}_s, \alpha_s)ds + g\left( \tau \wedge T, X^{t, \mathbf{x}, \alpha}_{\tau \wedge T}\right) \right]\\
& = v(t, x).
\end{align*}
Moreover, for each $t\in [0, \infty]$, $\mathbf{y}\mapsto \Phi(t, \mathbf{y}_{t\wedge \cdot})$ is continuous ($\mathbb C(\mathbb R_+, \mathbb R^2)$ is endowed with the topology of uniform convergence on compact subsets of $\mathbb R_+$). Since $f$ and $g$ are bounded, the last assumption of Theorem 4.5. in \cite{elkaroui2015} is satisfied. Then applying Theorem 4.5 in \cite{elkaroui2015} and integrating at time $t=0$ with respect to $m_0^*$ (see Theorem 3.1 (ii) in \cite{elkaroui2015}), we get 
$$V^S = V^W = V^R.$$
The result follows.
\end{proof}

\paragraph{The case $\mathcal{O}$ bounded.}
{\color{black}In this paragraph, instead of Assumption \ref{as_single_agent}, we impose the following assumption:}
{\color{black}
\begin{assumption}\label{as PDE} \leavevmode
\begin{enumerate}[(1)]
\item The domain $\mathcal{O}$ is a bounded open domain of class $C^2$.
\item $\sigma$ does not depend on the control $a$ and is continuous on $[0, T]\times \bar{\mathcal{O}}$. Moreover, there exists $c_\sigma>0$ such that for all $(t, x)\in [0, T]\times \bar{\mathcal{O}}$, $\sigma(t, x)\geq c_\sigma$.
\item The coefficients $b:[0, T]\times \mathbb R\times A\rightarrow \mathbb R$ and $\sigma:[0, T]\times \mathbb R\rightarrow \mathbb R_+$ are measurable, bounded and Lipschitz in $x$ uniformly on the other variables.
\item $f$ is measurable, bounded and continuous on $\bar{\mathcal{O}}$, uniformly with respect to $t$ and $a$.
\item For fixed $(t,x)\in [0, T]\times \bar{\mathcal{O}}$, $a\mapsto b(t,x,a)$ and $a\mapsto f(t,x,a)$ are continuous.
\item $g\in C^{1, 2}_b([0, T]\times \bar{\mathcal{O}})$ and $g(t, x)=0$ for $(t, x)\in (0, T)\times \partial \mathcal{O}$.
\item If $(\mu, m)\in \mathcal{R}$, then $m_t(dx, A)dt$ admits an square integrable density with respect to the Lebesgue measure on $[0, T]\times \bar{\mathcal{O}}$.
\end{enumerate} 
\end{assumption}}

\begin{remark}
In Appendix \ref{sec D} we give sufficient conditions under which (7) in the above assumption is satisfied.
\end{remark}


Let us recall an existence theorem for the strong formulation. The theorem is a particular case of Theorem 3.2, Chapter 4, in \cite{bensoussan1982}.

\begin{theorem}\label{theoremstrong}
Let Assumption \ref{as PDE} be satisfied. Let $v$ be the value function defined in \eqref{value_fct}. Then $v$ is the unique solution belonging to $C([0, T]\times \bar{\mathcal{O}})\cap W^{1, 2, 2}((0, T)\times \mathcal{O})$\footnote{The Sobolev space $W^{1, 2, 2}((0, T)\times \mathcal{O})$ represents the set of functions $u$ such that $u$, $\partial_t u$, $\partial_x u$, $\partial_{xx}u \in L^2((0, T)\times \mathcal{O})$, where the derivatives are understood in the sense of distributions.}, satisfying the following Hamilton-Jacobi-Bellman Variational Inequality (HJBVI)
\begin{equation}\label{HJBVI}
\begin{aligned} 
\min \left(-\frac{\partial v}{\partial t}(t, x)-\sup_{a\in A}\left[\mathcal{L} v(t, x, a) + f(t, x, a)\right], v(t, x)- g(t, x)\right)=0,&\quad (t, x) \in(0, T) \times \mathcal{O}, \\ 
v(t, x)=0,&\quad (t, x) \in(0, T)\times \partial \mathcal{O}, \\ 
v(T, x)=g(T, x),&\quad x \in \mathcal{O}. \end{aligned}
\end{equation}
Moreover, optimal controls are given by
\begin{equation}\label{optimal_control}
\alpha^\star_t(x):=\alpha\left(t, X_t^{x, \alpha^\star}\right), \quad\text{where} \quad \alpha(t, x)\in \arg\max_{a\in A}\left[\mathcal{L} v(t, x, a) + f(t, x, a)\right],  
\end{equation}
\begin{equation}\label{optimal_stop}
\tau^\star(x) :=\inf \left\{0 \leq t \leq T : v\left(t, X_{t}^{x, \alpha^\star}\right)=g\left(t, X_{t}^{x, \alpha^\star}\right)\right\}.
\end{equation}
\end{theorem}

\begin{remark}
Observe that in \cite{bensoussan1982}, they suppose that $b$ and $f$ are continuous on $t$. This assumption is used in their proof to establish continuity of the Hamiltonian, however we need only measurability on the Hamiltonian to use the measurable selection theorem.
\end{remark}

The next Theorem is a slight extension of Theorem 5.2 in \cite{bdt2020}. For sake of clarity we give the proof in Appendix \ref{sec E}.

\begin{theorem}\label{PDE}
Suppose Assumption \ref{as PDE} is in force. Then, the following are true
\begin{enumerate}[(1)]
\item $V^S=V^{LP}$.
\item Let $(\mu^\star, m^\star)$ be a maximizer of the LP program. Then $m^\star$ satisfies
\begin{enumerate}[(a)]
\item 
$$\int_{\mathcal{S}\times A} \left(f+\frac{\partial g}{\partial t}+\mathcal{L} g\right)(t, x, a) m_t^\star(dx, da) dt = 0,$$
with $\mathcal{S}:=\{(t, x)\in [0, T]\times \mathcal{O}: v(t, x)=g(t, x)\}$.
\item 
$$-\int_{\mathcal{C}\times A} f(t, x, a)m_t^\star(dx, da)dt= \int_{\mathcal{C}\times A} \left(\frac{\partial v}{\partial t}+\mathcal{L} v\right)(t, x, a)m_t^\star(dx, da)dt,$$
where $\mathcal{C}:=([0, T]\times \mathcal{O})\setminus \mathcal{S}$.
\item For all $C^\infty$ functions $\phi$ such that $\operatorname{supp} (\phi)\subset \mathcal{C}$, the following holds
\begin{equation}\label{phi_fp_eq}
\int_{\mathcal{O}} \phi(0, x) m_{0}^{*}(dx)+\int_{0}^{T} \int_{\mathcal{O}\times A}\left(\frac{\partial \phi}{\partial t}+\mathcal{L} \phi\right)(t, x, a) m_t^\star(dx, da) dt = 0.
\end{equation}
\end{enumerate}
Note that (2)(c) holds true if and only if $\mu^\star(\mathcal{C})=0$, which is also equivalent to $\mu^\star(\mathcal{S}\cup ([0, T]\times \partial \mathcal{O}))=1$.
\end{enumerate}
\end{theorem}

\begin{proposition}
Let Assumption \ref{as PDE} hold true, and assume that for each $(t, x)\in [0, T]\times \bar{\mathcal{O}}$, the subset 
$$K(t, x):=\{(b(t, x, a), \sigma^2(t, x, a), z):a\in A, z\leq f(t, x, a)\}$$ 
of $\mathbb R \times\mathbb R_+ \times \mathbb R$ is convex. Let $(\mu^\star, m^\star)$ an LP solution, then, there exists a measurable function $(t, x)\mapsto\alpha^\star(t,x)$ such that $\bar{m}_t(dx) \equiv m^\star_t(dx,A)$ satisfies the following system:
$$
\begin{cases}
\int_{\mathcal{S}} \left(f+\frac{\partial g}{\partial t}+\mathcal{L} g\right)(t, x,\alpha^\star(t,x)) \bar{m}_t(dx) dt = 0,\\
\alpha^\star(t, x)\in \arg\max_{a\in A}\left[\mathcal{L} v(t, x, a) + f(t, x, a)\right] \quad \bar{m}_t(dx) dt-a.e. \text{ on } \mathcal{C}, \nonumber \\
\int_{\mathcal{O}} \phi(0, x) m_{0}^{*}(dx)+\int_{0}^{T} \int_{\mathcal{O}\times A}\left(\frac{\partial \phi}{\partial t}+\mathcal{L} \phi\right)(t, x, \alpha^\star(t, x)) \bar{m}_t(dx) dt = 0, \\ \text{for all } C^\infty \text{ functions } \phi \text{ such that } \operatorname{supp} (\phi)\subset \mathcal{C}.
\end{cases}
$$
\end{proposition}
\begin{proof}
Follows by Theorem \ref{PDE} and a similar argument as in the proof of Proposition \ref{strict cont}.
\end{proof}

\section{MFG problem}

Throughout this section, we let the following assumptions hold true.

\begin{assumption}\label{as_mfg}\leavevmode
\begin{enumerate}[(1)]
\item The functions $b:[0, T]\times \mathbb R\times \mathcal{M}(\bar{\mathcal{O}}\times A)\times A\rightarrow \mathbb R$ and $\sigma:[0, T]\times \mathbb R\times \mathcal{M}(\bar{\mathcal{O}}\times A)\times A\rightarrow \mathbb R_+$ are Lipschitz in $x$ uniformly on $(t, m, a)$.

\item For all $(t, x, z, m, a)\in [0, T]\times \mathbb R\times \bar{\mathcal{O}}\times \mathcal{M}(\bar{\mathcal{O}}\times A)\times A$, 
$$b(t, x, m, a)= \bar b\left(t, x, \int_{\bar{\mathcal{O}} \times A} \hat b (t,y)m(dy,du), a\right),$$
$$\sigma(t, x, m,a)=\bar \sigma\left(t, x, \int_{\bar{\mathcal{O}} \times A} \hat \sigma (t,y)m(dy,du), a\right),$$
$$f(t, z, m, a)= \bar f\left(t, z, \int_{\bar{\mathcal{O}} \times A} \hat f (t,y)m(dy,du), a\right),$$
where $\bar b:[0, T]\times \mathbb R\times \mathbb R^{d}\times A\rightarrow\mathbb R$, $\bar \sigma: [0, T]\times \mathbb R\times \mathbb R^{d}\times A\rightarrow\mathbb R_+$ and $\bar f: [0, T]\times \bar{\mathcal{O}}\times \mathbb R^{d}\times A\rightarrow\mathbb R$, for some $d\in \mathbb N^*$. We assume that $\bar b$, $\bar \sigma$ and $\bar f$ are bounded, measurable and continuous for each fixed $t\in [0, T]$ and that the functions $\hat b:[0, T]\times \bar{\mathcal{O}} \rightarrow \mathbb R^{d}$, $\hat \sigma: [0, T]\times \bar{\mathcal{O}} \rightarrow \mathbb R^{d}$ and $\hat f: [0, T]\times \bar{\mathcal{O}} \rightarrow \mathbb R^{d}$ are continuous and bounded.

\item The function $g:[0, T]\times \bar{\mathcal{O}}\times \mathcal{P}([0, T]\times \bar{\mathcal{O}})\rightarrow \mathbb R$ is such that for all $ (t, x, \mu)\in [0, T]\times \bar{\mathcal{O}}\times \mathcal{P}([0, T]\times \bar{\mathcal{O}})$
$$g(t, x, \mu)=\bar g \left(t, x, \int_{[0, T]\times \bar{\mathcal{O}}}\hat g(s, y)\mu(ds, dy)\right),$$
where $\bar g:[0, T]\times \bar{\mathcal{O}}\times \mathbb R^{d}\rightarrow \mathbb R$ and $\hat g:[0, T]\times \bar{\mathcal{O}}\rightarrow \mathbb R^{d}$ are continuous and bounded.

\item The initial measure $m_0^*$ satisfies $\int_\mathcal{O} |x|^2m_0^*(dx)<\infty$.

\item One of the following statements is true:
\begin{enumerate}
\item The coefficients $b$ and $\sigma$ do not depend on the measure.
\item \textit{Unattainable boundary}: $b$, $\sigma$ and $\mathcal{O}$ are such that, for every filtered probability space $(\Omega, \mathcal{F}, \mathbb F, \mathbb P)$, $\mathbb F$-stopping time $\tau$ such that $\tau\leq T$ $\mathbb P$-a.s., $\mathbb F$-progressive measurable process $\nu$ with values in $\mathcal{P}(A)$, $\mathbb F$-martingale measure $M$ such that $M^\tau$ has intensity $\nu_t(da)\mathds{1}_{t\leq \tau}dt$, $m\in V$ and $\mathbb F$-adapted process $X$ such that
$$dX_t= \int_A b(t, X_t, m_t, a)\nu_t(da)dt + \int_A \sigma(t, X_t, m_t, a)M(dt, da),\quad t\leq \tau, \quad \mathbb P \circ X_0^{-1}= m_0^*,$$
we have
$$\mathbb P \left(\tau_\mathcal{O}^{\tilde X}\geq T \right)=1,$$
where $\tilde X=X_{\cdot\wedge\tau}$.
\item \textit{Attainable boundary}: $\mathcal{O}$ is an open interval, $\sigma$ does not depend on the control $a$ and for all $(t,x, m)\in [0, T]\times \mathbb R\times \mathcal{M}(\bar{\mathcal{O}}\times A)$, $\sigma(t,x, m)\geq c_\sigma$ for some $c_\sigma>0$. 
\end{enumerate}
\end{enumerate}
\end{assumption}

\paragraph{The strong and LP MFG formulations.}

Let us first provide the \textit{strong formulation of the MFG problem}.
\begin{definition}[\textit{Strong formulation of the MFG problem}]\leavevmode
\begin{enumerate}
\item First step: fix $\mu\in \mathcal{P}([0, T]\times \bar{\mathcal{O}})$ and $m\in V$ and find the solution to the mixed control problem
\begin{equation}\label{opti1}
\begin{aligned}
\max_{\tau \in \mathcal{T}, \alpha \in \mathbb{A}} &  \mathbb{E}\left[\int_{0}^{\tau\wedge \tau_\mathcal{O}^{\alpha, m}} f\left(t, X^{\alpha, m}_t, m_t, \alpha_t\right) dt + g\left(\tau\wedge \tau_\mathcal{O}^{\alpha, m}, X^{\alpha,m}_{\tau\wedge \tau_\mathcal{O}^{\alpha, m}}, \mu\right)\right],\\
\text{s.t. } \quad & dX_{t}^{\alpha, m}=b\left(t, X_{t}^{\alpha, m}, m_t, \alpha_t \right) dt+\sigma\left(t, X_{t}^{\alpha, m}, m_t, \alpha_t \right) dW_{t},\\
& X_0^{\alpha, m} \sim m_0^*,
\end{aligned}
\end{equation}
where $\tau_\mathcal{O}^{\alpha, m}=\inf\{t\geq 0:X_t^{\alpha, m}\notin \mathcal{O}\}$.
\item Given the mixed  optimal stopping-control $(\tau^{\mu, m}, \alpha^{\mu, m})$ (solution of the problem (\ref{opti1})) for the agent with initial distribution $m_0^*$ facing a mean-field $(\mu, m)$, the second step consists in finding $\mu\in\mathcal{P}([0, T]\times \bar{\mathcal{O}})$ and the family of distributions $m\in V$  such that
\begin{equation*}
m_{t}(B)=\mathbb{P}\left[(X_{t}^{\alpha^{\mu,m}, m}, \alpha^{\mu,m}_t) \in B, t\leq \tau^{\mu, m}\wedge \tau_{\mathcal{O}}^{\alpha^{\mu, m}, m}\right], \quad B \in \mathcal{B}(\bar{\mathcal{O}} \times A),\quad t \in[0, T],
\end{equation*}
and 
$$\mu=\mathcal{L}\left(\tau^{\mu, m}\wedge \tau_{\mathcal{O}}^{\alpha^{\mu, m}, m}, X^{\alpha^{\mu,m}, m}_{\tau^{\mu, m}\wedge \tau_{\mathcal{O}}^{\alpha^{\mu, m}, m}}\right).$$
\end{enumerate}
\end{definition}

We now give the formulation of the \textit{linear programming MFG problem}. To this end, we first provide a preliminary definition.

\begin{definition}
Let $\mathcal{R}_0$ be the set of pairs $(\mu, m)\in \mathcal{P}([0, T]\times \bar{\mathcal{O}})\times V$, such that for all $u\in C_b^{1, 2}([0, T]\times \bar{\mathcal{O}})$,
\begin{align*}
\int_{[0, T]\times \bar{\mathcal{O}}} u(t, x)\mu(dt, dx)&\leq \int_\mathcal{O} u(0, x)m_0^*(dx) + \int_0^T \int_{\bar{\mathcal{O}}\times A} \frac{\partial u}{\partial t} (t, x)m_t(dx, da)dt \\
&\quad + C(u) \int_0^T m_t(\bar{\mathcal{O}}\times A)dt,
\end{align*}
where $C(u)$ is the supremum of $|\mathcal{L}u|$ over $[0, T]\times \bar{\mathcal{O}}\times \mathcal{M}(\bar{\mathcal{O}}\times A)\times A$, with
$$\mathcal{L}u(t, x, m, a)=b(t, x, m, a)\partial_x u(t, x) + \frac{\sigma^2}{2}(t, x, m, a)\partial_{xx} u(t, x).$$
\end{definition}

\begin{definition}[\textit{LP formulation of the MFG problem}]
Fix a pair $(\bar \mu, \bar m)\in \mathcal{P}([0, T]\times \bar{\mathcal{O}})\times V$ and define $\mathcal{R}[\bar m]$ as the set of pairs $(\mu, m)\in \mathcal{P}([0, T]\times \bar{\mathcal{O}})\times V$, such that for all $u\in C_b^{1, 2}([0, T]\times \bar{\mathcal{O}})$,
$$\int_{[0, T]\times \bar{\mathcal{O}}} u(t, x)\mu(dt, dx)= \int_\mathcal{O} u(0, x)m_0^*(dx) + \int_0^T \int_{\bar{\mathcal{O}}\times A} \left(\frac{\partial u}{\partial t} +\mathcal L u\right) (t, x, \bar m_t, a)m_t(dx, da)dt.$$
Let $\Gamma[\bar \mu, \bar m]: \mathcal{R}_0\rightarrow \mathbb R$ be defined as 
$$\Gamma[\bar \mu, \bar m] (\mu, m)= \int_0^T \int_{\bar{\mathcal{O}}\times A} f(t, x, \bar m_t, a) m_t(dx, da)dt + \int_{[0, T]\times \bar{\mathcal{O}}} g(t, x, \bar \mu) \mu (dt, dx).$$
We say that $(\mu^\star, m^\star)\in \mathcal{P}([0, T]\times \bar{\mathcal{O}})\times V$ is an LP MFG Nash equilibrium if $(\mu^\star, m^\star)\in \mathcal{R}[m^\star]$ and for all $(\mu, m)\in \mathcal{R}[m^\star]$,
$$\Gamma[\mu^\star, m^\star] (\mu, m)\leq \Gamma[\mu^\star, m^\star] (\mu^\star, m^\star).$$
The real number $\Gamma[\mu^\star, m^\star] (\mu^\star, m^\star)$ is called Nash value.
\end{definition}

\begin{remark}
Note that for all $\bar m\in V$, $\mathcal{R}[\bar m]$ has the same structure as $\mathcal{R}$ of the previous section, thus it satisfies the same properties. Moreover, the set $\mathcal{R}_0$ has been introduced in order to be able to apply the fixed point arguments specific to the MFG setting; more precisely, it satisfies all properties as the set $\mathcal{R}$ (see theorem below) and contains all the sets $\mathcal{R}[m]$ for $m\in V$.
\end{remark}

\begin{theorem}[\textit{Properties of the set $\mathcal{R}_0$}]\label{R0}
The set $\mathcal{R}_0$ is compact, convex, nonempty, contains the set $\mathcal{R}[m]$ for all $m\in V$, and Lemmas \ref{lemma1}, \ref{lemma2bis} and \ref{lemma3} are still valid if one replaces $\mathcal{R}$ with $\mathcal{R}_0$.
\end{theorem}

\begin{proof}
The same proofs of Section \ref{sec 2.1} can be applied. 
\end{proof}

\begin{definition}
Define the set valued mapping $\mathcal{R}^\star:\mathcal{R}_0 \rightarrow 2^{\mathcal{R}_0}$ as
$$\mathcal{R}^\star(\bar \mu, \bar m)= \mathcal{R}[\bar m].$$
Define $\Theta:\mathcal{R}_0 \rightarrow 2^{\mathcal{R}_0}$ as 
$$\Theta(\bar\mu,\bar m)=\underset{( \mu,  m) \in \mathcal{R}^\star(\bar \mu, \bar m)}{\arg \max } \Gamma [\bar \mu, \bar m]( \mu,  m).$$
\end{definition}

\begin{remark}
Note that the set of LP MFG Nash equilibria coincides with the set of fixed points of $\Theta$.
\end{remark}

\subsection{Existence of LP MFG Nash equilibria}

We shall first provide some convergence results, which will be useful in the proof of existence of LP MFG Nash equilibria.

\begin{lemma}\label{lemma gronwall}
Let $(\Omega, \mathcal{F}, \mathbb{F}, \mathbb{P})$ be a filtered probability space. Let $\tau$ be a bounded $\mathbb{F}$-stopping time and let $M$ be an $\mathbb{F}$-martingale measure with intensity $q_t(da)\mathds{1}_{t\leq \tau}dt$, where $(q_t)_{t\in [0, T]}$ is an $\mathbb{F}$-predictable process with values in $\mathcal{P}(A)$. Consider $(\bar \mu^n, \bar m^n)_{n\geq 1}\subset \mathcal{R}_0$ such that $\bar m^n\rightarrow \bar m$ in $V$ and let $X$ and $(X^n)_{n\geq 1}$ be $\mathbb F$-adapted processes satisfying, 
$$dX_t= \int_A b(t, X_t, \bar m_t, a)q_t(da)dt + \int_A\sigma (t, X_t, \bar m_t, a)M(dt, da),\quad t \leq \tau, \quad X_0\sim m_0^*.$$
$$dX_t^n=\int_Ab(t, X_t^n, \bar m_t^n, a)q_t(da)dt + \int_A\sigma (t, X_t^n, \bar m_t^n, a)M(dt, da),\quad t \leq \tau, \quad X_0^n= X_0.$$
Then, up to a subsequence,
$$\mathbb E^{\mathbb P} \left[\sup_{t\leq T}|X_{t\wedge\tau}^n - X_{t\wedge\tau}|^2\right]\underset{n\rightarrow\infty}{\longrightarrow} 0.$$
\end{lemma}
\begin{proof}
We will denote by $C>0$ any constant independent from $n$. To simplify the formulas, in this proof we shall use the following shorthand notation: $b_n(t, x, a):=b(t,x, \bar m^n_t, a)$, $b_0(t, x, a):=b(t, x, \bar m_t, a)$, $\sigma_n(t, x, a):=\sigma(t,x, \bar m^n_t, a)$ and $\sigma_0(t, x, a):=\sigma(t, x, \bar m_t, a)$. Let $0\leq s\leq t\leq T$. We have
\begin{align*}
|X_{s\wedge \tau}^n-X_{s\wedge \tau}|^2 &\leq C \left[ \left(\int_0^{s\wedge \tau} \int_A(b_n(r, X^n_r, a)-b_0(r, X_r, a))q_r(da)dr \right)^2\right. \\
&\quad \left.+ \left(\int_0^{s\wedge \tau} \int_A (\sigma_n(r, X^n_r, a)-\sigma_0(r, X_r, a))M(dr, da) \right)^2 \right].
\end{align*}
Using Burkholder-Davis-Gundy inequality, we get
\begin{align*}
&\mathbb E^{\mathbb P}\left[\sup_{s\leq t}\left(\int_0^{s\wedge \tau}\int_A (\sigma_n(r, X^n_r, a)-\sigma_0(r, X_r, a))M(dr, da) \right)^2 \right]\\
&\leq C \mathbb E^{\mathbb P}\left[ \int_0^{t\wedge \tau}\int_A(\sigma_n(r, X^n_r, a)-\sigma_0(r, X_r, a))^2q_r(da)dr \right].
\end{align*}
Define 
$$g_n(t)= \mathbb E^{\mathbb P} \left[ \sup_{s\leq t}|X_{s\wedge \tau}^n -X_{s\wedge \tau}|^2 \right].$$
From the above estimates,
\begin{align*}
g_n(t) &\leq C\mathbb E^{\mathbb P} \left[ \int_0^{t\wedge \tau}\int_A (b_n(r, X^n_r, a)-b_0(r, X_r, a))^2q_r(da)dr\right. \\
&\quad \left.+ \int_0^{t\wedge \tau}\int_A(\sigma_n(r, X^n_r, a)-\sigma_0(r, X_r, a))^2q_r(da)dr\right].
\end{align*}
Now, by the Lipschitz assumption on $b$,
\begin{align*}
&\int_0^{t\wedge \tau} \int_A (b_n(r, X^n_r, a)-b_0(r, X_r, a))^2q_r(da)dr \\
&=  \int_0^{t\wedge \tau} \int_A(b_n(r, X^n_r, a)- b_n(r, X_r, a) + b_n(r, X_r, a)- b_0(r, X_r, a))^2q_r(da)dr\\
&\leq C \left[ \int_0^{t\wedge \tau}\int_A (b_n(r, X^n_r, a)- b_n(r, X_r, a))^2 q_r(da)dr \right.\\
&\quad \left. + \int_0^{t\wedge \tau}\int_A (b_n(r, X_r, a)- b_0(r, X_r, a))^2q_r(da)dr\right]\\
&\leq C \left[ \int_0^{t} \sup_{r\leq s} |X_{r\wedge \tau}^n-X_{r\wedge \tau}|^2 ds +  \int_0^{t\wedge \tau}\int_A (b_n(r, X_r, a)- b_0(r, X_r, a))^2q_r(da)dr\right].
\end{align*}
Similarly,
\begin{align*}
&\int_0^{t\wedge \tau} \int_A (\sigma_n(r, X^n_r, a)-\sigma_0(r, X_r, a))^2q_r(da)dr \\
&\leq C \left[ \int_0^{t} \sup_{r\leq s} |X_{r\wedge \tau}^n-X_{r\wedge \tau}|^2 ds +  \int_0^{t\wedge \tau}\int_A (\sigma_n(r, X_r, a)- \sigma_0(r, X_r, a))^2q_r(da)dr\right].
\end{align*}
We get finally,
$$g_n(t)\leq C \left( \int_0^t g_n(s)ds + B_n + S_n \right),$$
where 
$$B_n :=\mathbb E^{\mathbb P} \left[ \int_0^T\int_A (b_n(r, X_r, a)- b_0(r, X_r, a))^2q_r(da)dr \right],$$
$$S_n :=\mathbb E^{\mathbb P} \left[ \int_0^T\int_A (\sigma_n(r, X_r, a)- \sigma_0(r, X_r, a))^2q_r(da)dr \right].$$
By Gronwall's inequality,
\begin{equation}\label{gronwall}
g_n(T)\leq C \left(B_n + S_n\right)e^{CT}.
\end{equation}
Let us show that $B_n\rightarrow 0$ as $n\rightarrow 0$. We fix $\omega\in \Omega$. We are going to use Lemma \ref{slutsky_stable} for this fixed $\omega$ and then use dominated convergence for the expectation. We set $\Theta=[0, T]$, $\mathcal{X}=A$, $\eta(dr)=dr$,
$$\psi^n(r)=\int_{\bar O \times A}\hat b(r, y)\bar m^n_r(dy, du),\quad \psi(r)=\int_{\bar O \times A}\hat b(r, y)\bar m_r(dy, du),$$
$\nu^n_r(da)=q_r(\omega)(da),$
$$\varphi(r, a, y)=[\bar b (r, X_r(\omega), y, a)-\bar b (r, X_r(\omega), \psi(r), a)]^2.$$
By Theorem \ref{R0} and Lemma \ref{lemma2bis}, $\psi^n$ converges to $\psi$ in $L^1([0, T];\mathbb R^{d})$. Since the hypothesis of Lemma \ref{slutsky_stable} are satisfied, we get for all $\omega\in \Omega$,
$$I_n(\omega)=\int_0^T\int_A (b_n(r, X_r(\omega), a)- b_0(r, X_r(\omega), a))^2q_r(\omega)(da)dr\underset{n\rightarrow\infty}{\longrightarrow}0.$$
Since $\bar b$ is bounded and $q_r$ are probabilities, we get by the dominated convergence theorem $B_n \underset{n\rightarrow\infty}{\longrightarrow}0$. The convergence of $S_n$ to $0$ follows by the same arguments. Taking $n\rightarrow\infty$ in \eqref{gronwall} we get the result.
\end{proof}

We now prove the continuity of the set $\mathcal{R}^\star$ in the sense of set-valued mappings.
\begin{proposition}[\textit{Continuity of $\mathcal{R}^\star$}]\label{Rcont}
The set-valued mapping $\mathcal{R}^\star$ is continuous (in the sense of Definition \ref{Defcont}).
\end{proposition}

\begin{proof}
\textit{Step 1.} We first prove  \textit{the upper hemicontinuity} (in the sense of Definition \ref{def up}). By the Closed Graph Theorem (see Theorem \ref{upperhemi_closedgraph}), it suffices to show that $\mathcal{R}^*$ has closed graph. Let $(\mu^n, m^n)\in \mathcal{R}^\star(\bar\mu^n, \bar m^n)=\mathcal{R}[\bar m^n]$ such that $\mu^n\rightharpoonup \mu$, $m^n\rightarrow m$ in $V$, $\bar \mu^n\rightharpoonup \bar \mu$ and $\bar m^n\rightarrow \bar m$ in $V$. For all $n\geq 1$ and $u\in C_b^{1, 2}([0, T]\times \bar{\mathcal{O}})$,
\begin{align*}
\int_{[0, T]\times\bar{\mathcal{O}}} u(t, x)\mu^n(dt, dx) &= \int_\mathcal{O} u(0, x)m_0^*(dx)  + \int_0^T \int_{\bar{\mathcal{O}} \times A} \left(\frac{\partial u}{\partial t} +\mathcal L u\right) (t, x, \bar m_t^n, a)m_t^n(dx, da)dt.
\end{align*}
By Theorem \ref{R0} and Lemma \ref{lemma3}, we get the stable convergence of $m^n_t(dx, da)dt$ to $m_t(dx, da)dt$. In particular, 
$$\int_0^T \int_{\bar{\mathcal{O}} \times A} \frac{\partial u}{\partial t} (t, x)m_t^n(dx, da)dt \underset{n\rightarrow\infty}{\longrightarrow} \int_0^T \int_{\bar{\mathcal{O}} \times A} \frac{\partial u}{\partial t} (t, x)m_t(dx, da)dt.$$
By Theorem \ref{R0} and Lemma \ref{lemma2bis}, 
$$\psi^n(t)=\int_{\bar O \times A}\hat b(t, y)\bar m^n_t(dy, du)\underset{n\rightarrow\infty}{\longrightarrow} \psi(t)=\int_{\bar O \times A}\hat b(t, y)\bar m_t(dy, du)$$
in $L^1([0, T];\mathbb R^{d})$. We conclude by Lemma \ref{slutsky_stable} that
$$\int_0^T \int_{\bar{\mathcal{O}} \times A} \left(\frac{\partial u}{\partial x}b\right)(t, x, \bar m_t^n, a) m_t^n(dx, da)dt  \underset{n\rightarrow\infty}{\longrightarrow} \int_0^T \int_{\bar{\mathcal{O}} \times A} \left(\frac{\partial u}{\partial x}b\right)(t, x, \bar m_t, a) m_t(dx, da)dt.$$
By the same argument,
$$\int_0^T \int_{\bar{\mathcal{O}} \times A} \left(\frac{\partial^2 u}{\partial x^2}\frac{\sigma^2}{2}\right)(t, x, \bar m_t^n, a) m_t^n(dx, da)dt  \underset{n\rightarrow\infty}{\longrightarrow} \int_0^T \int_{\bar{\mathcal{O}} \times A} \left(\frac{\partial^2 u}{\partial x^2}\frac{\sigma^2}{2}\right)(t, x, \bar m_t, a) m_t(dx, da)dt.$$
The above results, together with the convergence  $$\int_{[0, T]\times \bar{\mathcal{O}}} u(t, x)\mu^n(dt, dx) \underset{n\rightarrow\infty}{\longrightarrow} \int_{[0, T]\times \bar{\mathcal{O}}} u(t, x)\mu(dt, dx),$$ 
lead to
$$\int_{[0, T]\times \bar{\mathcal{O}}} u(t, x)\mu(dt, dx)= \int_\mathcal{O} u(0, x)m_0^*(dx) + \int_0^T \int_{\bar{\mathcal{O}} \times A} \left(\frac{\partial u}{\partial t} +\mathcal L u\right) (t, x, \bar m_t, a)m_t(dx, da)dt,$$
which means that $(\mu, m)\in \mathcal{R}[\bar m]=\mathcal{R}^\star(\bar \mu, \bar m)$.\\

\textit{Step 2.} We now prove \textit{the lower hemicontinuity } (in the sense of Definition \ref{def low}). Consider a sequence $(\bar \mu^n, \bar m^n)_{n\geq 1}\subset \mathcal{R}_0$ such that $(\bar \mu^n, \bar m^n)\rightarrow (\bar \mu, \bar m)$ and let $(\mu, m)\in \mathcal{R}^\star(\bar \mu, \bar m) = \mathcal{R}[\bar m]$. We need to prove that up to a subsequence, we can find $(\mu^n, m^n)_{n\geq 1}\subset\mathcal{R}_0$ such that $(\mu^n, m^n)\in \mathcal{R}^\star(\bar \mu^n, \bar m^n) = \mathcal{R}[\bar m^n]$ and $(\mu^n, m^n)\rightarrow (\mu, m)$. This result is trivial if Assumption \ref{as_mfg} (5)(a) holds true, therefore consider in the sequel the cases (5)(b) or (5)(c). Let $\nu_{t, x}(da)$ be such that 
$$m_t(dx, da)dt=\nu_{t, x}(da)m_t(dx, A)dt.$$
By Theorem \ref{SDE rep}, there exists a filtered probability space $(\Omega, \mathcal{F}, \mathbb F, \mathbb P)$, an $\mathbb F$-adapted process $X$, an $\mathbb F$-stopping time $\tau$ such that $\tau\leq T\wedge\tau_\mathcal{O}^X$ $\mathbb P$-a.s., an $\mathbb F$-martingale measure $M$ with intensity $\nu_{t, X_t}(da)\mathds{1}_{t\leq \tau}dt$, such that
$$X_{t\wedge \tau}= \int_0^{t\wedge \tau}\int_A b(t, X_t, \bar m_t, a)\nu_{t, X_t}(da)dt + \int_0^{t\wedge \tau}\int_A \sigma(t, X_t,\bar m_t, a)M(dt, da), \quad \mathbb P \circ X_0^{-1}= m_0^*,$$
$$\mu =\mathbb P \circ (\tau, X_\tau)^{-1},$$
$$m_t(B\times C)= \mathbb E^{\mathbb P}\left[ \mathds{1}_B(X_t)\nu_{t, X_t}(C) \mathds{1}_{t\leq \tau}\right],  \quad B\in \mathcal{B}(\bar{\mathcal{O}}), \quad C\in \mathcal{B}(A), \quad t-a.e.$$
On the same filtered probability space, define 
$$m_t^n(B \times C):=\mathbb E^\mathbb P\left[\mathds 1_B(X_t^n)  \nu_{t,X_t}(C) \mathds 1_{t\leq \tau\wedge \tau_{\mathcal{O}}^{X^n}}\right], \quad \mu^n:=\mathbb P \circ\left(\tau\wedge \tau_{\mathcal{O}}^{X^n}, X_{\tau\wedge \tau_{\mathcal{O}}^{X^n}}^n\right)^{-1},$$
where $X^n$ denotes the unique strong solution of 
$$dX_t^n=\int_A b(t, X_t^n, \bar m_t^n, a)\nu_{t,X_t}(da)dt + \int_A \sigma (t, X_t^n, \bar m_t^n, a)M(dt,da), \quad X_0^n= X_0.$$
Note that existence and uniqueness follow by the Lipschitz and boundedness condition on the coefficients and the square integrability of $m_0^*$. We have that $(\mu^n, m^n)\in \mathcal{R}[\bar m^n]=\mathcal{R}^\star(\bar \mu^n, \bar m^n)$ by a similar argument as in Proposition \ref{SDE to measures}. Let us now prove that $m^n\rightarrow m$ in $V$. By Remark 8.3.1 and Exercise 8.10.71 in \cite{bogachev2007} (Volume 2), it is sufficient to use bounded and Lipschitz functions as test functions. Consider a bounded and Lipschitz function $\phi:[0, T]\times \bar{\mathcal{O}}\times A\rightarrow \mathbb R$ and denote by $C$ the maximum between $\|\phi\|_\infty$ and the Lipschitz constant of $\phi$. Compute
\begin{align*} 
& \left| \int_0^T \int_{\bar{\mathcal{O}} \times A} \phi(t,x,a)m_t(dx,da)dt - \int_0^T \int_{\bar{\mathcal{O}} \times A} \phi(t,x,a)m^n_t(dx,da)dt \right|  \\   
&=\left| \mathbb{E}^\mathbb P \left[ \int_0^{\tau} \int_{A} \phi(t,X_t,a)\nu_{t, X_t}(da)dt  - \int_0^{\tau\wedge \tau_{\mathcal{O}}^{X^n}} \int_{A} \phi(t,X_t^n,a)\nu_{t, X_t}(da)dt\right]\right|  \\
&\leq \left| \mathbb{E}^\mathbb P \left[ \int_0^{\tau\wedge \tau_{\mathcal{O}}^{X^n}} \int_{A} (\phi(t,X_t,a)-\phi(t, X_t^n, a))\nu_{t, X_t}(da)dt  \right]\right|  \\
&\quad + \left| \mathbb{E}^\mathbb P \left[ \int_0^{\tau} \int_{A} \phi(t,X_t,a)\nu_{t, X_t}(da)dt  - \int_0^{\tau\wedge \tau_{\mathcal{O}}^{X^n}} \int_{A} \phi(t,X_t,a)\nu_{t, X_t}(da)dt\right]\right|  \\
&\leq C T^{\frac{1}{2}}\left(\mathbb{E}^\mathbb P\left[\sup_{t \leq \tau} |X_t-X_t^n|^2\right]\right)^{\frac{1}{2}} + C \left( \mathbb E^\mathbb P \left[\tau\right]- \mathbb E^\mathbb P \left[\tau\wedge \tau_{\mathcal{O}}^{X^n}\right] \right),\\
\end{align*}
Now, by Lemma \ref{lemma gronwall}, we get the convergence of the first term. The convergence of the second one is trivial under the condition (5)(b) of Assumption \ref{as_mfg}. Suppose now condition (5)(c) of Assumption \ref{as_mfg} holds. Then, by Theorem \ref{SDE rep}, the martingale measure $M$ is replaced by an $\mathbb F$-Brownian motion $W$ and we get
$$X_{t\wedge \tau}= X_0+ \int_0^{t\wedge \tau}\int_A b(t, X_t,\bar m_t, a)\nu_{t, X_t}(da)dt + \int_0^{t\wedge \tau}\sigma(t, X_t, \bar m_t)dW_t,$$
$$X_t^n=X_0+\int_0^t\int_A b(t, X_t^n, \bar m_t^n, a)\nu_{t,X_t}(da)dt + \int_0^t \sigma (t, X_t^n, \bar m_t^n)dW_t.$$
Define $X^0$ as the unique strong solution to
$$X_t^0=X_0+\int_0^t\int_A b(t, X_t^0, \bar m_t, a)\nu_{t,X_t}(da)dt + \int_0^t \sigma (t, X_t^0, \bar m_t)dW_t.$$
By pathwise uniqueness type arguments, we get that $X^0_t=X_t$ on ${t\leq \tau}$, which implies that $\tau_{\mathcal{O}}^{X^0}\geq \tau$ $\mathbb P$-a.s. We have that for all $\delta>0$ and $C>0$, there exists $n_0\geq 1$ such that for all $n\geq n_0$,
$$\mathbb P \left( \sup_{t\in [0, T]} |X_t^n- X_t^0|\geq C \right)<\delta.$$
We have also that, for all $\delta>0$, there exists $R>0$ such that,
$$\mathbb P \left( \sup_{t\in [0, T]} |X_t^0|\geq R \right)<\delta,$$
Using these two last properties, we get by Theorem 5.1 and Remark 5.4 in \cite{shevchenko2015} that $\tau_{\mathcal{O}}^{X^n} \wedge T \overset{\mathbb P}{\underset{n\rightarrow\infty}{\longrightarrow}}\tau_{\mathcal{O}}^{X^0}\wedge T$. To be more precise, by assumption, $\mathcal{O}=]c_1, c_2[$, $c_1<c_2$, then one can choose for the assumptions in \cite{shevchenko2015} the function
$$\varphi(t, x)=(T-t)(x-c_1)(x-c_2).$$ 
Therefore, we get $(\tau, \tau_{\mathcal{O}}^{X^n}\wedge T)\overset{\mathbb P}{\underset{n\rightarrow\infty}{\longrightarrow}}(\tau,\tau_{\mathcal{O}}^{X^0}\wedge T)$ and by the continuous mapping theorem, 
$$\tau\wedge \tau_{\mathcal{O}}^{X^n}=\tau\wedge \tau_{\mathcal{O}}^{X^n}\wedge T\overset{\mathbb P}{\underset{n\rightarrow\infty}{\longrightarrow}}\tau\wedge \tau_{\mathcal{O}}^{X^0}\wedge T =\tau .$$
Since this sequence is uniformly bounded by $T$ we get the convergence in $L^1$. Finally, we can conclude that $m^n\rightarrow m$ in $V$. Now, by the convergence of $m^n$ towards $m$ in $V$ and since $(\mu^n, m^n) \in \mathcal{R}[\bar{m}^n]$, we get that, $\mu^n\rightharpoonup \mu$ (using the same results as for the upper hemicontinuity).
\end{proof}

We now prove an existence result of LP Nash equilibria.

\begin{theorem}[\textit{Existence of LP MFG equilibria}]\label{Nash_exist_LP}
The set of LP MFG equilibria is compact and nonempty.
\end{theorem}
\begin{proof}
The proof is based on Kakutani-Fan-Glicksberg's fixed point theorem for set-valued maps (Theorem \ref{KFG}). Note that the space $\mathcal{R}_0$ is a subset of the locally convex Hausdorff space $\mathcal{M}^s([0, T]\times \bar{\mathcal{O}}) \times V_1$, where $\mathcal{M}^s([0, T]\times \bar{\mathcal{O}})$ is the set of Borel finite signed measures on $[0, T]\times \bar{\mathcal{O}}$. Moreover, $\mathcal{R}_0$ is nonempty, compact and convex (see Theorem \ref{R0}). Remark also that the map $\Theta$ has convex values. Let us show that it has closed graph and nonempty values. To this end, we apply Berge's Maximum Theorem (Theorem \ref{berge}), for which we need the previous result we have shown (Proposition \ref{Rcont}), and the Closed Graph Theorem (Theorem \ref{upperhemi_closedgraph}). Therefore, it only remains to show that 
$$((\bar \mu, \bar m), (\mu, m))\in \operatorname{Gr}(\mathcal{R}^\star)\mapsto \Gamma[\bar \mu, \bar m](\mu, m)$$
is continuous. Let $((\bar \mu^n, \bar m^n), (\mu^n, m^n))_{n\geq 1}\subset \operatorname{Gr}(\mathcal{R}^\star)$ converging to $((\bar \mu, \bar m), (\mu, m))\in \operatorname{Gr}(\mathcal{R}^\star)$, that is $m^n\rightarrow m$ in $V$, $\bar m^n\rightarrow \bar m$ in $V$, $\mu^n \rightharpoonup \mu$ and $\bar \mu^n \rightharpoonup \bar \mu$. Using the same arguments as in Proposition \ref{Rcont} (upper hemicontinuity), we get
$$\int_0^T \int_{\bar{\mathcal{O}}\times A} f(t, x, \bar m_t^n, a) m_t^n(dx, da)dt \underset{n\rightarrow\infty}{\longrightarrow} \int_0^T \int_{\bar{\mathcal{O}}\times A} f(t, x, \bar m_t, a) m_t(dx, da)dt.$$
By Lemma \ref{slutsky}, 
$$\int_{[0, T]\times \bar{\mathcal{O}}} g(t, x, \bar \mu^n) \mu^n (dt, dx) \underset{n\rightarrow\infty}{\longrightarrow}\int_{[0, T]\times \bar{\mathcal{O}}} g(t, x, \bar \mu) \mu (dt, dx).$$
We conclude that 
$$\Gamma[\bar \mu^n, \bar m^n] (\mu^n, m^n) \underset{n\rightarrow\infty}{\longrightarrow} \Gamma[\bar \mu, \bar m] (\mu, m),$$
which shows the continuity.
\end{proof}

\subsection{Nash value and selection of Nash equilibria}

\paragraph{Case of measure independent coefficients} In the case where the coefficients $b$ and $\sigma$ do not depend on the measure, we  can prove uniqueness of the Nash value, which holds under the well known anti-monotonicity conditions on $f$ and $g$.

When the coefficients do not depend on the measure, an LP Nash equilibrium is a pair $(\mu^\star, m^\star) \in \mathcal{R}$ such that for all $(\mu, m) \in \mathcal{R}$,
\begin{equation*}
\Gamma[\mu^\star, m^\star](\mu, m)\leq \Gamma[\mu^\star, m^\star](\mu^\star, m^\star).
\end{equation*}

\begin{theorem}[\textit{Uniqueness of the Nash value}]
Suppose that the coefficients do not depend on the measure. Suppose also that $f$ and $g$ take the following form
$$f(t, x, m, a) = f_1(t, x, a) f_2\left(t, \int_{\bar{\mathcal{O}}\times A}f_1(t, y, u)m(dy, du)\right) + f_3(t, x, a)$$
$$g(t, x, \mu) = g_1(t, x) g_2\left(\int_{[0, T]\times \bar{\mathcal{O}}}g_1(s, y)\mu(ds, dy)\right) + g_3(t, x),$$
where $f_1$, $f_2$, $f_3$, $g_1$, $g_2$, $g_3$ are bounded and measurable, $f_2$ is non-increasing in the second argument and $g_2$ is non-increasing. Let $(\mu^1, m^1)$ and $(\mu^2, m^2)$ be two LP Nash equilibria. Then,
\begin{equation*}
f_2\left(t, \int_{\bar{\mathcal{O}}\times A}f_1(t, y, u)m^1_t(dy, du)\right)= f_2\left(t, \int_{\bar{\mathcal{O}}\times A}f_1(t, y, u)m^2_t(dy, du)\right),
\end{equation*}
almost everywhere on $[0,T]$, and 
$$g_2\left(\int_{[0, T]\times \bar{\mathcal{O}}}g_1(s, y)\mu^1(ds, dy)\right)= g_2\left(\int_{[0, T]\times \bar{\mathcal{O}}}g_1(s, y)\mu^2(ds, dy)\right).$$
In particular they lead to the same Nash value, that is 
\begin{equation*}
\Gamma[\mu^1, m^1](\mu^1, m^1)=\Gamma[\mu^2, m^2](\mu^2, m^2).
\end{equation*}
\end{theorem}

\begin{proof}
The proof is a slight modification of the one of Theorem 4.4 in \cite{bdt2020}, therefore we omit it.
\end{proof}

\paragraph{Case of measure dependent coefficients}

When the coefficients depend on the measure, we do not prove the uniqueness of the Nash value, but instead we can show that there exists a maximal Nash value. Let $\mathcal{N}^\star$ be the set of Nash equilibria.

\begin{proposition}\label{maxNash}
There exists $(\mu^\star, m^\star)\in\mathcal{N}^\star$ such that for all $(\mu, m)\in \mathcal{N}^\star$,
$$\Gamma[\mu, m](\mu, m) \leq \Gamma[\mu^\star, m^\star](\mu^\star, m^\star)$$
\end{proposition}
\begin{proof}
By Theorem \ref{Nash_exist_LP}, the set $\mathcal{N}^\star$ is compact and nonempty. Consider the functional $v:\mathcal{N}^\star\rightarrow \mathbb R$ defined by
$$v(\mu, m)=\Gamma[\mu, m](\mu, m).$$
As in Theorem \ref{Nash_exist_LP}, we can show that $v$ is continuous. By compactness of $\mathcal{N}^\star$ and continuity of $v$, we conclude the existence of a maximizer.
\end{proof}

\paragraph{Selection of equilibria} In both cases we have not proved uniqueness of Nash equilibria, we study only the Nash value. The natural question arising in this context is how to select the equilibria. In \cite{delarue2018} the authors propose several ways of choosing equilibria in a particular model of MFGs, one of them is to choose the equilibria by maximizing the Nash value. We have shown in Proposition \ref{maxNash} that this method is always possible under our assumptions.

\subsection{Relation with \textit{MFG equilibria in the weak formulation}}

In this section we show the equivalence between linear programming MFG equilibria and MFGs in the weak formulation as defined below.

{
\begin{definition}[\textit{Weak MFG solution with strict optimal stopping/control}]
For $(\mu, m)\in \mathcal{P}([0, T]\times \bar{\mathcal{O}})\times V$, define $\mathcal{U}^W[\mu, m]$ as the set of tuples $U=(\Omega, \mathcal F, \mathbb F, \mathbb P, W, \alpha, \tau, X)$ such that $(\Omega, \mathcal{F}, \mathbb F, \mathbb P)$ is a filtered probability space, $\tau$ is an $\mathbb F$-stopping time such that $\tau\leq T$ $\mathbb P$-a.s., $\alpha$ is an $\mathbb{F}$-progressively measurable process with values in $A$, $W$ is an $\mathbb F$-Brownian motion, $X$ is an $\mathbb F$-adapted process such that
$$dX_t= b(t, X_t, m_t, \alpha_t)dt + \sigma(t, X_t, m_t, \alpha_t)dW_t,\quad t\leq \tau, \quad \mathbb P \circ X_0^{-1}= m_0^*.$$
Let $\mathcal{H}^W[\mu, m]:\mathcal{U}^W[\mu, m]\rightarrow \mathbb R$ defined by
$$\mathcal{H}^W[\mu, m](U)=\mathbb{E}^{\mathbb P}\left[\int_{0}^{\tau \wedge \tau_{\mathcal{O}}^{X}} f\left(t, X_{t}, m_t, \alpha_t \right)dt  + g\left(\tau\wedge \tau_{\mathcal{O}}^{X}, X_{\tau\wedge \tau_{\mathcal{O}}^{X}}, \mu\right)\right]$$
for all $U=(\Omega, \mathcal F, \mathbb F, \mathbb P, W, \alpha, \tau, X)\in\mathcal{U}^W[\mu, m]$. The value of the optimization problem in the weak formulation with strict optimal stopping/control associated to $(\mu, m)$ is defined by 
\begin{equation}\label{value weak strict 2}
V^W[\mu, m] := \sup_{U\in \mathcal{U}^W[\mu, m]} \mathcal{H}^W[\mu, m](U).    
\end{equation}
Moreover, we say that $U^\star=(\Omega, \mathcal F, \mathbb F, \mathbb P, W, \alpha, \tau, X)$ is a \textit{weak MFG Nash equilibrium with strict control} if $U^\star\in \mathcal{U}^W[\mu^\star, m^\star]$, where
\begin{align}\label{c1}
   m_{t}^\star(B \times C)=\mathbb{E}^\mathbb P\left[\mathds{1}_B(X_{t}) \mathds{1}_C(\alpha_t) \textbf{1}_{t \leq \tau\wedge \tau_{\mathcal{O}}^X}\right], \quad B \in \mathcal{B}(\bar{\mathcal{O}}), \quad C \in \mathcal{B}(A), \quad t \in[0, T], 
\end{align}
\begin{align}\label{c2}
\mu^\star=\mathbb P \circ \left(\tau\wedge \tau_{\mathcal{O}}^X, X_{\tau\wedge \tau_{\mathcal{O}}^X}\right)^{-1},
\end{align}
and 
\begin{align}\label{c3}
\mathcal{H}^{W}[\mu^\star, m^\star](U^\star)=V^{W}[\mu^\star, m^\star].
\end{align}
\end{definition}

\begin{definition}[\textit{Weak MFG solution with strict optimal stopping and relaxed control}]
For $(\mu, m)\in \mathcal{P}([0, T]\times \bar{\mathcal{O}})\times V$, define $\mathcal{U}^R[\mu, m]$ as the set of tuples $U=(\Omega, \mathcal F, \mathbb F, \mathbb P, M, \nu, \tau, X)$ such that $(\Omega, \mathcal{F}, \mathbb F, \mathbb P)$ is a filtered probability space, $\tau$ is an $\mathbb F$-stopping time such that $\tau\leq T$ $\mathbb P$-a.s., $\nu$ is an $\mathbb{F}$-progressively measurable process with values in $\mathcal{P}(A)$, $M$ is a continuous $\mathbb F$-martingale measure such that $M^\tau$ has intensity $\nu_t(da)\mathds{1}_{t\leq\tau}dt$, $X$ is an $\mathbb F$-adapted process such that
$$dX_t= \int_A b(t, X_t, m_t, a)\nu_t(da)dt + \int_A \sigma(t, X_t, m_t, a)M(dt, da),\quad t\leq \tau, \quad \mathbb P \circ X_0^{-1}= m_0^*.$$
Let $\mathcal{H}^R[\mu, m]:\mathcal{U}^R[\mu, m]\rightarrow \mathbb R$ defined by
$$\mathcal{H}^R[\mu, m](U)=\mathbb{E}^{\mathbb P}\left[\int_{0}^{\tau \wedge \tau_{\mathcal{O}}^{X}} \int_A f\left(t, X_{t}, m_t, a \right) \nu_t(da)dt  + g\left(\tau\wedge \tau_{\mathcal{O}}^{X}, X_{\tau\wedge \tau_{\mathcal{O}}^{X}}, \mu\right)\right]$$
for all $U=(\Omega, \mathcal F, \mathbb F, \mathbb P, M, \nu, \tau, X)\in\mathcal{U}^R[\mu, m]$. The value of the optimization problem in the weak formulation with strict optimal stopping and relaxed control associated to $(\mu, m)$ is defined by 
\begin{equation}\label{value weak relaxed 2}
V^R[\mu, m] := \sup_{U\in \mathcal{U}^R[\mu, m]} \mathcal{H}^R[\mu, m](U).    
\end{equation}
Moreover, we say that $U^\star=(\Omega, \mathcal F, \mathbb F, \mathbb P, M, \nu, \tau, X)$ is a \textit{weak MFG Nash equilibrium with relaxed control} if $U^\star\in \mathcal{U}^R[\mu^\star, m^\star]$, where
\begin{align}\label{c1}
   m_{t}^\star(B \times C)=\mathbb{E}^\mathbb P\left[\textbf{1}_B(X_{t}) \nu_t(C) \textbf{1}_{t \leq \tau\wedge \tau_{\mathcal{O}}^X}\right], \quad B \in \mathcal{B}(\bar{\mathcal{O}}), \quad C \in \mathcal{B}(A), \quad t \in[0, T], 
\end{align}
\begin{align}\label{c2}
\mu^\star=\mathbb P \circ \left(\tau\wedge \tau_{\mathcal{O}}^X, X_{\tau\wedge \tau_{\mathcal{O}}^X}\right)^{-1},
\end{align}
and 
\begin{align}\label{c3}
\mathcal{H}^{R}[\mu^\star, m^\star](U^\star)=V^{R}[\mu^\star, m^\star].
\end{align}
\end{definition}
}

The above definition is equivalent to the following formulation of MFG equilibrium via the \textit{controlled/stopped martingale problem}.

\begin{definition}[\textit{MFG equilibrium via the controlled/stopped martingale problem}]
Find a filtered probability space $(\Omega, \mathcal{F}, \mathbb F , \mathbb P)$, an $\mathbb F$-stopping time $\tau$ such that $\tau\leq T$ $\mathbb P$-a.s., an $\mathbb{F}$-progressively measurable process $(\nu_t(da))_{t \geq 0}$ with values in $\mathcal{P}(A)$ and an adapted process $X$ such that 
\begin{enumerate}[(1)]
\item $\mathbb P \circ X_0^{-1}=m_0^*$.
\item For all $\varphi\in C^2_b(\mathbb R)$, the process $(M_{t\wedge \tau}(\varphi))_{t\geq 0}$ is an $(\mathbb F, \mathbb P)$-martingale, where
$$M_t(\varphi):=\varphi(X_t)- \int_0^t\int_A \mathcal{L}\varphi(s, X_s, m_s, a)\nu_s(da)ds,$$
and
$$m_{t}(B \times C)=\mathbb{E}^\mathbb P\left[\textbf{1}_B(X_{t}) \nu_t(C) \textbf{1}_{t \leq \tau\wedge \tau_{\mathcal{O}}^X}\right], \quad B \in \mathcal{B}(\bar{\mathcal{O}}), \quad C \in \mathcal{B}(A), \quad t \in[0, T].$$
\item If $(\Omega', \mathcal{F}', \mathbb F', \mathbb P')$ is another filtered probability space, $\tau'$ an $\mathbb F'$-stopping time such that $\tau'\leq T$ $\mathbb P'$-a.s., $(\nu'_t(da))_{t \geq 0}$ an $\mathbb{F}'$-progressively measurable process with values in $\mathcal{P}(A)$, and an adapted process $X'$ such that $\mathbb P'\circ (X_0')^{-1}=m_0^*$ and for all $\varphi\in C^2_b(\mathbb R)$, the process $(M_{t\wedge \tau'}'(\varphi))_{t\geq 0}$ is an $(\mathbb F', \mathbb P')$-martingale, where
$$M_t'(\varphi):=\varphi(X_t')- \int_0^t\int_A \mathcal{L}\varphi(s, X_s', m_s, a)\nu_s'(da)ds,$$
then,
\begin{align*}
&\mathbb{E}^{\mathbb P'}\left[\int_{0}^{\tau'\wedge \tau_{\mathcal{O}}^{X'}} \int_A f\left(t, X_{t}', m_t, a\right)\nu'_t(da)dt  + g\left(\tau'\wedge \tau_{\mathcal{O}}^{X'}, X_{\tau'\wedge \tau_{\mathcal{O}}^{X'}}', \mu\right)\right] \\
&\leq \mathbb{E}^{\mathbb P}\left[\int_{0}^{\tau\wedge \tau_{\mathcal{O}}^X } \int_A f\left(t, X_{t}, m_t, a\right)\nu_t(da)dt  + g\left(\tau\wedge \tau_{\mathcal{O}}^X, X_{\tau\wedge \tau_{\mathcal{O}}^X}, \mu\right)\right],
\end{align*}
where
$$\mu=\mathbb P \circ \left(\tau\wedge \tau_{\mathcal{O}}^X, X_{\tau\wedge \tau_{\mathcal{O}}^X}\right)^{-1}.$$
\end{enumerate}
\end{definition}

\begin{remark}This definition is also equivalent to the problem of finding an MFG equilibrium via the controlled/stopped martingale problem on the canonical space (see \cite{lacker2015}), where the optimization is considered over the set of probabilities on the canonical space instead of all the tuples $(\Omega, \mathcal{F}, \mathbb F , \mathbb P, \tau, \nu, X)$. We refer to \cite{elkaroui2015}, p. 18, for more details on this equivalence.
\end{remark}

\begin{theorem}[\textit{Equivalence between LP MFG equilibria and weak MFG equilibria}]\label{Nash_exist_weak}
Suppose Assumption \ref{as_mfg} with either (5)(b) or (5)(c) holding true. Then, the LP MFG problem and the weak MFG problem are equivalent. More specifically, 
\begin{itemize}
\item[(i)] Given an \textit{LP MFG Nash equilibrium} $(\mu^\star, m^\star)$, there exists a \textit{weak MFG Nash equilibrium} (with Markovian relaxed control) $U^\star\in \mathcal{U}^R[\mu^\star, m^\star]$ such that
\begin{align}\label{c4}
\Gamma[\mu^\star, m^\star] (\mu^\star, m^\star)= \mathcal{H}^{R}[\mu^\star, m^\star](U^\star).
\end{align}
\item[(ii)] Given $U^\star$  a \textit{weak MFG Nash equilibrium}, that is $U^\star\in \mathcal{U}^R[\mu^\star, m^\star]$, with $m^\star$ (resp. $\mu^\star$) given by $\eqref{c1}$ (resp. $\eqref{c2}$), then $(\mu^\star, m^\star)$ is an \textit{LP MFG Nash equilibrium} and \eqref{c4} holds.
\end{itemize} 
\end{theorem}
\begin{proof}
Considering measure dependent coefficients, the equivalence follows from Proposition \ref{SDE to measures} and Theorem \ref{SDE rep}.
\end{proof}

\begin{corollary}
Suppose Assumption \ref{as_mfg} with either (5)(b) or (5)(c) holding true, then there exists a weak Nash equilibrium (with Markovian relaxed control).
\end{corollary}
\begin{proof}
By Theorem \ref{Nash_exist_LP} we get the existence of LP MFG Nash equilibrium, which implies by Theorem \ref{Nash_exist_weak} the existence of a weak Nash equilibrium (with Markovian relaxed control).
\end{proof}

\begin{remark}
In the case when there is only control, we recover the existence result of Markovian relaxed controls of \cite[ Corollary 3.8]{lacker2015}. In that paper, this result is shown by using the Mimicking Theorem  (or Markovian projection theorem) from Corollary 3.7. in \cite{brunick2013}, while in our case this result follows naturally by the disintegration
$$m_t(dx, da)dt=\nu_{t, x}(da)m_t(dx, A)dt.$$
\end{remark}

{
\begin{proposition}
Suppose $\mathcal{O}=\mathbb R$ and let Assumption \ref{as_mfg} with either (5)(b) or (5)(c) holding true. Let $(\mu^\star, m^\star)$ be an LP Nash equilibrium. Consider the value function given by 
\begin{equation}\label{value_fct_mfg}
v^\star(t, x)=\sup _{\tau \in \mathcal{T}_t, \alpha \in \mathbb{A}_t} \mathbb{E}\left[\int_{t}^{\tau \wedge \tau_{\mathcal{O}}^{t, x, \alpha}} f\left(s, X_{s}^{t, x, m^\star, \alpha}, m_t^\star, \alpha_s\right) ds + g\left(\tau \wedge \tau_{\mathcal{O}}^{t, x, m^\star,\alpha}, X_{\tau \wedge \tau_{\mathcal{O}}^{t, x, m^\star, \alpha}}^{t, x, m^\star, \alpha}, \mu^\star\right)\right],
\end{equation}
where $(t, x) \in[0, T] \times \mathbb{R}$, $\tau_{\mathcal{O}}^{t, x, m^\star, \alpha} :=\inf \left\{s \geq t : \, X_{s}^{t, x, m^\star, \alpha} \notin \mathcal{O}\right\}$ and $(X^{t, x, m^\star, \alpha}_s)_{s\in [t, T]}$ is the unique  strong solution of the following stochastic differential equation:
$$X^{t, x, m^\star, \alpha}_{s}=x + \int_t^s b\left(r, X^{t, x, m^\star, \alpha}_{r}, m_r^\star, \alpha_r \right) dr+\int_t^s\sigma\left(r, X^{t, x, m^\star, \alpha}_{r}, m_r^\star, \alpha_r\right) dW_{r}, \quad s\in [t, T].$$
We have the following equality:
$$\int_\mathcal{O}v^\star(0, x)m_0^*(dx)=V^W[\mu^\star, m^\star]=V^R[\mu^\star, m^\star]=\Gamma[\mu^\star, m^\star](\mu^\star, m^\star).$$
\end{proposition}
\begin{proof}
Since $(\mu^\star, m^\star)$ is fixed in the functions $b$, $\sigma$, $f$ and $g$, we can apply Theorem \ref{eq values} noticing that Assumption \ref{as values} is satisfied.
\end{proof}}

\begin{proposition}
Suppose that Assumption \ref{as_mfg} with either (5)(b) or (5)(c) holding true and that for all $(t, x, (\mu,m))\in [0, T]\times \bar{\mathcal{O}}\times\mathcal{R}_0$, the subset
$$K[m](t, x):=\{(b(t, x, m_t, a), \sigma^2(t, x, m_t, a), z):a\in A, z\leq f(t, x, m_t, a)\}$$
of $\mathbb R \times\mathbb R_+ \times \mathbb R$ is convex. Then there exist a strict control LP Nash equilibrium and a weak Nash equilibrium with Markovian strict control.
\end{proposition}
\begin{proof}
The proof is almost the same as that of Proposition \ref{strict cont}; it relies on the fact that the dependence of $b$, $\sigma^2$ and $f$ in the measure is of the form
$$\int_{\bar{\mathcal{O}}\times A}h(t, x)m_t(dx, da),$$
for some function $h$, which is independent of the control.
\end{proof}

\subsection{Relation with \textit{mixed solutions}}

In this subsection, to establish the link with PDE formulation, we shall need the following assumptions:
\begin{assumption}\label{as mixed sol} \leavevmode
\begin{enumerate}[(1)]
\item The domain $\mathcal{O}$ is a bounded open domain of class $C^2$.
\item The volatility $\sigma$  does not depend on the control $a$ and on the measure $m$, and is continuous on $[0, T]\times \bar{\mathcal{O}}$. Moreover, there exists $c_\sigma>0$ such that for all $(t, x)\in [0, T]\times \bar{\mathcal{O}}$, $\sigma(t, x)\geq c_\sigma$.
\item There exists $c_\sigma>0$ such that for all $(t, x)\in [0, T]\times \bar{\mathcal{O}}$, $\sigma(t, x)\geq c_\sigma$.
\item $f$ is measurable, bounded and {\color{black}continuous in $x $ on $\bar{\mathcal{O}}$,} uniformly with respect to $t$, $m$ and $a$.
\item For fixed $(t,x, m)\in [0, T]\times \bar{\mathcal{O}}\times \mathcal{M}(\bar{\mathcal{O}}\times A)$, $a\mapsto b(t,x,m, a)$ and $a\mapsto f(t,x,m,a)$ are continuous.
\item For a fixed $\mu\in \mathcal{P}([0, T]\times \bar{\mathcal{O}})$, $(t, x)\mapsto g(t, x, \mu)\in C^{1, 2}_b([0, T]\times \bar{\mathcal{O}})$ and $g(t, x, \mu)=0$ for $(t, x)\in (0, T)\times \partial \mathcal{O}$.
\item If $(\mu, m)\in \mathcal{R}[\hat m]$ for some $\hat m\in V$, then $m_t(dx, A)dt$ admits an square integrable density with respect to the Lebesgue measure on $[0, T]\times \bar{\mathcal{O}}$.
\end{enumerate} 
\end{assumption}

{
\begin{theorem}[\textit{Relation with mixed solutions}]\label{mixedsol}
Suppose Assumptions \ref{as_mfg} and \ref{as mixed sol} hold true. Let $(\mu^\star, m^\star)$ be an LP Nash equilibrium. Consider the value function given by \eqref{value_fct_mfg}. We have the following relations.
\begin{enumerate}[(1)]
\item Relation with the strong formulation:
$$\int_\mathcal{O}v^\star(0, x)m_0^*(dx)=\Gamma[\mu^\star, m^\star](\mu^\star, m^\star).$$
\item Relation with mixed solutions:
\begin{enumerate}[(a)]
\item 
$$\int_{\mathcal{S}^\star\times A} \left(f+\frac{\partial g}{\partial t}+\mathcal{L} g\right)(t, x, m_t^\star, \mu^\star, a) m_t^\star(dx, da) dt = 0,$$
with $\mathcal{S}^\star:=\{(t, x)\in [0, T]\times \mathcal{O}: v^\star(t, x)=g(t, x, \mu^\star)\}$.
\item 
$$-\int_{\mathcal{C}^\star\times A} f(t, x, m_t^\star, a)m_t^\star(dx, da)dt= \int_{\mathcal{C}^\star\times A} \left(\frac{\partial v}{\partial t}+\mathcal{L} v\right)(t, x, m_t^\star, a)m_t^\star(dx, da)dt,$$
where $\mathcal{C}^\star:=([0, T]\times \mathcal{O})\setminus \mathcal{S}^\star$.
\item For all $C^\infty$ functions $\phi$ such that $\operatorname{supp} (\phi)\subset \mathcal{C}^\star$, the following holds
\begin{equation*}
\int_{\mathcal{O}} \phi(0, x) m_{0}^{*}(dx)+\int_{0}^{T} \int_{\mathcal{O}\times A}\left(\frac{\partial \phi}{\partial t}+\mathcal{L} \phi\right)(t, x, m_t^\star, a) m_t^\star(dx, da) dt = 0.
\end{equation*}
\end{enumerate}
Note that (2)(c) holds true if and only if $\mu^\star(\mathcal{C}^\star)=0$, which is also equivalent to $\mu^\star(\mathcal{S}^\star\cup ([0, T]\times \partial \mathcal{O}))=1$.
\end{enumerate}
\end{theorem}}

\begin{proof}
The proof follows by applying Theorem \ref{PDE} taking into account that the inputs $(b, \sigma, f, g)$ depend now on $(m^\star, \mu^\star)$ but still satisfy the required assumptions.
\end{proof}

\begin{corollary}
Let Assumptions \ref{as_mfg} and \ref{as mixed sol} hold true and assume that for all $(t, x, (\mu,m))\in [0, T]\times \bar{\mathcal{O}}\times\mathcal{R}_0$, the subset
$$K[m](t, x):=\{(b(t, x, m_t, a), \sigma^2(t, x, m_t, a), z):a\in A, z\leq f(t, x, m_t, a)\}$$
of $\mathbb R \times\mathbb R_+ \times \mathbb R$ is convex. Let $(\mu^\star, m^\star)$ an LP MFG equilibrium, then, there exists $\alpha^\star(t,x)$ such that $\bar{m}_t(dx) \equiv m^\star_t(dx,A)$ satisfies the following system:
$$
\begin{cases}
\int_{\mathcal{S}} \left(f+\frac{\partial g}{\partial t}+\mathcal{L} g\right)(t, x,m^\star_t, \mu^\star, \alpha^\star(t,x)) \bar{m}_t(dx) dt = 0,\\
\alpha^\star(t, x)\in \arg\max_{a\in A}\left[\mathcal{L} v(t, x, m_t^\star, a) + f(t, x, m_t^\star, a)\right] \quad \bar{m}_t(dx) dt-a.e. \text{ on } \mathcal{C}, \nonumber \\
\int_{\mathcal{O}} \phi(0, x) m_{0}^{*}(dx)+\int_{0}^{T} \int_{\mathcal{O}\times A}\left(\frac{\partial \phi}{\partial t}+\mathcal{L} \phi\right)(t, x, m_t^\star \alpha^\star(t, x)) \bar{m}_t(dx) dt = 0, \\ \text{for all } C^\infty \text{ functions } \phi \text{ such that } \operatorname{supp} (\phi)\subset \mathcal{C}.
\end{cases}
$$
\end{corollary}

\begin{remark}
The above result gives the link with the notion of mixed solution in the case of optimal stopping/continuous control introduced in \cite{bertucci2017} in a less general framework (in particular, the author considers the drift to be zero and the volatility $\sqrt{2}$).
\end{remark}

\section*{Acknowledgement}Peter Tankov gratefully acknowledges financial support from the ANR (project EcoREES ANR-19-CE05-0042) and from the
FIME Research Initiative.


\appendix

\section{Structure of $V_1$ and $V$}\label{sec A}

We show in this Appendix that $V_1$ is a Hausdorff locally convex topological vector space and $V$ is metrizable. 

Let $\mathcal{M}^s:=\mathcal{M}^s([0, T]\times\bar{\mathcal{O}}\times A)$ be the set of Borel finite signed measures on $[0, T]\times\bar{\mathcal{O}}\times A$. Endow this set with the weak topology $\tau_w:=\sigma(\mathcal{M}
^s, \mathcal{F}^s)$, where 
$$\mathcal{F}^s:=\left\{\mathcal{M}^s\ni\mu\rightarrow\int\phi d\mu: \phi\in C_b([0, T]\times \bar{\mathcal{O}}\times A)\right\}.$$
In other words, $\tau_w$ is the topology generated by the sets
$$U(\mu, \phi, \varepsilon):=\left\{\nu\in \mathcal{M}^s: \left| \int \phi d\nu - \int \phi d\mu \right|<\varepsilon \right\}, \quad \mu\in \mathcal{M}^s,\;\phi\in C_b([0, T]\times\bar{\mathcal{O}}\times A),\; \varepsilon>0.$$
Since $C_b([0, T]\times\bar{\mathcal{O}}\times A))$ is separating, i.e. for all $\mu, \nu\in \mathcal{M}^s$,
$$\int\phi d\mu=\int\phi d\nu, \quad \forall\phi\in C_b([0, T]\times\bar{\mathcal{O}}\times A))\Rightarrow \mu=\nu,$$
then $\mathcal{F}^s$ is total, which implies that $\tau_w$ is Hausdorff (see p. 48 in \cite{aliprantis2007}). Moreover, $(\mathcal{M}^s, \tau_w)$ is a locally convex topological vector space, since weak topologies with respect to a family of real valued functions make the space locally convex. Define the map $\pi:V_1\rightarrow \mathcal{M}^s$ by
$$\pi(m)=m_t(dx, da)dt.$$
The map $\pi$ is injective since the elements of $V_1$ are identified $t$-a.e. We define $\mathcal{M}^s_1:=\pi(V_1)$ and consider the relative weak topology on $\mathcal{M}^s_1$ which is given by $\sigma(\mathcal{M}^s_1, \mathcal{F}_1^s)$, where $\mathcal{F}_1^s:=\mathcal{F}^s|_{\mathcal{M}^s_1}$ (Lemma 2.53 in \cite{aliprantis2007}). Note that $(\mathcal{M}^s_1,\sigma(\mathcal{M}^s_1, \mathcal{F}_1^s))$ is also a Hausdorff locally convex topological vector space. We have that $\pi:V_1\mapsto \mathcal{M}^s_1$ is a linear bijection. Finally, we endow $V_1$ with the projective topology $\tau_1:=\pi^{-1}(\sigma(\mathcal{M}^s_1, \mathcal{F}_1^s))$ (that is the topology of weak convergence of the associated measures on $[0, T]\times\bar{\mathcal{O}}\times A$). With this definition, $\pi$ is an isomorphim between the topological vector spaces, which implies that $(V_1, \tau_1)$ is a Hausdorff locally convex topological vector space.

The relative topology on $V$ is metrizable since $\pi(V)\subset \mathcal{M}([0, T]\times \bar{\mathcal{O}}\times A)$ and the weak convergence topology on $\mathcal{M}([0, T]\times \bar{\mathcal{O}}\times A)$ is metrizable, in particular, we can define a natural distance on $V$ associated to $\pi$.\\

We recall that the set $\mathcal{P}([0, T]\times \bar{\mathcal{O}})$ endowed with the topology of weak convergence is also metrizable, and hence the product space $\mathcal{P}([0, T]\times \bar{\mathcal{O}})\times V$ is metrizable.

\section{Martingale measures and controlled/stopped martingale problem}\label{sec B}

\subsection{Martingale measures}\label{sec B.1}

For the sake of clarity we present the definition of martingale measures and some related concepts. This content is taken from \cite{elkaroui1990} and \cite{meleard1992}. Throughout the section we fix a filtered probability space $(\Omega, \mathcal{F}, \mathbb F, \mathbb P)$ and a Polish space $A$ with Borel $\sigma$-algebra $\mathcal B(A)$.

\begin{definition}
We say that $M:\Omega\times \mathbb R_+\times \mathcal{B}(A)\rightarrow \mathbb R$ is an (orthogonal) martingale measure if it satisfies the following properties:
\begin{enumerate}[(1)]
\item For all $B\in \mathcal{B}(A)$, $M(\cdot, B)$\footnote{Note that we suppress the argument $\omega\in\Omega$ from the notation as usual in probability theory.} is a square integrable martingale and $M(0, B)=0$.
\item For all $t\in \mathbb R_+$, $B, C\in \mathcal{B}(A)$ such that $B\cap C=\emptyset$, $M(t, B\cup C)=M(t, B)+M(t, C)$ a.s.
\item There exists a non-decreasing sequence of $(A_n)_{n\geq 1}\subset \mathcal{B}(A)$ such that
\begin{itemize}
\item $\cup_{n\geq 1}A_n=A$.
\item For all $t\in \mathbb R_+$, {\color{black}and all $n\geq 1$,}
$$\sup_{B\in \mathcal{B}(A_n)}\mathbb E[M(t, B)^2]<\infty.$$
\item For all $t\in \mathbb R_+$, $n\geq 1$ and $(B_k)_{k\geq 1}\subset \mathcal{B}(A_n)$ a decreasing sequence such that $\cap_{k\geq 1}B_k= \emptyset$,
$$\mathbb E[M(t, B_k)^2]\underset{k\rightarrow\infty}{\longrightarrow}0.$$
\end{itemize}
\item For all $B, C\in \mathcal{B}(A)$ such that $B\cap C=\emptyset$, the martingales $(M(t,B))_{t\geq 0}$ and $(M(t,C))_{t\geq 0}$ are orthogonal, i.e. $(M(t,B)M(t,C))_{t\geq 0}$ is a martingale.
\end{enumerate}
\end{definition}

A martingale measure $M$ is said to be continuous if for all $B\in \mathcal{B}(A)$, $t\mapsto M(t, B)$ is continuous a.s.

\begin{remark}
If $\tau$ is a stopping time and $M$ a martingale measure, then $M^\tau(t, B):=M(t\wedge\tau, B)$ is also a martingale measure.
\end{remark}

\begin{theorem}[Theorem I-4 in \cite{elkaroui1990}]
If $M$ is a martingale measure, then there exists a random $\sigma$-finite positive measure $\nu$ on $\mathbb R_+\times A$, such that for each $B\in \mathcal{B}(A)$, $(\nu([0, t]\times B))_{t\geq 0}$ is the predictable quadratic variation of $(M(t, B))_{t\geq 0}$. The measure $\nu$ is called the intensity of $M$.
\end{theorem}

Let $M$ be a martingale measure with intensity $\nu$ and let $L^2_\nu$ the set of functions $\phi:\Omega\times \mathbb R_+\times A\rightarrow\mathbb R$ measurable with respect to the product of the predictable $\sigma$-algebra and $\mathcal{B}(A)$, such that
$$\mathbb E\left[\int_{\mathbb R_+\times A}\phi^2(s, a)\nu(ds, da)\right]<\infty.$$
Then for any $\phi\in L^2_\nu$ one can construct a stochastic integral of $\phi$ with respect to $M$, which is a function from $\Omega\times \mathbb R_+\times \mathcal{B}(A)$ to $\mathbb R$. It is denoted by $\phi\cdot M$. We will also denote
$$\int_0^t\int_B \phi(s, a)M(ds, da):=(\phi\cdot M)(t, B), \quad t\in\mathbb R_+,\ B\in \mathcal{B}(A).$$
The construction is analogous to the one of the Itô integral.

\begin{proposition}[Proposition I-6 in \cite{elkaroui1990}]
Let $M$ be a martingale measure with intensity $\nu$.
\begin{enumerate}[(1)]
\item If $\phi\in L^2_\nu$, then $\phi\cdot M$ is a martingale measure with intensity $\phi^2(s, a)\nu(ds, da)$. Moreover, if $M$ is continuous, then $\phi\cdot M$ is also continuous.
\item If $\phi, \psi\in L^2_\nu$ and $B, C\in \mathcal{B}(A)$, then for all $t\in\mathbb R_+$,
$$\left< \phi\cdot M(\cdot, B), \psi\cdot M(\cdot, C) \right>_t=\int_0^t\int_{B\cap C}\phi(s, a)\psi(s,a)\nu(ds,da).$$
\end{enumerate}
\end{proposition}

A consequence of this proposition is that
$$\left( \int_0^t \int_A \phi(s, a) M(ds,da) \right)_{t\geq 0}$$
is a martingale with quadratic variation
$$\left( \int_0^t \int_A \phi^2(s, a)\nu(ds, da) \right)_{t\geq 0}.$$
This fact allows the use of Burkholder-Davis-Gundy inequality, which can be applied to prove existence of strong solutions to SDEs of the type 
$$X_t=\xi+\int_0^t\int_A b(s, X_s, a)\nu(ds, da) + \int_0^t\int_A \sigma(s, X_s, a)M(ds, da), \quad t\geq 0.$$
under standard assumptions.

\subsection{Controlled/stopped martingale problem}\label{sec B.2}

Recall that the linear operator $\mathcal{L}$ is given by
$$\mathcal{L}\varphi(t, x, a)=b(t, x, a) \varphi'(x) + \frac{\sigma^2}{2}(t, x, a)\varphi''(x), \quad (t, x, a)\in [0, T]\times \mathbb R\times A,\; \varphi\in C^2_b(\mathbb R).$$

\begin{definition}
The tuple $(\Omega, \mathcal{F}, \mathbb F, \mathbb P, \nu, \tau, X)$ is said to be a solution of the controlled/stopped martingale problem if 
\begin{enumerate}[(1)]
\item $(\Omega, \mathcal{F}, \mathbb F, \mathbb P)$ is a filtered probability space supporting an $\mathbb F$-progressively measurable process $\nu$ with values in $\mathcal{P}(A)$, an $\mathbb F$-stopping time $\tau$ and an $\mathbb F$-adapted process $X$.
\item For all $\varphi\in C^2_b(\mathbb R)$, the process $(M_{t\wedge \tau}(\varphi))_{t\geq 0}$ is a martingale, where
$$M_t(\varphi):=\varphi(X_t)- \int_0^t\int_A \mathcal{L}\varphi(s, X_s, a)\nu_s(da)ds.$$
\end{enumerate}
\end{definition}

\begin{theorem}\label{EM}
Let $(\Omega, \mathcal{F}, \mathbb F, \mathbb P, \nu, \tau, X)$ be a solution of the controlled/stopped martingale problem. Suppose that $X_{\cdot\wedge\tau}$ is continuous, $\tau$ is bounded and the coefficients $b$ and $\sigma$ are bounded. Then, on an extension of the filtered probability space, there exists a continuous martingale measure $M$ with intensity $\nu_t(da)\mathds{1}_{t\leq \tau}dt$ such that
$$dX_t=\int_A b(t, X_t, a)\nu_t(da)dt + \int_A\sigma(t, X_t, a)M(dt, da), \quad t\leq \tau.$$
Moreover, there exists a Brownian motion $W$ such that $M(t, A)=M(t\wedge\tau, A)=W_{t\wedge\tau}$. In particular, if $\sigma$ is uncontrolled,
$$dX_t=\int_A b(t, X_t, a)\nu_t(da)dt + \sigma(t, X_t)dW_t, \quad t\leq \tau.$$
\end{theorem}
\begin{proof}
Using the same proof as in Lemma 3.2 of \cite{lacker2015}, there exists an $\mathbb F$-predictable process $\bar \nu$ with values in $\mathcal{P}(A)$ such that $\bar\nu_t=\nu_t$ $t$-a.e. on $[0, T]$. In particular, $(\Omega, \mathcal{F}, \mathbb F, \mathbb P, \bar\nu, \tau, X)$ is a solution of the controlled/stopped martingale problem. With some abuse of notation we denote $\bar \nu$ by $\nu$. For all $u\in C_b^2(\mathbb R)$,
$$u(X_{t\wedge\tau})-u(X_0)-\int_0^{t\wedge\tau} \int_A \mathcal{L}u(s, X_s, a)\nu_s(da)ds$$
is an $\mathbb F$-martingale. Define $\tilde X:=X_{\cdot\wedge \tau}$ and $q_t(da)=\nu_t(da)\mathds{1}_{t\leq\tau}$. Then, for all $u\in C_b^2(\mathbb R)$,
$$u(\tilde X_t)-u(\tilde X_0)-\int_0^{t} \int_A \mathcal{L}u(s, \tilde X_s, a)q_s(da)ds$$
is an $\mathbb F$-martingale. Moreover, since the processes $(\nu_t)_{t\in [0, T]}$ and $(\mathds{1}_{t\leq\tau})_{t\in [0, T]}$ are $\mathbb F$-predictable and the map $\pi:\mathbb R_+\times \mathcal{P}(A)\rightarrow \mathcal{M}(A)$ given by $\pi(\lambda, \nu)=\lambda \nu$ is continuous, we get that the process $(q_t)_{t\in [0, T]}$ is $\mathbb F$-predictable. By Theorem IV-2 in \cite{elkaroui1990}, there exists an extension of the filtered probability space, denoted by $(\Omega', \mathcal{F}', \mathbb F', \mathbb P')$ supporting a martingale measure $M$ with intensity $q_t(da)dt$ such that
$$\tilde X_t=\tilde X_0 + \int_0^t\int_A b(s, \tilde X_s, a)q_s(da)ds + \int_0^t\int_A \sigma(s, \tilde X_s, a)M(ds, da), \quad t\geq 0.$$
Since $\nu_t(A)=1$ for all $t\geq 0$, we get that $(M(t, A))_{t\geq 0}$ is a continuous square integrable martingale with quadratic variation $(t\wedge\tau)_{t\geq 0}$. Define $M_t:=M(t, A)$ for $t\geq 0$ and note that since $(M_{t+\tau}-M_\tau)_{t\geq 0}$ is an $(\mathcal{F}_{t+\tau}')_{t\geq 0}$ martingale, 
$$\mathbb E^{\mathbb P'}[M_{t+\tau}-M_\tau]^2=\mathbb E^{\mathbb P'}[\left<M\right>_{t+\tau}-\left<M\right>_{\tau}]=0, \quad t\geq 0,$$
which means that $\mathbb P'$-a.s., $M_t=M_{t\wedge \tau}$, $t\geq 0$. Consider the filtration $\mathbb G$ given by $\mathcal{G}_t:=\mathcal{F}_{t\wedge\tau}'\subset\mathcal{F}_t'$. By Theorem 1.7, Chapter V, in \cite{revuz1999}, on an extension of $(\Omega', \mathcal{F}', \mathbb G, \mathbb P')$ denoted by $(\tilde \Omega, \tilde{\mathcal{F}}, \tilde{\mathbb{F}}, \tilde{\mathbb{P}})$, there exists an $\tilde{\mathbb{F}}$-Brownian motion $W$ such that $W_{t\wedge\tau}=M_{t\wedge\tau}=M_t$, $t\geq 0$. Note that the definition of the stochastic integral
$$\int_0^t\int_A\sigma(s, \tilde X_s, a)M(ds, da)$$
depends on the filtration, but since $\tilde X$ is $\mathbb G$-progressively measurable, its extension is $\tilde{\mathbb F}$-progressively measurable, therefore the integrals in both spaces coincide. Analogously to the standard stochastic integral, the stopped integral is equal to the integral with respect to the stopped martingale measure $M^\tau:=M_{\cdot\wedge\tau}$, which together with $M=M^\tau$, gives
$$\int_0^{t}\int_A\sigma(s, \tilde X_s, a)M(ds, da)=\int_0^{t}\int_A\sigma(s, \tilde X_s, a)M^\tau(ds, da)=\int_0^{t\wedge\tau}\int_A\sigma(s, X_s, a)M(ds, da).$$
We conclude that
$$X_{t\wedge\tau}=X_0 + \int_0^{t\wedge\tau}\int_A b(s, X_s, a)\nu_s(da)ds + \int_0^{t\wedge\tau}\int_A \sigma(s, X_s, a)M(ds, da), \quad t\geq 0.$$
If $\sigma$ is uncontrolled, by the construction of the integral with respect to $M$, one can deduce that
$$\int_0^t\int_A\sigma(s, \tilde X_s)M(ds, da)=\int_0^t\sigma(s, \tilde X_s)dW^\tau_s=\int_0^{t\wedge\tau}\sigma(s, X_s)dW_s,$$
which allows to write,
$$X_{t\wedge\tau}=X_0 + \int_0^{t\wedge\tau}\int_A b(s, X_s, a)\nu_s(da)ds + \int_0^{t\wedge\tau}\sigma(s, X_s)dW_s, \quad t\geq 0.$$
\end{proof}

In the case where the relaxed control $\nu$ is replaced by some strict control $\alpha$, we can also find a SDE representation with respect to a Brownian motion.

\begin{theorem}\label{EK}
Let $(\Omega, \mathcal{F}, \mathbb F, \mathbb P, \nu, \tau, X)$ be a solution of the controlled/stopped martingale problem. Suppose that $X_{\cdot\wedge\tau}$ is continuous, $\nu_t=\delta_{\alpha_t}$ for some $\mathbb F$-progressively measurable process $\alpha$, $\tau$ is bounded and the coefficients $b$ and $\sigma$ are bounded. Then, on an extension of the filtered probability space, there exists a Brownian motion $W$ such that
$$dX_t= b(t, X_t, \alpha_t)dt + \sigma(t, X_t, \alpha_t)dW_t, \quad t\leq \tau.$$
\end{theorem}
\begin{proof}
Adapting the proof of Theorem 3.3 in \cite{ethier1986} to random coefficients, we get the result for the case without stopping time. Using the same techniques as in Theorem \ref{EM}, we get the result.
\end{proof}

\section{Link between linear programming and the weak formulation}

We have seen in Proposition \ref{SDE to measures} that to any controlled and stopped diffusion we can associate a pair $(\mu, m)\in\mathcal{R}$. In Theorem \ref{SDE rep} we will prove that any $(\mu, m)\in\mathcal{R}$ can be represented in terms of a controlled and stopped diffusion.

\begin{lemma}\label{lemma exit closure}
Consider a filtered probability space $(\Omega, \mathcal{F}, \mathbb F, \mathbb P)$ supporting an $\mathbb F$-Brownian motion $W$. Let $T>0$, $\xi$ an $\mathcal{F}_0$-measurable random variable supported in $\mathcal{O}$, $b$ a bounded $\mathbb F$-progressively measurable process and $\sigma$ a bounded $\mathbb F$-progressively measurable process bounded below by a constant $c_\sigma>0$ and above by a constant $C_\sigma>0$. Let $Y$ be defined by 
$$Y_t=\xi + \int_0^tb_sds + \int_0^t\sigma_s dW_s .$$
Then $\tau_{\mathcal{O}}^Y=\tau_{\bar{\mathcal{O}}}^Y$ $\mathbb P$-a.s.
\end{lemma}

\begin{proof}
{\color{black}Let $B=\{\tau_\mathcal{O}^Y=\infty\}$, then for $\omega\in B$ we obtain $\tau_\mathcal{O}^Y(\omega)=\tau_{\bar{\mathcal{O}}}^Y(\omega)$.} We remark that since $b$ is bounded, on the event $B^c$,
$$
\lim_{s\downarrow 0}\frac{\int_{\tau^Y_\mathcal O}^{\tau^Y_\mathcal O+s} b_r dr }{\sqrt{s}} = 0.
$$
By Dambis-Dubbins-Schwarz theorem, there exists a Brownian motion $\widetilde W$ such that $\int_0^t \sigma_s dW_s = \widetilde W_{\int_0^t \sigma_s^2 ds}$ for all $t\geq 0$. Finally, recall the classical result: for a Brownian motion $B$, 
$$
\liminf_{s\downarrow 0} \frac{B_s}{\sqrt{s}} = -\infty\quad \text{and}\quad \limsup_{s\downarrow 0} \frac{B_s}{\sqrt{s}} = +\infty,
$$
holds a.s., and by the strong Markov property, this result holds true at any stopping time. We denote by $C$ the event where this result holds true at time $\tau_\mathcal{O}^Y$, which has probability one. Therefore, using that $0<c_\sigma\leq \sigma_t\leq C_\sigma$, on the event $B^c\cap C$,
\begin{align*}
\limsup_{s\downarrow 0} \frac{Y_{\tau^Y_\mathcal O + s} - Y_{\tau^Y_\mathcal O }}{\sqrt{s}} &=  \limsup_{s\downarrow 0} \frac{\int_{\tau^Y_\mathcal O}^{\tau^Y_\mathcal O+s}\sigma_r dW_r}{\sqrt{s}}\\ 
&= \limsup_{s\downarrow 0} \frac{\widetilde W_{\int_0^{\tau^Y_\mathcal O+s}\sigma^2_r dr}-\widetilde W_{\int_0^{\tau^Y_\mathcal O}\sigma^2_r dr}}{\sqrt{\int_{\tau^Y_\mathcal O}^{\tau^Y_\mathcal O+s}\sigma^2_r dr}}\sqrt{\frac{\int_{\tau^Y_\mathcal O}^{\tau^Y_\mathcal O+s}\sigma^2_r dr}{s}} \\
& = +\infty,
\end{align*}
and similarly, 
$$
\liminf_{s\downarrow 0} \frac{Y_{\tau^Y_\mathcal O + s} - Y_{\tau^Y_\mathcal O }}{\sqrt{s}}  = -\infty. 
$$
Together, these two results imply that on the event $B^c \cap C$, $\tau^Y_\mathcal O = \tau^Y_{\bar{\mathcal O}}$. Since $B\cup (B^c\cap C)$ has probability one we conclude the proof.
\end{proof}

Let us recall some definitions and results of \cite{kurtz1998} Section 2. Let $E$ be a complete, separable metric space. We denote by $B(E)$ the set of bounded and measurable functions from $E$ to $\mathbb R$. Let $L\subset B(E)\times B(E)$ be the graph of an operator $L$ (we abuse of notation as it is usual to identify an operator with its graph). Let $L_S$ be the linear span of an operator $L$.

\begin{definition}
Let $L:\mathcal{D}(L)\subset B(E)\rightarrow B(E)$ an operator and $\nu_0\in \mathcal{P}(E)$. We say that a measurable $\mathcal{P}(E)$-valued function (we endow $\mathcal{P}(E)$ with the Borel $\sigma$-algebra generated by the topology of weak convergence) $\nu$ on $\mathbb R_+$ is solution of the forward equation for $(L, \nu_0)$ if for all $\phi\in \mathcal{D}(L)$ and $t\in \mathbb R_+$,
$$\int_E \phi(x)\nu_t(dx)=\int_E \phi(x)\nu_0(dx) + \int_0^t\int_E L\phi(x)\nu_s(dx)ds.$$
\end{definition}

\begin{definition}
An operator $L\subset B(E)\times B(E)$ is dissipative if $L_S$ is dissipative, that is, for $(f, g)\in L_S$ and $\lambda>0$,
$$\|\lambda f-g\|_\infty\geq \lambda \| f \|_\infty.$$
\end{definition}

\begin{definition}
An operator $L\subset B(E)\times B(E)$ is a pre-generator if $L$ is dissipative and there are sequences of functions $\mu_n:E\rightarrow \mathcal{P}(E)$ and $\lambda_n:E\rightarrow[0, \infty)$ such that for each $(f, g)\in L$
$$g(x)=\lim_{n\rightarrow\infty}\lambda_n(x) \int_E (f(y)-f(x))\mu_n(x, dy).$$
\end{definition}

\begin{proposition}\label{pre-generator}
If $L\subset C_b(E)\times C_b(E)$ and for each $x\in E$, there exists a solution $\nu^x$ of the forward equation for $(L, \delta_x)$ that is right-continuous (in the weak topology) at zero, then $L$ is a pre-generator.
\end{proposition}

Now, we will show that any $(\mu, m)\in \mathcal{R}$ has a probabilistic representation in terms of a controlled and stopped diffusion. The first part of the proof is based on the works of Stockbridge and coauthors (see e.g. \cite{stockbridge1998,stockbridge2002,stockbridge2017}) with adaptations to our case. The second part uses the equivalence of the stopped/controlled martingale problem and the diffusions.

\begin{theorem}\label{SDE rep}
Suppose that Assumptions \ref{as_single_agent} (1-2) and \ref{as exit time} hold. Suppose that $(\mu, m)\in \mathcal{R}$. Let $\nu_{t, x}(da)$ be such that 
$$m_t(dx, da)dt=\nu_{t, x}(da)m_t(dx, A)dt.$$
Then there exist a filtered probability space $(\Omega, \mathcal{F}, \mathbb F, \mathbb P)$, an $\mathbb F$-adapted process $X$, an $\mathbb F$-stopping time $\tau$ such that $\tau\leq T\wedge\tau_\mathcal{O}^X$ $\mathbb P$-a.s., and an $\mathbb F$-martingale measure $M$ with intensity $\nu_{t, X_t}(da)\mathds{1}_{t\leq \tau}dt$, such that
$$X_{t\wedge \tau}= \int_0^{t\wedge \tau}\int_A b(t, X_t, a)\nu_{t, X_t}(da)dt + \int_0^{t\wedge \tau}\int_A \sigma(t, X_t, a)M(dt, da), \quad \mathbb P \circ X_0^{-1}= m_0^*,$$
$$\mu =\mathbb P \circ (\tau, X_\tau)^{-1},$$
$$m_t(B\times C)= \mathbb E^{\mathbb P}\left[ \mathds{1}_B(X_t)\nu_{t, X_t}(C) \mathds{1}_{t\leq \tau}\right],  \quad B\in \mathcal{B}(\bar{\mathcal{O}}), \quad C\in \mathcal{B}(A), \quad t-a.e.$$
Moreover, if $\sigma$ is uncontrolled or $\nu_{t, x}=\delta_{\alpha(t, x)}$ for some measurable function $\alpha$, then one can replace the martingale measure by a Brownian motion.
\end{theorem}

\begin{proof}
We divide the proof in 4 steps. The first one is the redefinition of the coefficients and measures in order to construct an operator and a measure verifying the stationary equation. The second one contains the verification of the conditions to apply Corollary 1.10 in \cite{stockbridge2017}. In the third step we apply this Corollary to obtain a controlled/stopped martingale problem formulation. Finally, in the fourth step, we go from the controlled/stopped martingale problem to the diffusion representation.

\textit{First step: Construction of the operator and the stationary measure.} We extend $\nu_{t, x}$ onto $(\mathbb R_+\times \mathbb R)\setminus ([0, T]\times \bar{\mathcal{O}})$ with the value $\delta_{a_0}$ for an arbitrary $a_0\in A$. Define the coefficients $\bar b:\mathbb R_+\times \mathbb R\rightarrow \mathbb R$ and $\bar\sigma: \mathbb R_+\times \mathbb R\rightarrow \mathbb R_+$ as follows:
$$
\bar b(t, x)=
\begin{cases}
\int_A b(t, x, a) \nu_{t, x}(da),&\quad \text{if } (t, x)\in [0, T]\times \mathbb R\\
0, &\quad \text{otherwise.}
\end{cases}
$$
$$
\bar \sigma(t, x)=
\begin{cases}
\left(\int_A \sigma^2(t, x, a) \nu_{t, x}(da)\right)^{\frac{1}{2}},&\quad \text{if } (t, x)\in [0, T]\times \mathbb R\\
1, &\quad \text{otherwise.}
\end{cases}
$$
Note that these coefficients are bounded and measurable. Define the measures
$$\tilde\mu_\tau(B\times C):=\mu ((B\cap [0, T])\times (C\cap \bar{\mathcal{O}})), \quad  B\in\mathcal{B}(\mathbb R_+), \quad C\in\mathcal{B}(\mathbb R),$$
$$\tilde\mu_0(B\times C):=\int_{B\cap [0, T]} m_t((C\cap \bar{\mathcal{O}})\times A)dt,  \quad  B\in\mathcal{B}(\mathbb R_+), \quad C\in\mathcal{B}(\mathbb R).$$
$$\tilde m_0(B):=m_0^* (B\cap \mathcal{O}),  \quad  B\in\mathcal{B}(\mathbb R).$$
This implies that $(\tilde\mu_\tau, \tilde\mu_0, \tilde m_0)\in \mathcal{P}(\mathbb R_+\times \mathbb R)\times \mathcal{M}(\mathbb R_+\times \mathbb R)\times \mathcal{P}(\mathbb R)$. Define the operator 
$$\hat{\mathcal{L}}(\gamma\phi)(t, x)= \gamma'(t)\varphi(x) + \gamma (t)\left[ \bar b(t, x) \varphi'(x) + \frac{\bar{\sigma}^2}{2}(t, x)\varphi''(x) \right],$$
for all $\gamma\in C^1_b(\mathbb R_+)$, $\varphi\in C_b^2(\mathbb R)$. Then, by definition of $\tilde\mu_\tau$, $\tilde\mu_0$ and $\tilde m_0$,
$$\int_{\mathbb R_+\times \mathbb R} \gamma (t) \varphi(x)\tilde\mu_\tau(dt, dx) = \gamma(0)\int_{\mathbb R} \varphi(x)\tilde m_0(dx) + \int_{\mathbb R_+\times \mathbb R} \hat{\mathcal{L}}(\gamma\varphi)(t, x)\tilde\mu_0(dt, dx).$$
Let $U=\{0,1\}$ and define a new operator $\bar{\mathcal{L}}$ by
\begin{align*}
\bar{\mathcal{L}}(\beta \gamma \varphi)(r, s, x, u)& =u\beta(r)\hat{\mathcal{L}}(\gamma \varphi)(s, x)\\
&\quad + (1-u)\left[\beta(0)\gamma(0)\int_{\mathbb R} \varphi(x)\tilde m_0(dx) - \beta(r)\gamma(s)\varphi(x) + \beta'(r)\gamma(s)\varphi(x) \right],
\end{align*}
where $\beta\in C^1_b(\mathbb R_+)$, $\gamma\in C^1_b(\mathbb R_+)$, $\varphi\in C_b^2(\mathbb R)$ and $(r, s, x, u)\in \mathbb R_+\times \mathbb R_+\times \mathbb R \times U$. We set 
$$\mathcal{D}(\bar{\mathcal{L}})=\{\beta \gamma \varphi: \beta\in C^1_b(\mathbb R_+),\; \gamma\in C^1_b(\mathbb R_+), \;\varphi\in C_b^2(\mathbb R)\}.$$
Define $\bar{\mu}\in \mathcal{P}(\mathbb R_+\times \mathbb R_+\times \mathbb R \times U)$ by
\begin{align*}
\bar{\mu}(dr, ds, dx, du)&=K^{-1}\left[ \delta_1(du)\delta_0(dr) \tilde\mu_0 (ds, dx) +\delta_0(du) e^{-r}\mathds{1}_{\mathbb R^+}(r)dr \tilde\mu_\tau(ds, dx)\right],
\end{align*}
where $K=\tilde\mu_0(\mathbb R_+\times \mathbb R)+1$. The conditional distribution of $u$ given $(r, s, x)$ under $\bar \mu$ is
$$\bar \eta (r, s, x, du)=\delta_1(du)\mathds{1}_{\{0\}}(r) + \delta_0(du)e^{-r}\mathds{1}_{\mathbb R^+}(r).$$
As in Theorem 3.3 of \cite{stockbridge2002}, one can show that $\int \bar{\mathcal{L}}(\beta \gamma \varphi)d\bar\mu=0$ for all $\beta \gamma \varphi\in \mathcal{D}(\bar{\mathcal{L}})$. 

\textit{Second step: Verification of the conditions to apply Corollary 1.10 in \cite{stockbridge2017}.} Let $V=\mathbb R\times \mathbb R_+$. Define the operator
$$\mathcal{L}_0:\mathcal{D}(\mathcal{L}_0):=\mathcal{D}(\bar{\mathcal{L}})\subset C_b(\mathbb R_+ \times \mathbb R_+ \times \mathbb R)\rightarrow C(\mathbb R_+ \times \mathbb R_+\times {\mathbb R}\times U\times V)$$  
\begin{align*}
\mathcal{L}_0(\beta\gamma \varphi)(r, s, x, u, v) &=u\beta(r)\left( \gamma'(s)\varphi(x) + \gamma(s) \left[v^1 \varphi'(x) + \frac{v^2}{2}\varphi''(x)\right] \right) \\
&\quad  + (1-u)\left[\beta(0)\gamma(0)\int_{\mathbb R} \varphi(x)\tilde m_0(dx) - \beta(r)\gamma(s)\varphi(x) + \beta'(r)\gamma(s)\varphi(x) \right]
\end{align*}
where $v=(v^1, v^2)\in V$. We aim to apply Corollary 1.10 in \cite{stockbridge2017} without singular control. Define the transition function $\eta_0$ from $\mathbb R_+ \times \mathbb R_+\times \mathbb R\times U$ to $V$ as
$$\eta_0(r, s, x, u, dv)=\delta_{\bar b(s, x)}(dv^1) \delta_{\bar{\sigma}^2(s, x)}(dv^2).$$
We have,
$$\bar{\mathcal{L}}(\beta\gamma \varphi)(r, s, x, u)=\int_V \mathcal{L}_0(\beta\gamma \varphi)(r, s, x, u, v)\eta_0(r, s, x, u, dv).$$
Let $\psi(r, s, x, u, v)=1+u(|v^1|+v^2)$. We have that  
$$\int_{\mathbb R_+ \times \mathbb R_+\times {\mathbb R}\times U\times V}\psi(r, s, x, u, v)\eta_0(r, s, x, u, dv)\bar \mu(dr, ds, dx, du)\leq 1+\|b\|_\infty+\|\sigma^2\|_\infty<\infty.$$
Let us check that $\psi$ and $\mathcal{L}_0$ verify Condition 1.3 in \cite{stockbridge2017}. Taking $\beta=\gamma=\varphi=1$ we obtain $\mathcal{L}_0(1)=0$. On the other hand we can verify that there exists a constant $a(\beta, \gamma, \varphi)$ such that for all $(r, s, x, u, v)$,
$$|\mathcal{L}_0(\beta\gamma \varphi)(r,s,x,u,v)|\leq a(\beta, \gamma, \varphi)\psi(r, s, x, u, v).$$
One can find a countable subset of $C^1_b(\mathbb R_+)$ approximating any function of $C^1_b(\mathbb R_+)$ under the point-wise convergence of $\beta$ and $\beta'$ (the same holds for $C^2_b(\mathbb R)$ with the point-wise convergence of $\varphi$, $\varphi'$ and $\varphi''$). Then, the controlled martingale problem associated with $\mathcal{L}_0$ is countably generated. Let us prove that for each $(u, v)\in U\times V$, the operator $A_{u, v}(\beta \gamma \varphi)(r, s, x):=\mathcal{L}_0(\beta \gamma \varphi)(r, s, x, u, v)$ is a pre-generator. Suppose first that $u=1$, then 
$$A_{1, v}(\beta \gamma \varphi)(r, s, x)= \beta(r)\left( \gamma'(s)\varphi(x) + \gamma(s) \left[v^1 \varphi'(x) + \frac{v^2}{2}\varphi''(x)\right] \right).$$
For $z=(r, s, x)\in \mathbb R_+ \times \mathbb R_+\times \mathbb R$, define the processes $R_t^{z}=r$, $S_t^z=s+t$ and $X_t^{z}=x+v^1t + \sqrt{v^2}W_t$. For $t\geq 0$ and $z\in \mathbb R_+ \times \mathbb R_+\times \mathbb R$, define the measures 
$$\nu^{z}_t(B\times C\times D)=\delta_{R_t^{z}}(B)\delta_{S_t^{z}}(C)\mathbb P(X_{t}^{z}\in D), \quad B\in\mathcal{B}(\mathbb R_+),\quad C\in\mathcal{B}(\mathbb R_+), \quad D\in\mathcal{B}(\mathbb R).$$
Since $\nu^{z}=(\nu^{z}_t)_{t\geq 0}$ solves the forward equation for $(A_{1,v}, \delta_{z})$ and is right continuous at zero by the continuity in time of each process, we get by Proposition \ref{pre-generator} that $A_{1, v}$ is a pre-generator. 

Suppose now that $u=0$, then
$$A_{0, v}(\beta \gamma \varphi)(r, s, x)= \beta(0)\gamma(0)\int_{\bar{\mathcal{O}}} \varphi(x)\tilde m_0(dx) - \beta(r)\gamma(s)\varphi(x) + \beta'(r)\gamma(s)\varphi(x).$$
We can rewrite the operator as
$$A_{0, v}(h)(z)= \int_{\mathbb R_+ \times \mathbb R_+\times \mathbb R}(h(y)- h(z))\hat \mu(dy) + \partial_r h(z), \quad h\in \mathcal{D}(\mathcal{L}_0), \quad z\in \mathbb R_+ \times \mathbb R_+\times \mathbb R,$$
where
$$\hat \mu(dr, ds, dx)=\delta_0(dr)\delta_0(ds)\tilde{m}_0(dx).$$
{\color{black} By Proposition 10.2 p. 265 in \cite{ethier1986}, for any initial probability distribution $\nu$ on $\mathbb R_+ \times \mathbb R_+\times \mathbb R$, there exists a solution to the martingale problem for $(A_{0, v}, \nu)$ with càdlàg paths. This implies existence of a right continuous at zero solution to the forward equation for $(A_{0,v}, \delta_{z})$, for any $z$, which in turn entails by Proposition \ref{pre-generator} that $A_{0, v}$ is a pre-generator.}

Finally, the set $\mathcal{D}(\mathcal{L}_0)$ is closed under multiplication and separates points since we can use bump functions.

\textit{Third step: Controlled/stopped martingale problem representation.} By Corollary 1.10 in \cite{stockbridge2017}, there exist a complete probability space $(\Omega, \mathcal{F}, \mathbb Q)$ and a stationary $\mathbb R_+ \times \mathbb R_+\times \mathbb R$-valued process $(R, S, Y)$ (which we may assume is defined for all $t\in \mathbb R$) such that
$$\beta (R_t)\gamma(S_t) \varphi(Y_t)-\int_0^t\int_U\bar{\mathcal{L}}(\beta\gamma \varphi)(R_s, S_s, Y_s, u)\bar{\eta}(R_s, S_s, Y_s, du)ds$$
is an $(\bar{\mathcal{F}}_{t+}^{R, S, Y})_t$-martingale for all $\beta \gamma \varphi\in \mathcal{D}(\bar{\mathcal{L}})$, where $(\bar{\mathcal{F}}_{t+}^{R, S, Y})_t$ is the complete and right continuous augmentation of the natural filtration $(\mathcal{F}_{t}^{R, S, Y})_t$.

Following the same proof as Theorem 3.3 in \cite{stockbridge2002}, we arrive to the existence of a complete filtered probability space $(\Omega, \mathcal{F}, \mathbb F, \mathbb P)$, where $\mathbb F$ satisfies the usual conditions, an $\mathbb F$-stopping time $\tau$ with values in $\mathbb R_+$, a process $\tilde S$ with values in $\mathbb R_+$ such that $\tilde S_t \mathds{1}_{t\leq \tau}= t \mathds{1}_{t\leq \tau}$, an $\mathbb F$-progressively measurable process $X$ with values in $\mathbb R$ such that $\mathbb P \circ  X_0^{-1}=\tilde m_0$. Furthermore, 
$$\tilde\mu_\tau=\mathbb P \circ (\tau,  X_\tau)^{-1},$$
$$\tilde\mu_0(\Gamma)=\mathbb E^{\mathbb P}\left[ \int_0^\tau \mathds{1}_{\Gamma}(t, X_t)dt\right], \quad \forall \Gamma\in \mathcal{B}(\mathbb R_+\times \mathbb R),$$
and
\begin{equation}\label{mart_prob}
\gamma (\tilde S_{t\wedge \tau}) \varphi( X_{t\wedge \tau})-\int_0^{t\wedge \tau} \hat{\mathcal{L}}(\gamma \varphi)(\tilde S_s, X_s)ds,
\end{equation}
is an $\mathbb F$-martingale for all $\gamma \in C^1_b(\mathbb R_+)$, $\varphi\in C_b^2(\mathbb R)$.
Note that
$$1 = m_0^*(\mathcal{O}) = \tilde m_0(\mathcal{O})= \mathbb P (X_0\in\mathcal{O}),$$
which implies that $X_0\in \mathcal{O}$ $\mathbb P$-a.s. and
$$m_0^*=\mathbb P \circ X_0^{-1}.$$
On the other hand, since
$$1=\mu ([0, T]\times {\bar{\mathcal{O}}}) = \tilde\mu_\tau([0, T]\times {\bar{\mathcal{O}}}) = \mathbb P (\tau\in [0, T], X_\tau \in\bar{\mathcal{O}}),$$
we conclude that $\tau\leq T$, $X_\tau \in\bar{\mathcal{O}}$ $\mathbb P$-a.s. and
$$\mu =\mathbb P \circ (\tau, X_\tau)^{-1}.$$
Observe also that
$$\int_\Gamma m_t(dx, A)dt=\tilde\mu_0(\Gamma)=\mathbb E^{\mathbb P}\left[ \int_0^\tau \mathds{1}_{\Gamma}(t, X_t)dt\right], \quad \forall \Gamma\in \mathcal{B}([0, T]\times \bar{\mathcal{O}}).$$
Then, for $B\in \mathcal{B}([0, T]), C\in  \mathcal{B}(\bar{\mathcal{O}}), D\in \mathcal{B}(A)$,
$$\int_B\int_{C\times D}m_t(dx, da)dt = \int_{B\times C}\nu_{t, x}(D) m_t(dx, A)dt= \int_B \mathbb E^\mathbb{P} \left[ \mathds{1}_{C}(X_t)\nu_{t, X_t}(D) \mathds{1}_{t\leq \tau}\right]dt,$$
which implies
$$m_t(B\times C)= \mathbb E^{\mathbb P}\left[ \mathds{1}_B(X_t)\nu_{t, X_t}(C) \mathds{1}_{t\leq \tau}\right],  \quad B\in \mathcal{B}(\bar{\mathcal{O}}), \quad C\in \mathcal{B}(A), \quad t-a.e.$$
By the definition of $\tilde \mu_0$ we have
$$0=\tilde\mu_0(\mathbb R_+ \times \bar{\mathcal{O}}^c) =\mathbb E^\mathbb P \left[ \int_0^\tau  \mathds{1}_{\bar{\mathcal{O}}^c}(X_t)dt \right],$$
implying that
\begin{equation}\label{ess exit time}
(\mathbb P \otimes \lambda) (\{(\omega, t)\in \Omega\times [0, T]: X_t(\omega)\in \bar{\mathcal{O}}^c, t\leq \tau(\omega) \})=0.
\end{equation}
Using that $\tilde S_s \mathds{1}_{s\leq \tau}= s \mathds{1}_{s\leq \tau}$ and taking $\gamma=1$ in \eqref{mart_prob}, we get that for all $\varphi\in C_b^2(\mathbb R)$,
$$\varphi(X_{t\wedge\tau})- \int_0^{t\wedge\tau}\hat{\mathcal{L}}(\varphi)(s, X_s)ds $$
is an $\mathbb F$-martingale. Extending $b$ by $0$ and $\sigma$ by $1$ for $t>T$, we obtain that for all $\varphi \in C_b^2(\mathbb R)$, 
$$\varphi(X_{t \wedge \tau})-\varphi(X_0)-\int_{0}^{t \wedge \tau} \int_A \left(b(s, X_s, a)\varphi'(X_s)+ \frac{\sigma^2}{2}(s, X_s, a)\varphi''(X_s)\right) \nu_{s,X_s}(da)ds$$ 
is an $\mathbb{F}$-martingale. 

\textit{Fourth step: SDE representation of the controlled/stopped martingale problem.} Define $\tilde X_t:=X_{t \wedge \tau}$ for all $t\in\mathbb R_+$. Let us show that $\tilde X$ is a continuous process. Setting
$$\hat b_s(\omega):=\mathds{1}_{s\leq \tau(\omega)}\int_A b(s, X_s(\omega), a)\nu_{s, X_s(\omega)}(da)$$
$$\hat c_s(\omega):=\mathds{1}_{s\leq \tau(\omega)}\int_A \sigma^2(s, X_s(\omega), a)\nu_{s, X_s(\omega)}(da),$$
we get that for all $\varphi\in C_b^2(\mathbb R)$, 
$$\varphi(\tilde X_t)- \varphi(\tilde X_0) - \int_0^t \hat b_s \varphi'(\tilde X_{s-})ds - \int_0^t \frac{\hat c_s}{2} \varphi''(\tilde X_{s-})ds$$
is an $\mathbb F$-martingale. We conclude by Theorem II.2.42 from \cite{jacod2003} that $\tilde X$ is a semimartingale with characteristics $(B, C, 0)$ where
$$B_t=\int_0^t\hat b_s ds, \quad C_t=\int_0^t\hat c_s ds.$$
This means that the compensator of the random measure defined by 
$$\mu^{\tilde X}(\omega, dt, dx)=\sum_{s\geq 0}\mathds{1}_{\{\Delta \tilde X_s(\omega)\neq 0\}}\delta_{(s, \Delta \tilde X_s(\omega))}(dt, dx),$$
is equal to $0$ $\mathbb P$-a.s. Applying Theorem II.1.8 (i) from \cite{jacod2003} with $W=1$, we get that $\mu^{\tilde X}(\cdot, \mathbb R_+\times \mathbb R)=0$ a.s., which implies that $\tilde X$ is a continuous process. Using the continuity and \eqref{ess exit time}, we can deduce that $\tilde X$ takes values in $\bar{\mathcal{O}}$. 

Since $(\Omega, \mathcal{F},\mathbb F, \mathbb P, (\nu_{t, X_t})_{t\geq 0}, \tau, X)$ is a solution of the controlled/stopped martingale problem, by Theorem \ref{EM}, on an extension of the filtered probability space, there exists a continuous martingale measure $M$ with intensity $\nu_{t, X_t}(da)\mathds{1}_{t\leq \tau}dt$ such that
$$dX_t=\int_A b(t, X_t, a)\nu_t(da)dt + \int_A\sigma(t, X_t, a)M(dt, da), \quad t\leq \tau.$$
Moreover, there exists a Brownian motion $W$ such that $M(t, A)=M(t\wedge\tau, A)=W_{t\wedge\tau}$. In particular, if $\sigma$ is uncontrolled,
$$dX_t=\int_A b(t, X_t, a)\nu_t(da)dt + \sigma(t, X_t)dW_t, \quad t\leq \tau.$$

Let us prove now that $\tau\leq \tau_\mathcal{O}^X$ $\mathbb P$-a.s. If the first part of Assumption \ref{as exit time} holds, then $\tau_{\mathcal{O}}^{\tilde X}\geq T$ $\mathbb P$-a.s. Since $X_{t}=\tilde X_{t}$ on $\{t\leq \tau\}$, we get that $\tau \leq \tau_{\mathcal{O}}^X$ $\mathbb P$-a.s. If we suppose now that the second part of Assumption \ref{as exit time} holds, then since $\sigma$ is uncontrolled, 
$$X_{t \wedge \tau}=X_{0}+\int_{0}^{t \wedge \tau} \int_A  b(s,X_s, a) \nu_{s, X_s}(da)ds+\int_{0}^{t \wedge \tau}  \sigma(s,X_s) dW_s.$$
We define
$$Y_{t}=X_{0}+\int_{0}^{t} \int_A  b(s,X_s, a) \nu_{s, X_s}(da)\mathds{1}_{s\leq \tau}ds+\int_{0}^{t} \sigma(s,X_s) dW_s.$$
By Lemma \ref{lemma exit closure} we get that $\tau_{\mathcal{O}}^Y=\tau_{\bar{\mathcal{O}}}^Y$ $\mathbb P$-a.s. Using that for all $t\geq 0$, $X_{t \wedge \tau}=Y_{t \wedge \tau}$ and $X_{t \wedge \tau}$ is $\bar{\mathcal{O}}$-valued, we get that $\tau \leq \tau_{\mathcal{O}}^X$ $\mathbb P$-a.s. 

The case where $\nu_{t, x}=\delta_{\alpha(t, x)}$ for some measurable function $\alpha$, follows by the same arguments and replacing Theorem \ref{EM} by Theorem \ref{EK}.
\end{proof}

{\color{black}
\section{Sufficient condition for the existence of a square integrable density for $m_t(dx,A)$}\label{sec D}

\begin{proposition}
Suppose that Assumption \ref{as PDE} (1-6) holds true. Moreover, assume that $\sigma^2$ is Lipschitz continuous on $[0, T]\times \bar{\mathcal{O}}$ and $m_0^*$ has a bounded density with respect to the Lebesgue measure. If $(\mu, m)\in \mathcal{R}$, then $m_t(dx, A)dt$ admits an square integrable density with respect to the Lebesgue measure on $[0, T]\times \bar{\mathcal{O}}$.
\end{proposition}

\begin{proof}
We set $\eta(dt, dx)=m_t(dx, A)dt$. By  Theorem \ref{SDE rep}, there exist a filtered probability space $(\Omega, \mathcal{F}, \mathbb F, \mathbb P)$, an $\mathbb F$-adapted process $X$, an $\mathbb F$-stopping time $\tau$ such that $\tau\leq T\wedge\tau_\mathcal{O}^X$ $\mathbb P$-a.s., and an $\mathbb F$-Brownian motion $W$, such that
$$X_{t\wedge \tau}= \int_0^{t\wedge \tau}\int_A b(t, X_t, a)\nu_{t, X_t}(da)dt + \int_0^{t\wedge \tau}\sigma(t, X_t)dW_t, \quad \mathbb P \circ X_0^{-1}= m_0^*,$$
$$m_t(B\times C)= \mathbb E^{\mathbb P}\left[ \mathds{1}_B(X_t)\nu_{t, X_t}(C) \mathds{1}_{t\leq \tau}\right],  \quad B\in \mathcal{B}(\bar{\mathcal{O}}), \quad C\in \mathcal{B}(A), \quad t-a.e.$$
We can rewrite $\eta$,
$$\eta(B\times C)=\int_0^T\mathbb E^\mathbb P\left[ \mathds{1}_{B}(t) \mathds{1}_{C}(X_t)\mathds{1}_{t\leq \tau} \right]dt, \quad B\in \mathcal{B}([0, T]),\; C\in \mathcal{B}(\bar{\mathcal{O}}).$$
Since $\tau\leq \tau_\mathcal{O}^X$, we get that 
$$(\lambda\times \mathbb P)(\{(t, \omega): X_t(\omega)\in \partial\mathcal{O}, t\leq \tau(\omega)\})=0,$$
which means that $\eta$ puts $0$ mass on $\partial \mathcal{O}$ and can thus be treated as a measure on $[0, T]\times \mathcal{O}$. By standard arguments of existence of strong solutions to SDEs, there exists a unique process $Y$ such that
$$
Y_t=
\begin{cases}
X_t &\text{if } t\leq \tau\\
X_\tau + \int_\tau^t \sigma(s, Y_s)dW_s &\text{if } t> (\tau, T].
\end{cases}
$$
Let
$$\bar b_s:=\int_A b(s, X_s, a)\nu_{s, X_s}(da)\mathds{1}_{s\leq \tau},$$
and rewrite $Y$ as 
$$Y_{t}=Y_0 + \int_0^t\bar b_s ds +\int_0^t\sigma(s, Y_s)dW_s, \quad s\in [0, T].$$
Define
$$\lambda_s=\frac{\bar b_s}{\sigma(s, X_s)},$$
$$Z_t=\exp\left[ -\int_0^t\lambda_s dW_s -\frac{1}{2}\int_0^t \lambda_s^2 ds\right].$$
Since $\lambda$ is a bounded process, by Girsanov's Theorem, under $\mathbb Q$, 
$$\widetilde{W}_t=W_t+\int_0^t\lambda_s ds$$
is an $(\mathbb F, \mathbb Q)$-Brownian motion, where
$$\frac{d\mathbb Q}{d\mathbb P}=Z_T.$$
The dynamics of $Y$ under $\mathbb Q$, are as follows
$$Y_{t}=Y_0 + \int_0^t\sigma(s, Y_s)d\widetilde{W}_s, \quad s\in [0, T].$$
By Remark 5.1 in \cite{bdt2020},
$$\tilde\eta(B\times C)=\mathbb E^\mathbb Q\left[ \int_0^T \mathds{1}_{B}(t) \mathds{1}_{C}(Y_t)\mathds{1}_{t\leq \tau_\mathcal{O}^Y}dt \right], \quad B\in \mathcal{B}([0, T]),\; C\in \mathcal{B}(\mathcal{O}),$$
admits a bounded density $(t, x)\mapsto \tilde\eta(t, x)$ with respect to the Lebesgue measure on $[0, T]\times \mathcal{O}$, i.e. $\tilde\eta(dt, dx)=\tilde\eta(t, x)dtdx$. Letting $B\in \mathcal{B}([0, T])$, $C\in \mathcal{B}(\mathcal{O})$, we have
\begin{align*}
\eta(B\times C) &= \int_0^T\mathbb E^\mathbb P\left[ \mathds{1}_{B}(t) \mathds{1}_{C}(X_t)\mathds{1}_{t\leq \tau} \right]dt\\
&= \int_0^T\mathbb E^\mathbb P\left[ \mathds{1}_{B}(t) \mathds{1}_{C}(Y_t)\mathds{1}_{t\leq \tau} \right]dt\\
&\leq \int_0^T\mathbb E^\mathbb P\left[ \mathds{1}_{B}(t) \mathds{1}_{C}(Y_t)\mathds{1}_{t\leq \tau_{\mathcal{O}}^Y} \right]dt\\
&= \int_0^T\mathbb E^\mathbb Q\left[ \mathds{1}_{B}(t) \mathds{1}_{C}(Y_t)\mathds{1}_{t\leq \tau_{\mathcal{O}}^Y} Z_T^{-1}\right]dt\\
&\leq \left(\mathbb E^\mathbb Q[Z_T^{-2}]\right)^{1/2} \left(\int_0^T\mathbb E^\mathbb Q\left[ \mathds{1}_{B}(t) \mathds{1}_{C}(Y_t)\mathds{1}_{t\leq \tau_{\mathcal{O}}^Y}\right]dt\right)^{1/2}.\\
&= C_1 \left(\tilde \eta(B\times C)\right)^{1/2},\\
\end{align*}
where $C_1=\left(\mathbb E^\mathbb Q[Z_T^{-2}]\right)^{1/2}$. This allows to deduce that $\eta(dt, dx)=\eta(t, x)dtdx$ for some non-negative $L^1$ function $\eta$. Moreover
$$\int_0^T\int_{\mathcal{O}}\eta^2(t, x)dtdx \leq C_1 \left(\int_0^T\int_{\mathcal{O}}\eta(t, x)\tilde\eta(t, x)dtdx\right)^{1/2}\leq C_1 \left( \|\eta\|_1\|\tilde\eta\|_\infty \right)^{1/2}<\infty,$$
which shows that the density $\eta$ is on $L^2$.
\end{proof}
}

\section{Proof of Theorem \ref{PDE}}\label{sec E}

\begin{proof}
(1)  We define $\tau^{x, \alpha}_\mathcal{O}:=\inf \left\{t \geq 0 : \, X_{t}^{x, \alpha} \notin \mathcal{O}\right\}$. By Theorem \ref{theoremstrong}, for all $x\in \mathcal{O}$,
\begin{equation}\label{value_init}
v(0, x)=\mathbb{E}\left[\int_{0}^{\tau^\star(x)} f\left(s, X_{s}^{x, \alpha^\star(x)}, \alpha_s^\star(x)\right) ds + g\left(\tau^\star(x), X_{\tau^\star(x)}^{x, \alpha^\star(x)}\right) \right].
\end{equation}
Note that, by definition of the value function $v$, we have $\tau^\star(x) \leq \tau_{\mathcal{O}}^{x, \alpha^\star} \wedge T$ a.s. Now consider the measures defined by
\begin{equation*}
\bar m_t(B)=\int_\mathcal{O}\mathbb P \left[\left(X_t^{x, \alpha^\star(x)},\alpha_t^\star(x)\right)\in B, t\leq \tau^\star(x)\right]m_0^*(dx), \quad B\in\mathcal{B}(\bar{\mathcal{O}}\times A),
\end{equation*}
$$\bar\mu(B)=\int_\mathcal{O} \mathbb P\left[\left(\tau^\star(x), X_{\tau^\star(x)}^{x, \alpha^\star(x)}\right)\in B\right]m_0^*(dx), \quad B\in\mathcal{B}([0, T]\times \bar{\mathcal{O}}).$$
Integrating with respect to $m^*_0$ in \eqref{value_init}, we derive that
$$V^S = \Gamma(\bar\mu, \bar m)$$
Since $(\bar \mu, \bar m)\in \mathcal{R}$, we conclude that $V^{S}\leq V^{LP}$.

We now prove that the converse inequality holds. Fix $(\mu, m) \in \mathcal{R}$. Since $v\in W^{1, 2, 2}((0, T)\times \mathcal{O})$, there exists a sequence $(u_n)_{n\geq 1}\subset C^{1, 2}_b([0, T]\times \bar{\mathcal{O}})$ such that $u_n\rightarrow v$ in $W^{1, 2, 2}((0, T)\times \mathcal{O})\cap C([0, T]\times \bar{\mathcal{O}})$. By condition (7) in Assumption \ref{as PDE} and Theorem \ref{SDE rep}, we get that $m_t(dx, A)dt$ has a square integrable density with respect to the Lebesgue measure. In particular we can change the set $\bar{\mathcal{O}}$ by $\mathcal{O}$ in the integrals with respect to $m$. Therefore, we get
\begin{equation}\label{v_fp}
\int_{[0, T]\times \bar{\mathcal{O}}}v(t, x)\mu(dt, dx)= \int_{\mathcal{O}}v(0, x)m_0^*(dx) + \int_0^T\int_{\mathcal{O}\times A} \left(\frac{\partial v}{\partial t}+\mathcal{L} v\right)(t, x, a) m_t(dx, da) dt.
\end{equation}
From the above equality, we derive that
\begin{equation}\label{v_fp2}
V^{S} = \int_{[0, T]\times \bar{\mathcal{O}}}v(t, x)\mu(dt, dx) -\int_{0}^{T} \int_{\mathcal{O}\times A}\left(\frac{\partial v}{\partial t}+\mathcal{L} v\right)(t, x, a) m_{t}(dx, da) dt.
\end{equation}
Now using the HJBVI (\ref{HJBVI}), we get $V^{S}\geq V^{LP}$.\\
(2) Let $(\mu^\star, m^\star)$ be a maximizer of the LP program. As before, $m_t(dx, A)dt$ admits a square integrable density with respect to the Lebesgue measure.\\
(a) By (1) we get $V^{LP}=V^{S}$, that is
\begin{equation*}
\int_{0}^{T} \int_{\mathcal{O}\times A} f(t, x, a) m^\star_t(dx, da) dt + \int_{[0, T]\times \bar{\mathcal{O}}}g(t, x)\mu^\star(dt, dx) = \int_{\mathcal{O}}v(0, x)m_0^*(dx).
\end{equation*}
Since $g\in C^{1, 2}_b([0, T]\times \bar{\mathcal{O}})$,
\begin{equation}\label{g_fp}
\int_{[0, T]\times \bar{\mathcal{O}}}g(t, x)\mu^\star(dt, dx)= \int_{\mathcal{O}}g(0, x)m_0^*(dx) + \int_0^T\int_{\mathcal{O}\times A} \left(\frac{\partial g}{\partial t}+\mathcal{L} g\right)(t, x, a) m_t^\star(dx, da) dt.
\end{equation}
Therefore, using the last two equalities
\begin{equation}\label{mixed_sol}
\begin{split}
&\int_{\mathcal{S}\times A} \left(f+\frac{\partial g}{\partial t}+\mathcal{L} g\right)(t, x, a) m_t^\star(dx, da) dt = \int_{\mathcal{O}}(v-g)(0, x)m_0^*(dx) \\
&\quad - \int_{\mathcal{C}\times A} \left(f+\frac{\partial g}{\partial t}+\mathcal{L} g\right)(t, x, a) m^\star_t(dx, da) dt\\
& \quad \geq \int_{\mathcal{O}}(v-g)(0, x)m_0^*(dx) + \int_{\mathcal{C}\times A} \left(\frac{\partial (v-g)}{\partial t}+\mathcal{L} (v-g)\right)(t, x, a) m_t^\star(dx, da) dt\\
& \quad = \int_{\mathcal{O}}(v-g)(0, x)m_0^*(dx) + \int_0^T\int_{\mathcal{O}\times A} \left(\frac{\partial (v-g)}{\partial t}+\mathcal{L} (v-g)\right)(t, x, a) m_t^\star(dx, da) dt.
\end{split}
\end{equation}
The inequality follows from the HJBVI (\ref{HJBVI}) and the last equality follows from the fact that for all $a\in A$, 
$$\left(\frac{\partial (v-g)}{\partial t}+\mathcal{L} (v-g)\right)(t, x, a)=0, \quad \text{a.e. on } \mathcal{S}.$$
By \eqref{v_fp} and \eqref{g_fp}, we obtain
$$\int_{\mathcal{S}\times A} \left(f+\frac{\partial g}{\partial t}+\mathcal{L} g\right)(t, x, a) m_t^\star(dx, da) dt = \int_{[0, T]\times \bar{\mathcal{O}}}(v-g)(t, x)\mu^\star (dt, dx) \geq 0.$$
Finally, since for all $a\in A$, 
$$\left(f+\frac{\partial g}{\partial t}+\mathcal{L} g\right)(t, x, a)\leq 0\quad \text{a.e. on } \mathcal{S},$$
we conclude that
$$\int_{\mathcal{S}\times A} \left(f+\frac{\partial g}{\partial t}+\mathcal{L} g\right)(t, x, a) m_t^\star(dx, da) dt = 0.$$
(b) The inequality in (\ref{mixed_sol}) is now an equality, so we have 
$$-\int_{\mathcal{C}\times A} f(t, x, a)m^\star_t(dx, da)dt= \int_{\mathcal{C}\times A} \left(\frac{\partial v}{\partial t}+\mathcal{L} v\right)(t, x, a)m^\star_t(dx, da)dt,$$
\begin{equation*}
\int_{\mathcal{O}}(v-g)(0, x)m_0^*(dx) + \int_0^T\int_{\mathcal{O}\times A} \left(\frac{\partial (v-g)}{\partial t}+\mathcal{L} (v-g)\right)(t, x, a) m_t^\star(dx, da) dt = 0.
\end{equation*}
Since
\begin{equation*}
\int_{[0, T]\times \bar{\mathcal{O}}}g(t, x)\mu^\star(dt, dx)= \int_{\mathcal{S}}v(t, x)\mu^\star(dt, dx) + \int_{\mathcal{C}}g(t, x)\mu^\star(dt, dx) + \int_{[0, T]\times \partial\mathcal{O}}v(t, x)\mu^\star(dt, dx),
\end{equation*}
we get 
$$\int_{\mathcal{C}}(v-g)(t, x)\mu^\star(dt, dx)=0.$$
We conclude that $\mu^\star(\mathcal{C})=0$.\\
(c) The result follows since $\mu^\star(\mathcal{C})=0$.\\
\end{proof}

\section{Two technical lemmas}

\begin{lemma}\label{slutsky}
Let $\mathcal{X}$ and $\mathcal{Y}$ complete, separable metric spaces, and let $\varphi:\mathcal{X}\times \mathcal{Y}\rightarrow\mathbb R$ be bounded and continuous. Then, the map 
$$\mathcal{Y}\times \mathcal{M}(\mathcal{X})\ni (y, \nu)\mapsto \int_{\mathcal{X}}\varphi(x, y)\nu(dx)$$
is continuous.
\end{lemma}
\begin{proof}
Let $y^n\rightarrow \bar y$ and $\nu^n\rightharpoonup \nu$, let us prove that
$$\int_{\mathcal{X}}\varphi(x, y^n)\nu^n(dx)\underset{n\rightarrow\infty}{\longrightarrow} \int_{\mathcal{X}}\varphi(x, \bar{y})\nu(dx).$$
It suffices to show that $\nu^n\times \delta_{y^n}\rightharpoonup \nu\times \delta_{\bar{y}}$. By Remark 8.3.1 and Exercise 8.10.71 in \cite{bogachev2007} (Volume 2), it is sufficient to use bounded and Lipschitz functions as test functions. Consider a bounded and Lipschitz function $\phi:\mathcal{X}\times \mathcal{Y}\rightarrow \mathbb R$ and denote by $L$ the Lipschitz constant of $\phi$. We have
\begin{align*}
\left|\int_{\mathcal{X}\times \mathcal{Y}}\phi(x,y)\delta_{y^n}(dy)\nu^n(dx) - \int_{\mathcal{X}\times \mathcal{Y}}\phi(x, y)\delta_{\bar{y}}(dy)\nu(dx)\right|&=
\left| \int_{\mathcal{X}}\phi(x, y^n)\nu^n(dx) - \int_{\mathcal{X}}\phi(x, \bar{y})\nu(dx) \right|\\ &\leq \left| \int_{\mathcal{X}}\phi(x, y^n)\nu^n(dx) - \int_{\mathcal{X}}\phi(x, \bar{y})\nu^n(dx) \right|\\
&\quad + \left| \int_{\mathcal{X}}\phi(x, \bar{y})\nu^n(dx) - \int_{\mathcal{X}}\phi(x, \bar{y})\nu(dx) \right|
\end{align*}
The second term converges to $0$ since $\nu^n\rightharpoonup \nu$. For the first term we get
$$\left| \int_{\mathcal{X}}\phi(x, y^n)\nu^n(dx) - \int_{\mathcal{X}}\phi(x, \bar{y})\nu^n(dx) \right|\leq L d_\mathcal{Y}(y^n, \bar{y})\sup_{n\geq 1}\nu^n(\mathcal{X}),$$
which also converges to $0$ since $\nu^n(\mathcal{X})$ converges to $\nu(\mathcal{X})$, which gives the uniform boundedness of $(\nu^n(\mathcal{X}))_{n\geq 1}$.
\end{proof}

\begin{lemma}\label{slutsky_stable}
Let $\Theta$, $\mathcal{X}$ complete, separable metric spaces. Let $\eta\in \mathcal{M}(\Theta)$. Let $\varphi:\Theta \times \mathcal{X}\times \mathbb R^d\rightarrow\mathbb R$, with $d\in\mathbb N^*$, be a bounded measurable map and assume that for every $t\in \Theta$, $\varphi(t, \cdot)$ is continuous. Suppose that a sequence of measurable functions $\psi^n:\Theta\rightarrow \mathbb R^d$ converges in $L^1(\Theta, \eta)$ to a measurable function $\psi:\Theta\rightarrow \mathbb R^d$ and that $(\nu^n_t(dx)\eta(dt))_{n\geq 1}\subset \mathcal{M}(\Theta\times \mathcal{X})$ converges to $\nu_t(dx)\eta(dt)$ in the stable topology, where $(\nu^n)_{n\geq 1}$ and $\nu$ are transition kernels from $\Theta$ to $\mathcal{X}$. Suppose also that there exists a constant $C>0$ such that $\eta$-a.e. $\sup_{n\geq 1}\nu_t^n(\mathcal{X})\leq C$. Then,
$$\int_\Theta \int_{\mathcal{X}}\varphi(t, x, \psi^n(t))\nu^n_t(dx)\eta(dt)\underset{n\rightarrow\infty}{\longrightarrow} \int_\Theta \int_{\mathcal{X}}\varphi(t, x, \psi(t))\nu_t(dx)\eta(dt).$$
\end{lemma}
\begin{proof}
We need to prove that
$$\int_{\Theta\times\mathcal{X}\times \mathbb R^d}\varphi(t, x, y)\delta_{\psi^n(t)}(dy)\nu^n_t(dx)\eta(dt)\underset{n\rightarrow\infty}{\longrightarrow} \int_{\Theta\times\mathcal{X}\times \mathbb R^d}\varphi(t, x, y)\delta_{\psi(t)}(dy)\nu_t(dx)\eta(dt).$$
It suffices to show that $\delta_{\psi^n(t)}(dy)\nu^n_t(dx)\eta(dt)$ converges to $\delta_{\psi(t)}(dy)\nu_t(dx)\eta(dt)$ in the stable topology. We are going to use Corollary 2.9 in \cite{jacod1981}. Since $\delta_{\psi^n(t)}(dy)$ has mass 1, the first condition of the Corollary follows by stable convergence of $\nu^n_t(dx)\eta(dt)$. Now, we need to show that $\delta_{\psi^n(t)}(dy)\nu^n_t(dx)\eta(dt)\rightharpoonup \delta_{\psi(t)}(dy)\nu_t(dx)\eta(dt)$. As in the previous Lemma, it is sufficient to use bounded and Lipschitz functions as test functions. Consider a bounded and Lipschitz function $\phi:\Theta\times \mathcal{X}\times \mathbb R^d\rightarrow \mathbb R$ and denote by $L$ the Lipschitz constant of $\phi$. We have
\begin{align*}
& \left| \int_{\Theta\times \mathcal{X}}\phi(t, x, \psi^n(t))\nu^n_t(dx)\eta(dt) - \int_{\Theta\times \mathcal{X}}\phi(t, x, \psi(t))\nu_t(dx)\eta(dt) \right| \\
&\leq \left| \int_{\Theta\times \mathcal{X}}\phi(t, x, \psi^n(t))\nu^n_t(dx)\eta(dt) - \int_{\Theta\times \mathcal{X}}\phi(t, x, \psi(t))\nu^n_t(dx)\eta(dt) \right|\\
&\quad + \left| \int_{\Theta\times \mathcal{X}}\phi(t, x, \psi(t))\nu^n_t(dx)\eta(dt) - \int_{\Theta\times \mathcal{X}}\phi(t, x, \psi(t))\nu_t(dx)\eta(dt) \right|\\
\end{align*}
The second term converges to $0$ since $\nu^n_t(dx)\eta(dt)$ converges to $\nu_t(dx)\eta(dt)$ in the stable topology. For the first term we get
$$\left| \int_{\Theta\times \mathcal{X}}\phi(t, x, \psi^n(t))\nu^n_t(dx)\eta(dt) - \int_{\Theta\times \mathcal{X}}\phi(t, x, \psi(t))\nu^n_t(dx)\eta(dt) \right|\leq CL \int_\Theta \|\psi^n(t)-\psi(t)\|\eta(dt),$$
which also converges to $0$.
\end{proof}

\section{Some results on set-valued analysis}\label{sec G}

Let us recall some theory about set-valued analysis, which can be found in Chapter 17 of \cite{aliprantis2007}. For the next definitions, consider a metric space $(X, d)$ and a set valued map $\varphi:X\rightarrow 2^X$. The graph of $\varphi$ is defined as the following set:
$$\operatorname{Gr}(\varphi):=\{(x, y)\in X^2: y\in \varphi(x)\}.$$

\begin{definition}\label{def up}
The correspondence $\varphi$ is said to be \textit{upper hemicontinuous} if for any sequence $(x_n, y_n)_{n\geq 1}$ in the graph of $\varphi$ such that $x_n\rightarrow x$, the sequence $(y_n)_{n\geq 1}$ has a limit point in $\varphi(x)$.
\end{definition}

\begin{theorem}[Closed Graph Theorem, Theorem 17.11 in \cite{aliprantis2007}]\label{upperhemi_closedgraph}
If $X$ is compact, the following statements are equivalent:
\begin{enumerate}[(i)]
\item $\varphi(x)$ is closed for all $x\in X$ and $\varphi$ is upper hemicontinuous.
\item The graph of $\varphi$ is closed.
\end{enumerate}
\end{theorem}

\begin{definition}\label{def low}
The correspondence $\varphi$ is said to be \textit{lower hemicontinuous} if whenever $x_n\rightarrow x$ and $y\in \varphi(x)$, there exists a subsequence $(x_{n_k})_{k\geq 1}$ of $(x_n)_{n\geq 1}$ and a sequence $(y_k)_{k\geq 1}$, such that $y_k\in\varphi(x_{n_k})$ and $y_k\rightarrow y$.
\end{definition}

\begin{definition}\label{Defcont}
We say that $\varphi$ is \textit{continuous} if it is both upper hemicontinuous and lower hemicontinuous.
\end{definition}

\begin{theorem}[Berge's Maximum Theorem, Theorem 17.31 in \cite{aliprantis2007}]\label{berge}
Let $(X, d)$ be a metric space. Consider $\mathcal{R}^\star:X\rightarrow 2^X$ a continuous correspondence with nonempty compact values and $F: \operatorname{Gr}(\mathcal{R}^\star)\rightarrow \mathbb R$ a continuous function. Define the function $\Theta:X\rightarrow 2^X$ by
$$\Theta(x)= \underset{y\in\mathcal{R}^\star(x)}{\arg \max } \, F(x, y).$$
Then $\Theta$ is upper hemicontinuous and has nonempty compact values.
\end{theorem}

\begin{theorem}[Kakutani-Fan-Glicksberg, Corollary 17.55 in \cite{aliprantis2007}]\label{KFG} Let $K$ be a nonempty compact convex subset of a locally convex Hausdorff space, and let the correspondence $\Theta:K\rightarrow 2^K$ have closed graph and nonempty convex values. Then the set of fixed points of $\Theta$ is compact and nonempty.
\end{theorem}

\end{document}